\documentclass[11pt]{article}
\usepackage{amssymb, amsmath, amsthm, oldgerm,srcltx}
\usepackage{graphicx}
\usepackage{xcolor}
\newcommand{\tg}{}
\newtheorem{thm}{Theorem}[section]
\newtheorem{lem}{Lemma}[section]
\newtheorem{cor}[lem]{Corollary}
\newtheorem{prop}[thm]{Proposition}
\newtheorem{rem}[thm]{Remark}
\numberwithin{equation}{section}

\newcommand{\abs}[1]{\left\vert#1\right\vert}

\newcommand{\E}{\mathbf{E}\,}

\newcommand{\R}{\mathbf{R}}

\newcommand{\re}{\mathrm{Re}\;\!}
\newcommand{\im}{\mathrm{Im}\;\!}
\newcommand{\Tr}{\mathrm{Tr}\;\!}

\newcommand{\lln}{\operatorname{llog}_n}

\newenvironment{Proof of}{\removelastskip\par\medskip
\noindent{\em Proof of} \rm}{\penalty-20\null\hfill$\square$\par\medbreak}

\def\be{\begin{equation}}
\def\en{\end{equation}}
\def\bee{\begin{eqnarray*}}
\def\ene{\end{eqnarray*}}

\def\E{{\bf E}}

\def\R{{\mathbb R}}

\def\Tr{{\rm Tr}\,}
\def\Im{{\rm Im}\,}

\def\<{\left<}
\def\>{\right>}
\def\1{{\bf 1}}
\def\4{\kern1pt}

\begin{document}
\bibliographystyle{}

\vspace{1in}
 \date{}
\title{\bf Optimal Bounds for  Convergence of  Expected Spectral Distributions to the Semi-Circular Law}

\author{{\bf F. G\"otze}\\{\small Faculty of Mathematics}
\\{\small University of Bielefeld}\\{\small Germany}
\and {\bf A. Tikhomirov}$^{1}$\\{\small Department of Mathematics
}\\{\small Komi Science Center of Ural Division of RAS,}\\{\small Syktyvkar State University}
\\{\small  Russia}
}
\maketitle
 \footnote{$^1$Research supported   by SFB 701 ``Spectral Structures and Topological Methods in Mathematics'' University of Bielefeld.
Tikhomirov's research supported   by grants  RFBR N~14-01-00500 and by Program of Fundamental Research Ural Division of RAS, Project N~12-P-1-1013}

\maketitle

\date{}

\begin{abstract}
Let $\mathbf X=(X_{jk})_{j,k=1}^n$ denote a  Hermitian random matrix with
entries $X_{jk}$, which are independent for $1\le j\le  k\le n$. We consider
the rate of convergence of the empirical  spectral distribution function of
the matrix $\mathbf X$ to the semi-circular law assuming that $\E X_{jk}=0$,
$\E X_{jk}^2=1$ and that 
$$
\sup_{n\ge1}\sup_{1\le j,k\le n}\E|X_{jk}|^4=:\mu_4<\infty,
$$
and 
$$
\sup_{1\le j,k\le n}|X_{jk}|\le D_0n^{\frac14}.
$$
By means of a recursion argument it is shown that the
 Kolmogorov distance between the expected spectral
distribution of the Wigner matrix $\mathbf W=\frac1{\sqrt n}\mathbf X$
and the  semicircular law is of order $O(n^{-1})$.
\end{abstract}
\maketitle


\section{Introduction}
\setcounter{equation}{0}

Consider a family $\mathbf X = \{X_{jk}\}$, $1 \leq j \leq k \leq n$,
of independent real random variables defined on some probability space
$(\Omega,{\textfrak M},\Pr)$, for any $n\ge 1$. Assume that $X_{jk} = X_{kj}$, for
$1 \leq k < j \leq n$, and introduce the symmetric matrices
\begin{displaymath}
 \mathbf W = \ \frac{1}{\sqrt{n}} \left(\begin{array}{cccc}
 X_{11} &  X_{12} & \cdots &  X_{1n} \\
 X_{21} & X_{22} & \cdots &  X_{2n} \\
\vdots & \vdots & \ddots & \vdots \\
 X_{n1} &  X_{n2} & \cdots &  X_{nn} \\
\end{array}
\right).
\end{displaymath}

The matrix $\mathbf W$ has a random spectrum $\{\lambda_1,\dots,\lambda_n\}$ and an
associated spectral distribution function
$\mathcal F_{n}(x) = \frac{1}{n}\ {\rm card}\,\{j \leq n: \lambda_j \leq
x\}, \quad x \in \R$.
Averaging over the random values $X_{ij}(\omega)$, define the expected
(non-random) empirical distribution functions
$ F_{n}(x) = \E\,\mathcal F_{n}(x)$.
Let $G(x)$ denote the semi-circular distribution function with density
$g(x)=G'(x)=\frac1{2\pi}\sqrt{4-x^2}\mathbb I_{[-2,2]}(x)$, where $\mathbb I_{[a,b]}(x)$
denotes the indicator--function of the interval $[a,b]$. 
 The rate of convergence to the semi-circular law has been studied by several authors.
 We proved  in \cite{GT:03}  that the  Kolmogorov distance between $\mathcal F_n(x)$ and
 the distribution function
 $G(x)$,
 $\Delta_n^*:=\sup_x|\mathcal F_n(x)-G(x)|$ is of order
 $O_P(n^{-\frac12})$ (i.e. $n^{\frac12}\Delta_n^*$ is bounded in probability). Bai et al.\cite{Bai:02}, \cite{Bai:93}  and Girko  \cite{Girko:02}
 showed that  $\Delta_n:=\sup_x| F_n(x)-G(x)|=O(n^{-\frac12})$.
 Bobkov, G\"otze and Tikhomirov \cite{BGT:08}  proved that $\Delta_n$ and $\E\Delta_n^*$
 have order $O(n^{-\frac23})$
assuming a Poincar\'e inequality for the
distribution of the  matrix elements. For the Gaussian Unitary Ensemble respectively for the Gaussian Orthogonal Ensemble, see  \cite{GT:2005}
respectively  \cite{TTKh:2008},
it has been shown that $\Delta_n=O(n^{-1})$.
Denote by $\gamma_{n1}\le\ldots\le \gamma_{nn}$, the quantiles of $G$,
i.e.   $G(\gamma_{nj})=\frac jn$,
and introduce the notation
 $\lln :=\log\log n$.
Erd\"os et al. \cite{ErdosYauYin:2010a}, \cite{ErdosYauYin:2010} showed,  for  matrices with
elements $X_{jk}$  which have a uniformly
sub exponential decay, i.e. 
\begin{equation}\label{exptails}
 \Pr\{|X_{jk}|>t\}\le A\exp\{-t^{\varkappa}\},
\end{equation}
 for some $\varkappa>0$, $A>0$ and for any $t\ge1$,
 the following result
\begin{align} \label{yau}
\Pr\Bigl\{\, \exists\,\, j:\,|\lambda_j-\gamma_{nj}|\ge (\log n)^{C \lln }
\Big[\min(j,N-j+1)\Big]^{-\frac13}n^{-\frac23} \Bigr\} \quad\quad\quad\quad\notag\\ \le C\exp\{-(\log n)^{c \4 \lln}\},
\end{align}
for $n$ large enough.
It is straightforward to check that this bound implies that
\begin{equation}\label{rate}
\Pr\Bigl\{\sup_x|\mathcal F_n(x)-G(x)|\le Cn^{-1}(\log n)^{C \lln}\Bigr\}\ge 1- C\exp\{-(\log n)^{c \lln}\}.
\end{equation}
From the last inequality it is follows that $\E \Delta_n^*\le C\4 n^{-1}(\log n)^{C \4 \lln}$.
Similar results were obtained in \cite[Theorem 32]{TaoVu:2011}, assuming additionally that the distributions of the entries of matrices have vanishing third moment.

In this paper we derive   the optimal bound  for the rate of convergence of the expected spectral distribution to the semi--circular law. 
 Using arguments
 similar to those used in \cite{SchleinMaltseva:2013}  we provide  a self-contained proof
 based on recursion  methods developed
in the papers of G\"otze and Tikhomirov \cite{GT:03}, \cite{GT:09} and \cite{T:09}.
It follows from the results of Gustavsson \cite{Gustavsson:2005} that the best possible bound in the Gaussian case for
the rate of convergence in probability is $O(n^{-1}\sqrt{\log n})$. The best possible bound for  $\Delta_n$ is of order $O(n^{-1})$. For  Gaussian matrices such bounds were
obtained  in   \cite{GT:2005} and \cite{TTKh:2008}.
Our setup includes the case that the distributions of  $X_{jk}=X_{jk}^{(n)}$ may depend on $n$. In the following we shall  investigate the rate of convergence of 
expected spectral distribution function $F_n(x)=\E\mathcal F_n(x)$ to 
the semi-circular distribution function by estimating the quantity $\Delta_n$. 
The main result of this paper is the following
\begin{thm}\label{main} Let $\E X_{jk}=0$, $\E X_{jk}^2=1$.  Assume that
\begin{equation}\label{moment}
 \sup_{n\ge1}\sup_{1\le j,k\le n}\E|X_{jk}|^4=: \mu_4<\infty.
\end{equation}
Assume as well that there exists a constant $D_0$ such that for all $n\ge 1$
\begin{equation}
 \sup_{1\le j,k\le n}|X_{jk}|\le D_0n^{\frac14}.
\end{equation}

Then,   there exists a positive  constant $C=C(D_0,\mu_4)$  depending on
 $D_0$ and $\mu_4$ only
such that
\begin{equation} \label{kolmog}
\Delta_n=\sup_x|F_n(x)-G(x)|\le Cn^{-1}.
\end{equation}

\end{thm}
\begin{cor}\label{cormain}Let $\E X_{jk}=0$, $\E X_{jk}^2=1$.  Assume that
\begin{equation}\label{moment1}
 \sup_{n\ge 1}\sup_{1\le j,k\le n}\E|X_{jk}|^8=:\mu_8<\infty.
\end{equation}
 Then,   there  exists a positive  constant $C=C(\mu_8)$  depending on
 $\mu_8$  only
such that
\begin{equation} \label{kolmog1}
\Delta_n\le Cn^{-1}.
\end{equation}
\end{cor}
\begin{rem}
 Note that the bound \eqref{kolmog} in Theorem \ref{main} and the bound \eqref{kolmog1} in the Corollary \ref{cormain} 
 are not improvable and coincide with the corresponding bounds in the Gaussian case.
\end{rem}
We state here as well the results for the Stieltjes transform of the  expected spectral distribution of the matrix $\mathbf W$.
Let $\mathbf R$ denote the resolvent matrix of the matrix $\mathbf W$,
\begin{equation}\notag
 \mathbf R:=\mathbf R(z)=(\mathbf W-z\mathbf I)^{-1}.
\end{equation}
Here and in what follows $\mathbf I$ denotes the unit matrix of corresponding dimension.
For any distribution function $F(x)$ we define the Stieltjes transform $s_F(z)$, for $z=u+iv$ with $v>0$, via formula
\begin{equation}\notag
 s_F(z)=\int_{-\infty}^{\infty}\frac1{x-z}dF(x). 
\end{equation}
Denote by $m_n(z)$ the Stieltjes transform of the distribution function $\mathcal F_n(x)$.
It is a well-known fact that
\begin{equation}\notag
 m_n(z)=\frac1n\sum_{j=1}^n\frac1{\lambda_j-z}=\frac1n\Tr \mathbf R.
\end{equation}
By $s(z)$ we denote the Stieltjes transform of the semi-circular law,
$$
s(z)=\frac{-z+\sqrt{z^2-4}}2.
$$
The Stieltjes transform of the semi-circular distribution
satisfies the equation
\begin{equation}\label{stsemi}
s^2(z)+zs(z)+1=0
\end{equation}
(see, for example,  \cite[ equality (4.20)]{GT:03}). 

Introduce for $z=u+iv$ and a positive constant $A_0>0$
 \begin{equation}\label{v0}
v_0:= A_0n^{-1},\text{ and }
\gamma:=\gamma(z):=|2-|u||. 
\end{equation}
For any $0<\varepsilon<\frac12$, and $A_0>0$, define a region
$\mathbb G=\mathbb G(A_0,n\varepsilon)\subset\mathbb C_+$, by 
\begin{equation}\label{definG}
 \mathbb G:=\{z=u+iv\in\mathbb C_+: -2+\varepsilon\le u\le 2-\varepsilon,\, v\ge v_0/\sqrt{\gamma(z)}\}.
\end{equation}
Let  $a>0$  be a positive number such that
\begin{equation} \label{constant*}
 \frac1{\pi}\int_{|u|\le a}\frac1{u^2+1}du=\frac34.
\end{equation}
We prove the following result.
\begin{thm}\label{stiltjesmain}Let  $\frac12>\varepsilon>0$ be a sequence of positive numbers in\eqref{definG} such that
\begin{equation}\label{avcond*}
 \varepsilon^{\frac32} =2v_0a.
\end{equation}
Assuming the conditions of Theorem \ref{main}, there exists a positive constant $C=C(D_0,A_0,\mu_4)$ depending on $D$, $A_0$ and $\mu_4$ only, such that, for 
$z\in\mathbb G$
\begin{align}\notag
 |\E m_n(z)-s(z)|\le \frac C{n v^{\frac34}}+\frac C{n^{\frac32}v^{\frac32}|z^2-4|^{\frac14}}.
\end{align}

\end{thm}

\subsection{Sketch of the Proof} 
{\noindent \bf 1.} We start with an estimate of the Kolmogorov-distance to the Wigner distribution via
an integral over the difference of the corresponding Stieltjes transforms along a contour 
in the upper half-plane using a  smoothing inequality \eqref{smoothing}  and Cauchy's formula
developed by the authors in \cite{GT:2013}
The resulting bound \eqref{smoothing10} involves an integral over a segment at a fixed distance, say 
$V=4$, from the real axis and a segment  $u+i A_0 n^{-1}(2-\abs{u})^{-\frac 12}, \; |u|\le 2$
at a distance of order $n^{-1}$ but avoiding to come close to the endpoints $\pm 2$ of the support.
These segments are part of the boundary of an $n$-dependent region $\mathbb G$ where
bounds of Stieltjes transforms are needed. Since the Stieltjes-transform
and the diagonal elements $R_{jj}(z)$ of the resolvent  of the Wigner-matrix $\mathbf W$ are uniformly bounded 
on the segment with $\Im z=V$  by $1/V$ (see Section \ref{firsttypeint})  proving a bound of order 
$O(n^{-1})$  for the latter segment near the x-axis is the essential problem.

{\noindent \bf 2.} In order to investigate this crucial part of the error 
we start  with the 2nd resolvent  or self-consistency equation for the expected Stieltjes transform resp.
the quantities $R_{jj}(z)$ of $\mathbf W$ (see \eqref{repr001} below) based on the difference of the resolvent 
of $\mathbf W^{(j)}$ ($j$th row and column removed) and $\mathbf W$. For the equivalent representation for the difference 
of Stieltjes transforms ( see \eqref{lambda'}) we have to show an error bound of order $O((nv)^{-\frac32}|z^2-4|^{-\frac14})$ for $z\in \mathbb G$. 
To prove this bound we use a recursive version
of this representations as in \eqref{repr01}. Obviously bounds for $\E|R_{jj}|^p$ for $z=u+iv$ 
close to the real line are needed for the sufficiently large $p$, ($p=O(\log n)$), which follow
once the  error terms $\varepsilon_j$ are small in the region $\mathbb G$. But proving that $\E|\varepsilon_j|^p$ is small requires in turn 
again bounds  of $\E|R_{jj}|^{2p}$.\\
 An approach suggested recently in \cite{SchleinMaltseva:2013} turns out to be very fruitful in dealing with this recursion problem.
 Assuming that $\E|R_{jj}|^{2p}\le C_0^{2p}$ for some $z=u+iv$, we can show that
$\E|R_{jj}|^p\le C_0^p$ with $z=u+iv/s_0$ with some fixed scale factor $s_0>1$. 
This allows us  to prove  by induction  a bound of type  $\E|R_{jj}|^q\le C_0^q$ for some fixed 
$q$ (independent of $n$) and $z=u+iv$ with $v\ge Cn^{-1}$ starting with $\E|R_{jj}|^p\le C_0^p$ for
 $p=s_0^q$  and $z=u+iv$ for fixed $v=4$, say. The latter assumption can be easily verified. \\
Note that one of the errors, that is  $\varepsilon_{j2}$,  in \eqref{repr001} is a quadratic form in 
independent random variables. Thus, in case that $\mathbf  W$ has entries with exponential or even 
sub-Gaussian tails,  inequalities for quadratic forms of independent random variables, like  \cite{GT:2013}, Lemma 3.8, or  \cite{SchleinMaltseva:2013}, 
Proposition A.1 could be applied.\\
Assuming eight moments in Corollary \ref{cormain}
or four moments and a truncation condition in Theorem \ref{main} only, we can't use these strong tail estimates
for quadratic forms anymore. Our solution is an recursive application of Burkholder's inequality
for the $p$th moment resulting in a bound involving moments of order $p/2$ of another quadratic form in independent variables in each step. 
This is the crucial part of the moment recursion for $R_{jj}$ described above.
Details of this procedure are described in Sections \ref{key} and \ref{diag}.

{\noindent \bf 3.} In Section \ref{expect} we prove a bound  for the error $\Lambda_n:=\E m_n(z)-s(z)$ of the
form 
$ (n v)^{-\frac34}+(n v)^{-\frac32}|z^2-4|^{-\frac14}$
which suffices to prove the rate $O(n^{-1})$
in Theorem  \ref{main}. Here we use a series of martingale-type
decompositions to evaluate the {\it expectation} $\E m_n(z)$ combined  with the bound $\E |\Lambda_n|^2 \le C (nv)^{-2}$ of Lemma  \ref{lam1} in the Appendix
which is again based on a recursive inequality  for
$\E |\Lambda_n|^2$ in \eqref{7.69}. A direct 
application of this bound to estimate the error terms
 $\varepsilon_{j4}$ would result in a less
precise bound of order $O(n^{-1}\log n)$ in Theorem \ref{main}. Bounds of such type will be shown
for the Kolmogorov distance of the {\it random} 
spectral distribution to Wigner's law in a separate
paper. For the expectation we provide sharper bounds
in Section  \ref{better} involving $m'_n(z)$.

{\noindent \bf 4.}
The necessary auxiliary bounds for all these steps
are collected in the Appendix.

\section{Bounds for the  Kolmogorov Distance of Spectral Distributions  via  Stieltjes Transforms}\label{smoothviastil}
To bound the error $\Delta_n$ we shall use an approach developed in previous work of the authors, see \cite{GT:03}.\\
We modify the bound of the  Kolmogorov distance between an arbitrary distribution function and the semi-circular distribution function
 via their Stieltjes transforms obtained in \cite[Lemma 2.1]{GT:03}. For $x\in[-2,2]$ define $\gamma(x):=2-|x|$.
 Given $\frac12>\varepsilon>0$ introduce the interval $\mathbb J_{\varepsilon}=\{x\in[-2,2]:\, \gamma(x)\ge\varepsilon\}$ and
$\mathbb J'_{\varepsilon}=\mathbb J_{\varepsilon/2}$.
For a distribution function $F$ denote by $S_F(z)$ its Stieltjes transform.
\begin{prop}\label{smoothing}Let $v>0$ and $a>0$ and $\frac12>\varepsilon>0$ be positive numbers such that
\begin{equation} \label{constant}
 \frac1{\pi}\int_{|u|\le a}\frac1{u^2+1}du=\frac34=:\beta,
\end{equation}
and
\begin{equation}\label{avcond}
 2va\le \varepsilon^{\frac32}.
\end{equation}
If $G$ denotes the  distribution function of the standard semi-circular law, and $F$ is any distribution function,
 there exist some absolute constants $C_1$ and $C_2$ such that
\begin{align}
\Delta(F,G)&:= \sup_{x}|F(x)-G(x)|\notag\\&\le 2
\sup_{x\in\mathbb J'_{\varepsilon}}\big|\im\int_{-\infty}^x(S_F(u+i\frac v{\sqrt{\gamma}})-S_G(u+i\frac v{\sqrt{\gamma}}))du\big|+C_1v
+C_2\varepsilon^{\frac32}.  \notag
\end{align}

\end{prop}
\begin{rem}\label{rem2.2}
 For any $x\in\mathbb J_{\varepsilon}$ we have
$\gamma=\gamma(x)\ge\varepsilon$
and according to condition \eqref{avcond},
$\frac{av}{\sqrt\gamma}\le \frac{\varepsilon}2$.
\end{rem}

For a  proof of this Proposition see  \cite[Proposition 2.1]{GT:2013} .

\begin{lem}\label{Cauchy}
 Under the conditions of Proposition \ref{smoothing}, for any $V>v$ and   $0<v\le \frac{\varepsilon^{3/2}}{2a}$
and $v'=v/\sqrt{\gamma}, \gamma= 2-|x|$, $x\in\mathbb J'_{\varepsilon}$
 as above, the following inequality holds
\begin{align}
 \sup_{x\in\mathbb J'_{\varepsilon}}&\left|\int_{-\infty}^x(\im(S_F(u+iv')-S_G(u+iv'))du\right|\notag\\&\le
\int_{-\infty}^{\infty}|S_F(u+iV)-S_G(u+iV)|du\notag\\&+
\sup_{x\in\mathbb J'_{\varepsilon}}\left|\int_{v'}^V\left(S_F(x+iu)-S_G(x+iu)\right)du\right|. \notag
\end{align}

\end{lem}

\begin{proof}Let $x\in \mathbb J'_{\varepsilon}$ be fixed. Let $\gamma=\gamma(x)$.
 Put $z=u+iv'$.   Since $v'=\frac v{\sqrt{\gamma}}\le \frac{\varepsilon}{2a}$, see \eqref{avcond},  we may assume without loss of
generality that $v'\le 4$
for  $x\in\mathbb J'_{\varepsilon}$.  Since the functions  $S_F(z)$ and $S_G(z)$ are analytic in the upper half-plane, it is enough to use Cauchy's theorem. We can write
 for $x\in\mathbb J'_{\varepsilon}$
\begin{equation} \notag
\int_{-\infty}^{x}\im(S_F(z)-S_G(z))du=\im\{\lim_{L\to\infty}\int_{-L}^x(S_F(u+iv')-S_G(u+iv'))du\}.
\end{equation}
By Cauchy's integral formula, we have
\begin{align}
 \int_{-L}^x(S_F(z)-S_G(z))du&=\int_{-L}^x(S_F(u+iV)-S_G(u+iV))du\notag\\&
+\int_{v'}^V(S_F(-L+iu)-S_G(-L+iu))du\notag\\&-\int_{v'}^V(S_F(x+iu)-S_G(x+iu))du. \notag
\end{align}
Denote by $\xi\text{ (resp. }\eta)$ a random variable with distribution function $F(x)$ (resp. $G(x)$). Then we have
\begin{equation} \notag
 |S_F(-L+iu)|=\left|\E\frac1{\xi+L-iu}\right|\le {v'}^{-1}\Pr\{|\xi|>L/2\}+\frac2L,
\end{equation}
for any $v'\le u\le V$.
Similarly,
\begin{equation} \notag
 |S_G(-L+iu)|\le {v'}^{-1}\Pr\{|\eta|>L/2\}+\frac2L.
\end{equation}
These inequalities imply that
\begin{equation} \notag
\left|\int_{v'}^V(S_F(-L+iu)-S_G(-L+iu))du\right|\to 0\quad\text{as}\quad L\to\infty,
\end{equation}
which completes the proof. 
\end{proof}
Combining the results of Proposition \ref{smoothing} and Lemma \ref{Cauchy}, we get
\begin{cor}\label{smoothing1}
 Under the conditions of Proposition \ref{smoothing} the following inequality holds
\begin{align}\label{smoothing10}
 \Delta(F,G)&\le 2\int_{-\infty}^{\infty}|S_F(u+iV)-S_G(u+iV)|du+C_1v_0+C_2\varepsilon^{\frac32}\notag\\&
  + 2 \sup_{x\in\mathbb J'_{\varepsilon}}\int_{v'}^V|S_F(x+iu)-S_G(x+iu)|du,
\end{align}
where $v'=\frac {v_0}{\sqrt{\gamma}}$ with $\gamma=2-|x|$ and $C_1,C_2 >0$ denote absolute constants.

\end{cor}

\section{Proof of Theorem \ref{main}}\begin{proof}
 We shall apply   Corollary \ref{smoothing1} to prove the Theorem \ref{main}. We choose $V=4$ and $v_0$ as defined in \eqref{v0} and use the quantity 
 $\varepsilon=(2av_0)^{\frac23}$ .
\subsection{Estimation of the First Integral in \eqref{smoothing1}  for $V=4$}\label{firsttypeint}
Denote by $\mathbb T=\{1,\ldots,n\}$. In the following we shall systematically
  use for any  $n\times n $ matrix $\mathbf W$
 together with its resolvent $\mathbf R$, its Stieltjes transform $m_n$ etc. the corresponding quantities $\mathbf W^{(\mathbb A)}$, its resolvent  $\mathbf R^{(\mathbb A)}$ 
 and its Stieltjes transform $m_n^{(\mathbb A)}$
 for the corresponding  sub matrix  with entries $X_{jk}, j, k \not \in \mathbb A$, $\mathbb A \subset \mathbb T=\{1,\ldots,n\}$.
Observe that
\begin{equation}\notag
 m_n^{(\mathbb A)}(z)=\frac1n\sum_{j\in\mathbb T_{\mathbb A}}\frac1{\lambda^{(\mathbb A)}-z}.
\end{equation}
Let $\mathbb T_{\mathbb A}=\mathbb T\setminus\mathbb A$.
By $\mathfrak M^{(\mathbb J)}$ we denote the $\sigma$-algebra generated by $X_{lk}$ with $l,k\in\mathbb T_{\mathbb J}$.
If $\mathbb A=\emptyset$ we shall omit the set $\mathbb A$ as exponent index. 

We shall use the representation
\be\label{repr01*}
R_{jj}=\frac1{-z+\frac1{\sqrt
n}X_{jj}-\frac1n{{\sum_{k,l\in\mathbb T_j}}}X_{jk}X_{jl}R^{(j)}_{kl}},
\en
(see,
for example,   \cite[equality (4.6)]{GT:03}). We may rewrite it as
follows
\begin{equation}\label{repr001}
 R_{jj}=-\frac1{z+m_n(z)}+\frac1{z+m_n(z)}\varepsilon_jR_{jj},
\end{equation}

where
$\varepsilon_j:=\varepsilon_{j1}+\varepsilon_{j2}+\varepsilon_{j3}+
\varepsilon_{j4}$   with
\begin{align}
\varepsilon_{j1}&:=\frac1{\sqrt n}X_{jj},\quad\varepsilon_{j2}:=-\frac1n{\sum_{k\ne l
\in\mathbb T_j}}X_{jk}X_{jl}R^{(j)}_{kl},\notag\\
\varepsilon_{j3}&:=-\frac1n{\sum_{k\in\mathbb T_j}}(X_{jk}^2-1)R^{(j)}_{kk},\quad
\varepsilon_{j4}:=\frac1n(\Tr \mathbf R-\Tr\mathbf R^{(j)}).\notag
\end{align}
Let
\begin{equation}\notag
\Lambda_n:=\Lambda_n(z):=m_n(z)-s(z)=\frac1n\Tr\mathbf R-s(z).    
\end{equation}
Summing equality \eqref{repr001} in $j=1,\ldots,n$ and solving with respect $\Lambda_n$, we get 
\begin{equation}\label{lambda'}
\Lambda_n=  m_n(z)-s(z)=\frac{T_n}{z+m_n(z)+s(z)},
\end{equation}
where
\begin{equation}\notag
 T_n=\frac1n\sum_{j=1}^n\varepsilon_jR_{jj}.
\end{equation}
Obvious bounds like $|z+s(z)|\ge 1$, $|\lambda_j-z|^{-1}\le v^{-1}$, $\max\{|R_{jj}^{(\mathbb J)}(z)|,\,|m_n^{(\mathbb J)}(z)|\}\le v^{-1}$,
 imply that for $V=4$ and for any $\mathbb J\subset \mathbb T$, 
\begin{align}
 |m_n^{(\mathbb J)}(z)|&\le \frac14\le \frac12|z+s(z)|,\notag\\ |s(z)-m_n^{(\mathbb J)}(z)|&\le \frac12 \le \frac12|z+s(z)|,\text{ a.s.}\notag
\end{align}
and therefore,
\begin{equation}\label{inequal10}
 |z+m_n^{(\mathbb J)}(z)+s(z)|\ge\frac12|z+s(z)|,\quad |z+m_n^{(\mathbb J)}(z)|\ge \frac12|s(z)+z|.
\end{equation}
Using  equality \eqref{lambda'}, we may write
\begin{align}
 \E\Lambda_n&=\frac1n\sum_{j=1}^n\E\frac{\varepsilon_{j}R_{jj}}{z+m_n(z)+s(z)}
 \notag\\&=\sum_{\nu=1}^4\frac1n\sum_{j=1}^n\E\frac{\varepsilon_{j\nu}s(z)}{z+m_n(z)+s(z)}+\sum_{\nu=1}^3\frac1n\sum_{j=1}^n\E\frac{\varepsilon_{j\nu}(R_{jj}-s(z))
 }{z+m_n(z)+s(z)}.\notag
\end{align}
We use that $\E\{\varepsilon_{j\nu}\Big|\mathfrak M^{(j)}\}=0$, for $\nu=1,2,3$ and obtain
\begin{align}
 \frac1n\sum_{j=1}^n\E\frac{\varepsilon_{j\nu}s(z)}{z+m_n(z)+s(z)}=-\E\frac{\varepsilon_{j\nu}\varepsilon_{j4}s(z)}{(z+m_n^{(j)}(z)+s(z))(z+m_n(z)+s(z))}.\notag
\end{align}
Thus, according to inequalities \eqref{inequal10}, Lemmas \ref{eps1*}, \ref{eps2}, \ref{eps3}, \ref{lem2} in the Appendix and equation \eqref{stsemi}, we obtain
\begin{align}
 |\frac1n\sum_{j=1}^n\E\frac{\varepsilon_{j\nu}s(z)}{z+m_n(z)+s(z)}|\le 4|s(z)|^3\frac1n\sum_{j=1}^n\E|\varepsilon_{j4}\varepsilon_{j\nu}|
 \le \frac {C|s(z)|^2}{n^{\frac32}},\notag
\end{align}
where $C$ depends on $\mu_4$ only.
For $\nu=4$,  Lemma \ref{lem2} in the Appendix, inequality \eqref{inequal10} and relation \eqref{stsemi} yield
\begin{align}\label{finish1}
 \frac1n\sum_{j=1}^n\frac{|s(z)||\varepsilon_{j4}|}{|z+m_n(z)+s(z)|}\le \frac Cn|s(z)|^2
\end{align}
with some absolute constant $C$.
Furthermore, applying the Cauchy -- Schwartz inequality and inequality \eqref{inequal10} and relation \eqref{stsemi}, we get
\begin{align}\label{eps0}
\Big|\sum_{\nu=1}^3\frac1n\sum_{j=1}^n\E\frac{\varepsilon_{j\nu}(R_{jj}-s(z))
 }{z+m_n(z)+s(z)} \Big|\le C|s(z)|\sum_{\nu=1}^3\frac1n\sum_{j=1}^n\E^{\frac12}|\varepsilon_{j\nu}|^2\E^{\frac12}|R_{jj}-s(z)|^2.
\end{align}
We may rewrite the representation \eqref{repr001} using $\Lambda_n=m_n(z)-s(z)$ and \eqref{stsemi} as (compare \eqref{repr01*})
\begin{align}\label{lambda''}
 R_{jj}=s(z)-s(z)\varepsilon_jR_{jj}-s(z)\Lambda_nR_{jj}.
\end{align}
Applying representations \eqref{lambda'} and \eqref{lambda''} together with \eqref{inequal10} and $|R_{jj}|\le \frac14$, we obtain
\begin{align}\label{eps}
 \E|R_{jj}(z)-s(z)|^2\le C|s(z)|^2\frac1n\sum_{l=1}^n\E|\varepsilon_l|^2.
\end{align}
Combining inequalities \eqref{eps0} and \eqref{eps}, we get
\begin{align}
 \Big|\sum_{\nu=1}^4\frac1n\sum_{j=1}^n\E\frac{\varepsilon_{j\nu}(R_{jj}-s(z))
 }{z+m_n(z)+s(z)}\Big|\le C|s(z)|^2\sum_{\nu=1}^4\frac1n\sum_{j=1}^n\E|\varepsilon_{j\nu}|^2.
\end{align}
Applying now Lemmas \ref{eps1*}, \ref{eps2}, \ref{eps3} and \ref{lem2}, we get
\begin{align}\label{finish2}
 \Bigg|\sum_{\nu=1}^3\frac1n\sum_{j=1}^n\E\frac{\varepsilon_{j\nu}(R_{jj}-s(z))
 }{z+m_n(z)+s(z)}\Bigg|\le \frac{C|s(z)|^2}{n}.
\end{align}
Inequality \eqref{finish1} and \eqref{finish2} together imply
\begin{equation}\label{disp}
 |\E\Lambda_n|\le \frac Cn|s(z)|^2.
\end{equation}

Consider now the integral
\begin{equation}\notag
 Int(V)=\int_{-\infty}^{\infty}|\E m_n(u+iV)-s(u+iV)|du
\end{equation}
for $V=4$.
Using inequality \eqref{disp}, we have
\begin{align}
 |Int(V)|&\le \frac Cn\int_{-\infty}^{\infty}|s(u+Vi)|^2du.\notag
\end{align}

Finally, we note that
\begin{equation}\label{finish6}
 \int_{-\infty}^{\infty}|s(z)|^2dx\le \int_{-\infty}^{\infty}\int_{-\infty}^{\infty}\frac1{(x-u)^2+V^2}dudF_n(x)\le \frac {\pi}V.
\end{equation}

Therefore,
\begin{equation}\label{final!}
 \int_{-\infty}^{\infty}|\E m_n(u+iV)-s(u+iV)|du\le \frac Cn.
\end{equation}
\subsection{Estimation of the Second Integral in \eqref{smoothing10}}To finish the proof of Theorem \ref{main} we need to bound the second integral
in \eqref{smoothing10} for $z\in\mathbb G$ and  $v_0=A_0n^{-1}$, where  $\varepsilon=(2av_0)^{\frac23}$   is defined in such a way that condition 
\eqref{avcond} holds.
We shall use the results of Theorem \ref{stiltjesmain}.
According to these results we have, for $z\in\mathbb G$,
\begin{equation}\label{jpf}
 |\E m_n(z)-s(z)|\le \frac C{n v^{\frac34}}+\frac C{n^{\frac32}v^{\frac32}|z^2-4|^{\frac14}}.
\end{equation}
We have
\begin{align}
\int_{v_0/\sqrt{\gamma}}^V|\E(m_n(x+iv)-s(x+iv))|dv&\le
 \frac1{n}\int_{\frac{v_0}{\sqrt{\gamma}}}^V\frac{dv}{v^{\frac34}}+\frac1{n\sqrt n\gamma^{\frac14}}\int_{\frac{v_0}{\sqrt{\gamma}}}^V\frac{dv}{v^{\frac32}}.\notag
\end{align}
After integrating we get
\begin{equation}\label{final!!}
 \int_{v_0/\sqrt{\gamma}}^V|\E(m_n(x+iv)-s(x+iv))|dv\le \frac C{n}+\frac{ C\gamma^{\frac14}}{n\sqrt n\gamma^{\frac14}v_0^{\frac12}}\le \frac Cn.
\end{equation}

Inequalities \eqref{final!} and \eqref{final!!} complete the proof of Theorem \ref{main}.
Thus Theorem \ref{main} is proved.
\end{proof}
\section{The proof of Corollary \ref{cormain}}To prove the Corollary \ref{cormain}
we consider truncated random variables $\widehat X_{jl}$ defined by
\begin{equation}\label{trunc000}
 \widehat X_{jl}:=X_{jl}\mathbb I\{|X_{jl}|\le cn^{\frac14} \}.
\end{equation}
Let $\widehat {\mathcal F}_n(x)$ denote the empirical spectral distribution function of the matrix $\widehat{\mathbf W}=\frac1{\sqrt n}(\widehat X_{jl})$.
\begin{lem}\label{trunc}
 Assuming the conditions of Theorem \ref{main} there exists a  constant $C>0$ depending on $\mu_8$ only such that
 \begin{equation}\notag
\E\{\sup_x|\mathcal F_n(z)-\widehat {\mathcal F}_n(x)|\}\le \frac{C}{n}.
 \end{equation}

\end{lem}
\begin{proof}
 We shall use  the rank inequality of Bai. See \cite{BaiSilv:2010},  Theorem A.43, p. 503.
 According this inequality
 \begin{equation}\notag
  \E\{\sup_x|{\mathcal F}_n(x)-\widehat {\mathcal F}_n(x)|\}\le \frac 1n\E\{\text{\rm rank}(\mathbf X-\widehat{\mathbf X})\}.
 \end{equation}
Observing that the rank of a matrix is not larger then numbers of its non-zero entries, we may write
\begin{equation}\notag
 \E\{\sup_x|{\mathcal F}_n(x)-\widehat{\mathcal F}_n(x)|\}\le \frac 1n\sum_{j,k=1}^n\E\mathbb I\{|X_{jk}|\ge Cn^{\frac14}\}\le \frac1{n^3}\sum_{j,k=1}^n\E|X_{jk}|^8\le \frac{C\mu_8}{n}.
\end{equation}
Thus, the Lemma is proved.
\end{proof}
Note that in the bound of the first integral in \eqref{smoothing10} we used the condition \eqref{moment} only. We shall compare the Stieltjes transform of the matrix 
$\widehat{\mathbf W}$ and the matrix obtained from $\widehat{\mathbf W}$ by centralizing and normalizing its entries. 
Introduce  $\widetilde X_{jk}=\widehat X_{jk}-\E\widehat X_{jk}$ and  $\widetilde{\mathbf W}=\frac1{\sqrt n}(\widetilde X_{jk})_{j,k=1}^n$. 
We normalize the r.v.'s $\widetilde {X}_{jk}$. 
Let $\sigma_{jk}^2=\E|\widetilde X_{jk}|^2$. We define the r.v.'s 
$\breve X_{jk}=\sigma_{jk}^{-1}\widetilde X_{jk}$. Finally, let $\breve m_n(z)$ ( resp. $\widehat m_n(z)$, $\widetilde m_n(z)$) denote Stieltjes transform of empirical spectral distribution 
function of the matrix $\breve{\mathbf W}=
\frac1{\sqrt n}(\breve X_{jk})_{j,k=1}^n$ (resp. $\widehat {\mathbf W}$, $\widetilde{\mathbf  W}$).
\begin{rem}\label{trunc00}
Note that
 \begin{equation}\label{trunc2}
  |\breve X_{jl}|\le D_1n^{\frac14}, \quad\E \breve X_{jl}=0 \text{ and }\E {\breve X}_{jk}^2=1,
 \end{equation}
 for some absolute constant $D_1$. That means that the matrix $\breve{\mathbf W}$ satisfies the conditions of Theorem \ref{stiltjesmain}.
\end{rem}
\begin{lem}\label{trunc2*}There exists some absolute constant $C$ depending on $\mu_8$ such that
\begin{equation}\notag
 \E|\widetilde m_n(z)-\breve m_n(z)|\le \frac{C}{n^{\frac32}v^{\frac32}}.
 \end{equation}
\end{lem}
\begin{proof}Note that
\begin{align}\notag
 \breve m_n(z)=\frac1n\Tr(\breve{\mathbf W}-z\mathbf I)^{-1}=:\frac1n\Tr\breve{\mathbf R},\,
 \widetilde m_n(z)=\frac1n\Tr(\widetilde{\mathbf W}-z\mathbf I)^{-1}=:\frac1n\Tr\widetilde{\mathbf R}.
\end{align}
Therefore,
\begin{equation}\label{su1}
 \widetilde m_n(z)-\breve m_n(z)=\frac1n\Tr(\widetilde{\mathbf R}-\breve{\mathbf R})
 =\frac1n\Tr (\widetilde{\mathbf W}-\breve{\mathbf W})\widetilde{\mathbf R}\widehat{\mathbf R}.
\end{equation}
Using the simple inequalities $|\Tr\mathbf A\mathbf B|\le \|\mathbf A\|_2\|\mathbf B\|_2$ and $\|\mathbf A\mathbf B\|_2\le\|\mathbf A\|\|\mathbf B\|_2$, we get
\begin{equation}\label{dif1}
\E|\widetilde m_n(z)-\breve m_n(z)|\le n^{-1}\E^{\frac12}\|\widetilde{\mathbf R}\|^2\|\breve{\mathbf R}\|_2^2\E^{\frac12}\|\widetilde{\mathbf W}-\breve{\mathbf W}\|_2^2.
\end{equation}
Furthermore, we note that, 
\begin{align}\label{dif2}
 \widetilde{\mathbf W}-\breve{\mathbf W}=\frac1{\sqrt n}((1-\sigma_{jk})\breve X_{jk}),
\end{align}
and
\begin{equation}\notag
 \|\widetilde{\mathbf W}-\breve{\mathbf W}\|_2\le \max_{1\le j,k\le n}\{1-\sigma_{jk}\}\|\breve{\mathbf W}\|_2.
\end{equation}
Since 
\begin{equation}\notag
 0<1-\sigma_{jk}\le 1-\sigma_{jk}^2\le Cn^{-\frac32}\mu_8,
\end{equation}
therefore
\begin{equation}\label{ss1}
 \E\|\widetilde{\mathbf W}-\breve{\mathbf W}\|_2^2\le C\mu_8^2n^{-2}.
\end{equation}

Applying Lemma \ref{resol00} inequality \eqref{res1} in the Appendix , we get
\begin{align}\label{4.7}
 \E^{\frac12}\|\widetilde{\mathbf R}\|^2\|\breve{\mathbf R}\|_2^2\le v^{-1}\E\|\breve{\mathbf R}\|_2^2\le v^{-1}\Big(\sum_{j,k}\E|\breve R_{jk}|^2\Big)^{\frac12}\le v^{-\frac32}\sqrt n(\im \breve m_n(z))^{\frac12}.
\end{align}

Usin now inequalities \eqref{ss1} and \eqref{4.7}, we obtain
\begin{equation}\notag
\E|\widetilde m_n(z)-\breve m_n(z)|\le Cn^{-\frac32}v^{-\frac32}\Big(\frac1n\sum_{j=1}^n\E|\breve {\mathbf R}_{jj}|\Big)^{\frac12}.
\end{equation}

According Remark \ref{trunc00}, we may apply Corollary \ref{cor8} in Section \ref{stiel} with $q=1$ to prove the claim.
Thus, Lemma \ref{trunc2*} is proved.
\end{proof}

\begin{lem}\label{trunc1}For some absolute constant $C>0$ we have
\begin{equation}\notag
 \E|\widetilde m_n(z)-\widehat m_n(z)|\le \frac{C\mu_8}{n^{\frac32}v^{\frac32}}.
 \end{equation}
\end{lem}
\begin{proof}
Similar to \eqref{su1}, we write
\begin{equation}\notag
 \widetilde m_n(z)-\widehat m_n(z)=\frac1n\Tr(\widetilde{\mathbf R}-\widehat{\mathbf R})
 =\frac1n\Tr (\widetilde{\mathbf W}-\widehat{\mathbf W})\widetilde{\mathbf R}\widehat{\mathbf R}.
\end{equation}
This yields
\begin{equation}\label{dif1*}
\E|\widetilde m_n(z)-\widehat m_n(z)|\le n^{-1}\E\|\widehat{\mathbf R}\|\|\widetilde{\mathbf R}\|_2\|\E\widehat{\mathbf W}\|_2.
\end{equation}
Furthermore, we note that, by definition \eqref{trunc000} and condition \eqref{moment1}, we have
\begin{align}\label{dif2*}
 |\E \widehat X_{jk}|&\le Cn^{-\frac 74}\mu_8.
\end{align}
Applying Lemma \ref{resol00}, inequality \eqref{res1}, in the Appendix and inequality \eqref{dif2*}, we obtain using $\|\widehat{\mathbf R}\|\le v^{-1}$,
\begin{align}\notag
 \E|\widetilde m_n(z)-\widehat m_n(z)|\le n^{-\frac74}v^{-\frac32}\E^{\frac12}|\widetilde m_n(z)|.
\end{align}
By Lemma \ref{trunc2*},
\begin{equation}\notag
 \E|\widetilde m_n(z)|\le \E|\breve m_n(z)|+C,
\end{equation}
for some constant $C$ depending on $\mu_8$ and $A_0$. According to Corollary \ref{cor8} in Section \ref{diag} with $q=1$
\begin{equation}\notag
 \E|\breve m_n(z)|\le \frac1n\sum_{j=1}^n\E|\breve R_{jj}|\le C,
\end{equation}
with a constant $C$ depending on $\mu_4$, $D_0$.
Using these inequalities, we get
\begin{equation}\notag
\E|\widetilde m_n(z)-\widehat m_n(z)|\le\frac {C\mu_8}{n^{\frac74}v^{\frac32}}\le \frac {C\mu_8}{n^{\frac32}v^{\frac32}}.
\end{equation}
Thus Lemma \ref{trunc1} is proved.
\end{proof}
\begin{cor}
 Assuming the conditions of Corollary \ref{cormain}, we have for $z\in \mathbb G$,
 \begin{equation}\notag
  |\E\widehat m_n(z)-s(z)|\le \frac{C}{(nv)^{\frac32}}+\frac C{n^2v^2\sqrt{\gamma}}.
 \end{equation}

\end{cor}
\begin{proof}
 The proof immediately follows from the inequality
 \begin{equation}\notag
  |\E\widehat m_n(z)-s(z)|\le |\E (\widehat m_n(z)-\breve m_n(z))|+|\E\breve m_n(z)-s(z)|,
 \end{equation}
Lemmas \ref{trunc2*} and \ref{trunc1} and Theorem \ref{stiltjesmain}.
\end{proof}

The proof of Corollary \ref{cormain} follows now from Lemma \ref{trunc}, Corollary \ref{smoothing1}, inequality \eqref{final!} and inequality
\begin{equation}\notag
 \sup_{x\in\mathbb J_{\varepsilon}}\int_{v_0/\sqrt{\gamma}}^V|\E\widehat m_n(x+iv)-s(x+iv)|dv\le \frac Cn.
\end{equation}

\section{Resolvent Matrices and Quadratic Forms}\label{stiel}The crucial problem in the proof of Theorem \ref{stiltjesmain} is the following bound for any $z\in\mathbb G$ 
\begin{equation}\notag
 \E|R_{jj}|^p\le C^p,
\end{equation}
for $j=1,\ldots,n$ and some absolute constant $C>0$.
To prove this bound we use an approach similar to the proof of Lemma 3.4 in \cite{SchleinMaltseva:2013}. 
In oder to arrive at our goal  we need additional bounds of quadratic forms of type
\begin{equation}\notag
 \E\Big|\frac1n\sum_{l\ne k}X_{jl}X_{jk}R^{(j)}_{kl}\Big|^p\le\left(\frac {Cp}{\sqrt {nv}}\right)^p. 
\end{equation}
To prove this bound we recurrently  use Rosenthal's and Burkholder's inequalities. 
\subsection{The Key Lemma}\label{key} In this Section we provide  auxiliary lemmas needed for the proof of Theorem \ref{main}.

For any $\mathbb J\subset \mathbb T$ introduce $\mathbb T_{\mathbb J}=\mathbb T\setminus\mathbb J$.
We introduce the quantity, for some $\mathbb J\subset \mathbb T$,
$$
B_p^{(\mathbb J)}:=\Big[\frac1n\sum_{q\in\mathbb T_{\mathbb J}}\Big(\sum_{r\in\mathbb T_{\mathbb J}}|R_{qr}^{(\mathbb J)}|^2\Big)^p\Big].
$$
By Lemma \ref{resol00},  inequality \eqref{res2} in the Appendix, we have
\begin{equation}\label{ineq00n}
\E B_p^{(\mathbb J)}\le v^{-p}\frac1n\sum_{q\in\mathbb T_{\mathbb J}}\E|R_{qq}^{(\mathbb J)}|^p.
\end{equation}
Furthermore, introduce the quantities
\begin{align}\label{qf}
Q^{(\mathbb J,k)}_{\nu}&=\sum_{l\in\mathbb T_{\mathbb J,k}}\Big|\sum_{r\in\mathbb T_{\mathbb J,k}\cap\{1,\ldots,l-1\}}X_{kr}a_{lr}^{(\mathbb J,k,\nu)}\Big|^2,\notag\\
Q^{(\mathbb J,k)}_{\nu1}&=\sum_{r\in\mathbb T_k}a_{rr}^{(\mathbb J,k,\nu+1)},\notag\\
Q^{(\mathbb J,k)}_{\nu2}&=\sum_{r\in\mathbb T_k}(X_{kr}^2-1)a_{rr}^{(\mathbb J,k,\nu+1)},\notag\\
Q^{(\mathbb J,k)}_{\nu3}&=\sum_{ r\ne q\in\mathbb T_k}X_{kr}X_{kq}a_{qr}^{(\mathbb J,k,\nu+1)},
\end{align}
where, $a^{(\mathbb J,k,0)}_{qr}$ are defined recursively via 
\begin{align}\label{defaq}
a_{qr}^{(\mathbb J,k,0)}&=\frac1{\sqrt n} R^{(\mathbb J,k)}_{qr},\notag\\
a_{qr}^{(\mathbb J,k,\nu+1)}&=\sum_{ l\in\{\max\{q,r\}+1,\ldots,n\}\cap\mathbb T_{\mathbb J,k}}a_{rl}^{(\mathbb J,k,\nu)}\overline a_{lq}^{(\mathbb J,k,\nu)},\text{ for }\nu=0,\ldots,L.
\end{align}
Using these notations we have
\begin{equation}\label{r1}
 Q^{(\mathbb J,k)}_{\nu}=Q^{(\mathbb J,k)}_{\nu1}+Q^{(\mathbb J,k)}_{\nu2}+Q^{(\mathbb J,k)}_{\nu3}.
\end{equation}
\begin{lem}\label{arr0}
 Under the conditions of Theorem \ref{main} we have
 \begin{align}\label{ar1}
\sum_{r\in\mathbb T_{\mathbb J,k}} |a_{qr}^{(\mathbb J,k,\nu+1)}|^2\le \Big(\sum_{l,r\in\mathbb T_{\mathbb J,k}}^n|a_{lr}^{(\mathbb J,k,\nu)}|^2\Big)
\Big(\sum_{l\in\mathbb T_{\mathbb J,k}}^n|a_{ql}^{(\mathbb J,k,\nu)}|^2\Big).
\end{align}
Moreover,
 \begin{align}\label{ar2}
  \sum_{q,r\in\mathbb T_{\mathbb J,k}}|a_{qr}^{(\mathbb J,k,\nu+1)}|^2\le (\sum_{q,r\in\mathbb T_{\mathbb J,k}}|a_{qr}^{(\mathbb J,k,\nu)}|^2)^2.
 \end{align}

\end{lem}
\begin{proof}
 We apply H\"older's inequality and obtain
 \begin{equation}\notag
  |a_{q,r}^{(\mathbb J,k,\nu+1)}|^2\le \sum_{l\in\mathbb T_{\mathbb J,k}}|a_{ql}^{(\mathbb J,k,\nu)}|^2\sum_{l\in\mathbb T_{\mathbb J,k}}|a_{lr}^{(\mathbb J,k,\nu)}|^2.
 \end{equation}
Summing in $q$ and $r$, \eqref{ar1} and \eqref{ar2} follow.

\end{proof}
\begin{cor}\label{arr1} Under the conditions of Theorem \ref{main} we have
\begin{align}\notag
\sum_{q,r\in\mathbb T_{\mathbb J,k}}|a_{qr}^{(\mathbb J,k,\nu)}|^2\le \left(\left(\im m_n^{(\mathbb J)}(z)+ \frac{1}{nv}\right)v^{-1}\right)^{2^{\nu}}
\end{align}
and 
 \begin{align}\notag
\sum_{r\in\mathbb T_k} |a_{qr}^{(\mathbb J,k,\nu)}|^2\le \left(\left(\im m_n^{(\mathbb J)}(z)+ 
\frac{1}{nv}\right)v^{-1}\right)^{2^{\nu}-1}\tg{n^{-1}}v^{-1}\im R_{qq}^{(\mathbb J,k)}.
\end{align} 
\end{cor}
\begin{proof}
 By definition of $a_{qr}^{(\mathbb J,k,0)}$, see  \eqref{defaq}, applying (Lemma \ref{resol00} equality \eqref{res1} in the Appendix), we get
 \begin{equation}
  \sum_{q,r\in\mathbb T_{\mathbb J,k}}|a_{qr}^{(\mathbb J,k,0)}|^2\le \frac1n\sum_{q,r\in\mathbb T_{\mathbb J,k}}|R_{qr}^{(\mathbb J,k)}|^2
  \le\left(\im m_n^{(\mathbb J)}(z)+ \frac{1}{nv}\right)v^{-1},
 \end{equation}
 and by definition \eqref{defaq},
 \begin{align}
 \sum_{r\in\mathbb T_{\mathbb J,k}}|a_{qr}^{(\mathbb J,k,0)}|^2\le \frac1n\sum_{r\in\mathbb T_{\mathbb J,k}}|R_{qr}^{(\mathbb J,k)}|^2.
 \end{align}

The general case follows now by induction in $\nu$, Lemma \ref{arr0}, and  Lemma \ref{resol00} inequality \eqref{res2} in the Appendix.
\end{proof}
\begin{cor}\label{arr2}
 Under the conditions of Theorem
  \ref{main} we have
 \begin{align}
  a_{rr}^{(\mathbb J,k,\nu+1)}\le \left(\left(\im m_n^{(\mathbb J)}(z)+ \frac{1}{nv}\right)v^{-1}\right)^{2^{\nu}-1}\tg{n^{-1}}v^{-1}\im R_{rr}^{(\mathbb J,k)}.
 \end{align}

\end{cor}
\begin{proof}
 The result immediately follows from the definition of $a_{rr}^{(k,\nu)}$ and Corollary \ref{arr1}.
\end{proof}
In what follows we shall use the notations
\begin{align}\label{not1}
 \Psi^{(\mathbb J)}&=\im m_n^{(\mathbb J)}(z)+\frac1{nv},
 \quad (A_{\nu,p}^{(\mathbb J)})^2=\E(\Psi^{(\mathbb J)})^{(2^{\nu}-1)2 p}, \quad 
 T_{\nu,p}^{(\mathbb J,k)}=\E|Q^{(\mathbb J,k)}_{\nu}|^p,\notag\\
A_p^{(\mathbb J)} &:=1+\E^{\frac14}|\Psi^{(\mathbb J)}|^{4p}.
\end{align}
Let $s_0$ denote some fixed number (for instance $s_0=2^8$). Let $A_1$ be a constant (to be chosen later) and $0<v_1\le 4$
a constant such that $v_0=A_0n^{-1}\le v_1$ for all $n\ge 1$.
\begin{lem}\label{bp1} Assuming the conditions of Theorem \ref{main} and for $p\le A_1(nv)^{\frac14}$
\begin{equation}\label{cond03}
\E|R_{jj}^{(\mathbb J)}|^p\le C_0^p,\text{ for }v\ge v_1,\text{ for all }j=1,\ldots,n,
\end{equation} 
we have for $v\ge v_1/s_0$  and  $p\le A_1(nv)^{\frac14}$, and $k\in\mathbb T_{\mathbb J}$

\begin{align}\label{q1}
 \E (Q^{(\mathbb J,k)}_0)^p\le6(\frac {C_3p}{\sqrt2})^{2p}v^{-p}A_p^{(\mathbb J)}.
\end{align}

\end{lem}

\begin{proof}

Using the representation \eqref{r1} and the triangle inequality, we get
\begin{equation}\label{po0}
\E| Q^{(\mathbb J,k)}_{\nu}|^p\le 3^p\Big(\E|Q^{(\mathbb J,k)}_{\nu1}|^p+\E|Q^{(\mathbb J,k)}_{\nu2}|^p+\E|Q^{(\mathbb J,k)}_{\nu3}|^p\Big).
\end{equation}
Let $\mathfrak M^{(\mathbb A)}$ denote the $\sigma$-algebra generated by r.v.'s 
$X_{j,l}$ for $j,l\in\mathbb T_{\mathbb A}$, for any set $\mathbb A$.
Conditioning on $\mathfrak M^{(\mathbb J,k)}$ ($\mathbb A=\mathbb J\cup \{k\}$) and applying Rosenthal's inequality (see Lemma \ref{rosent}), we get
\begin{align}
 \E|Q^{(\mathbb J,k)}_{\nu2}|^p\le C_1^pp^p\bigg(\E\Big(\sum_{r\in\mathbb T_{\mathbb J,k}}|a_{rr}^{(\mathbb J,k,\nu+1)}|^2\Big)^{\frac p2}
 +\sum_{r\in\mathbb T_{\mathbb J,k}}\E|a_{rr}^{(\mathbb J,k,\nu+1)}|^{p}\E|X_{kr}|^{2p}\bigg),
\end{align}
where $C_1$ denotes the  absolute constant in Rosenthal's inequality.
By Remark \ref{trunc00}, we get
\begin{align}\label{r01}
 \E|Q^{(\mathbb J,k)}_{\nu2}|^p&\le C_1^pp^p\Big(\E(\sum_{r\in\mathbb T_{\mathbb J,k}}|a_{rr}^{(\mathbb J,k,\nu+1)}|^2)^{\frac p2}+n^{\frac p2}\frac1{n}\sum_{r\in\mathbb T_{\mathbb J,k}}\E|a_{rr}^{(\mathbb J,k,\nu+1)}|^{p}\Big).
\end{align}
Analogously conditioning on $\mathfrak M^{(\mathbb J,k)}$ and applying Burkholder's inequality (see Lemma \ref{burkh1}), we get
\begin{align}\label{r3}
\E| Q^{(\mathbb J,k)}_{\nu3}|^p&\le C_2^pp^p\bigg(\E\Big(\sum_{r\in\mathbb T_{\mathbb J,k}}|\sum_{q\in\mathbb T_{\mathbb J,k}\cap\{1,\ldots,r-1\}}X_{kq}
a^{(\mathbb J,k,\nu+1)}_{rq}|^2\Big)^{\frac p2} 
\notag\\&\qquad\qquad\qquad\qquad+\sum_{q=1}^{n-1}\E\big|\sum_{r=1}^{q-1}X_{kr}a^{(\mathbb J,k,\nu+1)}_{rq}\big|^p\E|X_{kq}|^p\bigg),
\end{align}
where $C_2$ denotes the absolute constant in Burkholder's inequality.
Conditioning again on $\mathfrak M^{(\mathbb J,k)}$  and applying Rosenthal's inequality, we obtain
\begin{align}\label{r4}
 \E|\sum_{r\in\mathbb T_{\mathbb J,k}}X_{kr}a^{(\mathbb J,k,\nu+1)}_{rq}|^p&\le C_1^pp^p\Big(\E(\sum_{r=1}^{q-1}|a^{(\mathbb J,k,\nu+1)}_{rq}|^2)^{\frac p2}
 \notag\\&\qquad\qquad+\sum_{r\in\mathbb T_{\mathbb J,k}}\E|a^{(\mathbb J,k,\nu+1)}_{rq}|^p\E|X_{kr}|^p\Big).
\end{align}
Combining inequalities \eqref{r3} and \eqref{r4}, we get
\begin{align}
 \E| Q^{(\mathbb J,k)}_{\nu3}|^p&\le C_2^pp^p\E|Q^{(\mathbb J,k)}_{\nu+1}|^{\frac p2}+C_1^pC_2^pp^{2p}\sum_{q\in\mathbb T_{\mathbb J,k}}
 \E(\sum_{r\in\mathbb T_{\mathbb J,k}}|a^{(\mathbb J,k,\nu+1)}_{rq}|^2)^{\frac p2}\E|X_{kq}|^p\notag\\&\qquad\qquad\qquad+
C_1^pC_2^p p^{2p}\sum_{q\in\mathbb T_{\mathbb J,k}}\sum_{r\in\mathbb T_{\mathbb J,k}}\E|a^{(\mathbb J,k,\nu+1)}_{rq}|^p\E|X_{kq}|^p\E|X_{kr}|^p.
\end{align}
Using Remark \ref{trunc00}, this implies
\begin{align}\label{r5}
 \E| Q^{(\mathbb J,k)}_{\nu3}|^p&\le C_2^pp^p\E|Q^{(\mathbb J,k)}_{\nu+1}|^{\frac p2}+C_1^pC_2^pp^{2p}n^{\frac p4}
 \frac1n\sum_{q\in\mathbb T_{\mathbb J,k}}\E(\sum_{r\in\mathbb T_{\mathbb J,k}}|a^{(\mathbb J,k,\nu+1)}_{rq}|^2)^{\frac p2}\notag\\&\qquad\qquad\qquad+
 C_1^pC_2^pp^{2p}n^{\frac p2}\frac1{n^2}\sum_{q\in\mathbb T_{\mathbb J,k}}\sum_{r\in\mathbb T_{\mathbb J,k}}\E|a^{(\mathbb J,k,\nu+1)}_{rq}|^p.
\end{align}
Using the definition \eqref{qf} of $Q^{(\mathbb J,k)}_{\nu1}$ and the definition \eqref{defaq} of coefficients $ a^{(\mathbb J,k,\nu+1)}_{rr}$, 
it is straightforward to check that
\begin{equation}\label{r6}
 \E|Q^{(\mathbb J,k)}_{\nu1}|^p\le \E\Big[\sum_{q,r\in\mathbb T_{\mathbb J,k}}|a_{qr}^{(\mathbb J,k,\nu)}|^2\Big]^p.
\end{equation}
Combining \eqref{r01}, \eqref{r5} and \eqref{r6}, we get by \eqref{po0}
\begin{align}
 3^{-p}\E|Q^{(\mathbb J,k)}_{\nu}|^p&\le C_2^pp^p\E|Q^{(\mathbb J,k)}_{\nu+1}|^{\frac p2}
 +C_1^pC_2^pp^{2p}n^{\frac p4}\frac1n\sum_{q\in\mathbb T_{\mathbb J,k}}\E\Big(\sum_{r\in\mathbb T_{\mathbb J,k}}|a^{(\mathbb J,k,\nu+1)}_{rq}|^2\Big)^{\frac p2}\notag\\&+
 C_1^pC_2^pp^{2p}n^{\frac p2}\frac1{n^2}\sum_{q\in\mathbb T_{\mathbb J,k}}\sum_{r\in\mathbb T_{\mathbb J,k}}\E|a^{(\mathbb J,k,\nu+1)}_{rq}|^p\notag\\&
 +C_1^pC_2^p\E[\sum_{q,r\in\mathbb T_{\mathbb J,k}}|a_{qr}^{(\mathbb J,k,\nu)}|^2]^p\notag\\&+
 C_1^pC_2^pp^p\Big(\E\Big(\sum_{r\in\mathbb T_{\mathbb J,k}}|a_{rr}^{(\mathbb J,k,\nu+1)}|^2\Big)^{\frac p2}
 +n^{\frac p2}\frac1{n}\sum_{r\in\mathbb T_{\mathbb J,k}}\E|a_{rr}^{(\mathbb J,k,\nu+1)}|^{p}\Big).
\end{align}

Applying now Lemma \ref{arr0} and Corollaries \ref{arr1} and \ref{arr2}, we obtain 
\tg{\begin{align}
 \E|Q^{(\mathbb J,k)}_{\nu}|^p&\le C_3^pp^p\E|Q^{(\mathbb J,k)}_{\nu+1}|^{\frac p2}+
 C_3^p\E(\im m_n(z)+\frac{1}{nv})^{(2^{\nu}-1)p}v^{-(2^{\nu}-1) p}\notag\\&+C_3^pp^{2p}v^{-2^{\nu}p}n^{-\frac p4}\E(\im m_n^{(\mathbb J)}(z)
 +\frac{1}{nv})^{(2^{\nu}-1)\frac  2p} \Big(\frac1n\sum_{q\in\mathbb T_{\mathbb J,k}}|R_{qq}^{(\mathbb J,k)}|^{\frac p2}\Big)\notag\\
 &+C_3^pp^{2p}n^{-\frac p2}v^{-2^{\nu}p}\E(\im m_n^{(\mathbb J)}(z)+\frac{1}{nv})^{(2^{\nu}-1)p} 
 \Big(\frac1n\sum_{q\in\mathbb T_{\mathbb J,k}}|R_{qq}^{(\mathbb J,k)}|^{p}\Big),
\end{align}}
where $C_3=3C_1C_2$.
Applying Cauchy--Schwartz inequality, we may rewrite the last inequality in the form
\begin{align}\label{last1}
 \E|Q&^{(\mathbb J,k)}_{\nu}|^p\le C_3^pp^p\E|Q^{(\mathbb J,k)}_{\nu+1}|^{\frac p2}+
 C_3^p\E(\im m_n^{(\mathbb J)}(z)+\frac{1}{nv})^{(2^{\nu}-1) p}v^{-(2^{\nu}-1) p}\notag\\
 &+C_3^pp^{2p}v^{-2^{\nu}p}n^{-\frac p4}\E^{\frac12}(\im m_n^{(\mathbb J)}(z)
 +\frac{1}{nv})^{(2^{\nu}-1)p}\E^{\frac14}(\frac1n\sum_{q\in\mathbb T_{\mathbb J,k}}|R_{qq}^{(\mathbb J,k)}|^{2p})\notag\\
 &+C_3^pp^{2p}n^{-\frac p2}v^{-2^{\nu}p}\E^{\frac12}(\im m_n^{(\mathbb J)}(z)+\frac{1}{nv})^{(2^{\nu}-1)2p}
 \E^{\frac12}(\frac1n\sum_{q\in\mathbb T_{\mathbb J,k}}|R_{qq}^{(\mathbb J,k)}|^{2p}).
\end{align}

 Introduce the notation
\tg{ \begin{align}\notag
 \Gamma_p(z)&:=\E^{\frac12}\left(\frac1n\sum_{q\in\mathbb T_{\mathbb J,k}}|R_{qq}^{(\mathbb J,k)}|^{2p}\right).
 \end{align}}

We rewrite the inequality \eqref{last1} using $\Gamma_p(z)$ and the  notations of \eqref{not1} as follows
\begin{align}\label{rewrite}
 T_{\nu,p}^{(\mathbb J,k)}&\le (C_3p)^pT_{\nu+1,p/2}^{(\mathbb J,k)}+C_3^pA^{(\mathbb J)}_{\nu,p}v^{-(2^{\nu}-1) p}\notag\\&+
 (C_3p^2)^{p}\Big(v^{-2^{\nu}p}n^{-\frac p4}(A^{(\mathbb J)}_{\nu,p/2})^{\frac12}\Gamma_p^{\frac12}(z)+v^{-2^{\nu}p}n^{-\frac p2}A^{(\mathbb J)}_{\nu,p}
 \Gamma_p(z)\Big).
\end{align}

Note that
 \begin{equation}\notag
 A_{0,p}^{(\mathbb J)}=1,\quad
  A_{\nu,p/2^{\nu}}^{(\mathbb J)}\le \sqrt{1+\E(\Psi^{(\mathbb J)})^{2p}}\le 1+\E^{\frac14}(\Psi^{(\mathbb J)})^{4p},
 \end{equation}
 where $\Psi^{(\mathbb J)}=\im m_n^{(\mathbb J)}(z)+\frac{1}{nv}$.
Furthermore,
\begin{equation}\notag
 \Gamma_{p/2^{\nu}}\le \Gamma_{p}^{\frac1{2^{\nu}}}.
\end{equation}

Without loss of generality we may assume $p=2^L$ and $\nu=0,\ldots,L$.
We may write
\tg{\begin{align}\notag
 T_{0,p}^{(\mathbb J,k)}&\le (C_3p)^pT_{1,p/2}^{(\mathbb J,k)}+C_3^p+
 (C_3p^2)^{p}v^{- p}\Big(n^{-\frac p4}\Gamma_p^{\frac12}(z)+
 n^{-\frac p2}\Gamma_p(z)\Big).
\end{align}}
By induction we get
\begin{align}\label{t01}
 T_{0,p}^{(\mathbb J,k)}&\le \prod_{\nu=0}^L(C_3p/2^{\nu})^{p/2^{\nu}}T_{L,1}^{(\mathbb J,k)}+A_{p}^{(\mathbb J)}\sum_{l=1}^L
 \Big(\prod_{\nu=0}^{l-1}(C_3p/2^{\nu})^{p/2^{\nu}}\Big)v^{-(2^{l}-1)p/2^{l}}\notag\\&\qquad\qquad+
 A_{p}^{(\mathbb J)}v^{-p}\sum_{l=1}^L\Big(\prod_{\nu=0}^{l-1}(C_3p/2^{\nu})^{p/2^{\nu}}\Big)(n^{-p}\Gamma_p^2)^{\frac1{2^{l+1}}}
 \notag\\&\qquad\qquad + A_{p}^{(\mathbb J)}\sum_{l=1}^L\Big(\prod_{\nu=0}^{l-1}(C_3p/2^{\nu})^{p/2^{\nu}}\Big)(n^{-p}{\Gamma_p}^2)^{\frac1{2^{l}}}.
\end{align}
It is straightforward to check that
\begin{align}\notag
 \sum_{\nu=1}^{l-1}\frac{\nu}{2^{\nu}}=2(1-\frac{l+1}{2^{l}}).
\end{align}

Note that, for $l\ge 1$,
\begin{equation}\label{rel00}
\prod_{\nu=0}^{l-1}\big(C_3(p/2^{\nu})\big)^{p/2^{\nu}}=
\frac{(C_3p)^{2p(1-2^{-l})}}{2^{2p(1-\frac{l+1}{2^{l}})}}=2^{2p\frac l{2^{l}}}
\Big(\frac{C_3p}{2}\Big)^{2p(1-2^{-l})}.
\end{equation}
Applying this relation, we get
\begin{align}\notag
A_{p}^{(\mathbb J)}\sum_{l=0}^L\Big(\prod_{\nu=0}^{l-1}(C_3p/2^{\nu})^{p/2^{\nu}}\Big)v^{-(2^{l}-1)p/2^{l}}&\le A_p^{(\mathbb J)}(\frac{C_3p}2)^{2p}v^{-p}
\sum_{l=0}^{L-1}2^{\frac{2pl}{2^l}}\Big(\frac{4v}{C_3^2p^2}\Big)^{\frac p{2^l}}.
\end{align}
Note that for $l\ge 0$, $\frac l{2^l}\le \frac12$ and recall that $p=2^L$. Using this observation, we get
\begin{align}\notag
A_{p}^{(\mathbb J)}\sum_{l=0}^L\Big(\prod_{\nu=0}^{l-1}(C_3p/2^{\nu})^{p/2^{\nu}}\Big)v^{-(2^{l}-1)p/2^{l}}&\le A_p^{(\mathbb J)}(\frac{C_3p}2)^{2p}v^{-p}2^{p}
\sum_{l=0}^{L-1}\Big(\frac{4v}{C_3^2p^2}\Big)^{2^{L-l}}.
\end{align}
This implies that for 
$\frac {4v}{C_3^2p^2}\le \frac12$,
\begin{align}\notag
A_{p}^{(\mathbb J)}\sum_{l=1}^L\Big(\prod_{\nu=0}^{l-1}(C_3p/2^{\nu})^{p/2^{\nu}}\Big)v^{-(2^{l}-1)p/2^{l}}\le
(C_3p)^{2p}A_p^{(\mathbb J)}v^{-p}.
\end{align}
Furthermore, by definition of $T_{\nu,p}$, we have
\begin{equation}\notag
 T_{L,1}^{(\mathbb J,k)}=\E Q^{(\mathbb J,k)}_L\le \E\sum_{q,r\in\mathbb T_k}(a_{qr}^{(\mathbb J,k,L)})^2.
\end{equation}
Applying Corollary \ref{arr1} and H\"older's inequality, we get
\begin{align}\label{t02-}
 T_{L,1}^{(\mathbb J,k)}\le E(v^{-1}\Psi^{(\mathbb J)})^p\le v^{-p}A_p^{(\mathbb J)}.
\end{align}
By condition \eqref{cond03}, we have
\begin{equation}\notag
\Gamma_p:=\Gamma_p(u+iv)\le s_0^{2p}C_0^{2p}.
\end{equation}
Using this inequality, we get,
\tg{\begin{align}\label{t02+}
 A_{p}^{(\mathbb J)}v^{-p}\sum_{l=1}^L\Big(\prod_{\nu=0}^{l-1}&(C_3p/2^{\nu})^{p/2^{\nu}}\Big)n^{-\frac p{2^{l+1}}}\Gamma_p^{\frac 2{2^{l+1}}}
 \notag\\
 &\le  A_{p}^{(\mathbb J)}v^{-p}\sum_{l=0}^L\Big(\prod_{\nu=0}^{l-1}(C_3p/2^{\nu})^{p/2^{\nu}}\Big)
 (s_0^4C_0^4n^{-1})^{\frac p{2^{l+1}}}.
\end{align}}
Applying relation \eqref{rel00}, we obtain
\begin{align}
 A_{p}^{(\mathbb J)}v^{-p}\sum_{l=1}^L\Big(\prod_{\nu=0}^{l-1}&(C_3p/2^{\nu})^{p/2^{\nu}}\Big)n^{-\frac p{2^{l+1}}}\Gamma_p^{\frac1{2^{l}}}
 \notag\\
 &\le \Big(\frac{C_3p}{2}\Big)^{2p}A_{p}^{(\mathbb J)}v^{-p}\sum_{l=1}^L2^{2p\frac l{2^l}}
\Big(\frac{C_3p}{2}\Big)^{-\frac{2p}{2^{l}}}
 (s_0^4C_0^4n^{-1})^{\frac p{2^{l+1}}}\notag\\&=\Big(\frac{C_3p}{2}\Big)^{2p}A_{p}^{(\mathbb J)}v^{-p}\sum_{l=1}^L2^{p\frac l{2^{l-1}}}
\Big(\frac{C_3p}{2}\Big)^{-\frac{2p}{2^{l-1}}}
 ((s_0^4C_0^4n^{-1})^{\frac14})^{\frac p{2^{l-1}}}\notag\\&=\Big(\frac{C_3p}{\sqrt 2}\Big)^{2p}A_{p}^{(\mathbb J)}v^{-p}\sum_{l=1}^L
\Big(
 \frac{(s_0C_0)}{C_3pn^{\frac14}}\Big)^{{2^{L-l+1}}}.\notag
\end{align}
Without loss of generality we may assume that $C_3\ge 2(C_0s_0)$. Then we get 
\begin{align}\notag
 A_{p}^{(\mathbb J)}v^{-p}\sum_{l=0}^L\Big(\prod_{\nu=0}^{l-1}&(C_3p/2^{\nu})^{p/2^{\nu}}\Big)n^{-\frac p{2^{l+1}}}\Gamma_p^{\frac1{2^{l}}}\le 
 (C_3p)^{2p}A_{p}^{(\mathbb J)}v^{-p}.
\end{align}

Analogously we get
\begin{align}\label{t02*}
A_{p}^{(\mathbb J)}v^{-p}\sum_{l=1}^L\Big(\prod_{\nu=0}^{l-1}(C_3p/2^{\nu})^{p/2^{\nu}}\Big)n^{-\frac p{2^{l}}}\Gamma_p^{\frac1{2^{l-1}}}\le 
({C_3p})^{2p}v^{-p}A_p^{(\mathbb J)}.
\end{align}
Combining  inequalities \eqref{t01}, \eqref{t02-}, \eqref{t02+}, \eqref{t02*}, we finally arrive at
\begin{equation}\label{t0}
T_{0,p}^{(\mathbb J,k)}\le 6(C_3p)^{2p}v^{-p}A_p^{(\mathbb J)}.
\end{equation}

 Thus, Lemma \ref{bp1} is proved.
\end{proof}
\subsection{Diagonal Entries of the Resolvent Matrix}\label{diag}

We shall use the representation, for any $j\in\mathbb T_{\mathbb J}$,
\be\notag
R^{(\mathbb J)}_{jj}=\frac1{-z+\frac1{\sqrt
n}X_{jj}-\frac1n{{\sum_{k,l\in\mathbb T_{\mathbb J,j}}}}X_{jk}X_{jl}R^{(\mathbb J,j)}_{kl}}\,,
\en
(see,
for example, equality (4.6) in \cite{GT:03}). Recall the following relations (compare \eqref{repr001},\eqref{lambda''})
\begin{equation}\label{repr001a}
 R^{(\mathbb J)}_{jj}=-\frac1{z+m_n^{(\mathbb J)}(z)}+\frac1{z+m_n^{(\mathbb J)}(z)}\varepsilon_jR^{(\mathbb J)}_{jj},
\end{equation}
or
\be\label{repr01}
 R^{(\mathbb J)}_{jj}=s(z)-s(z)\Lambda_n^{(\mathbb J)}R^{(\mathbb J)}_{jj}-
s(z)\varepsilon_j^{(\mathbb J)}R^{(\mathbb J)}_{jj}, \en
where
$\varepsilon_j^{(\mathbb J)}:=\varepsilon_{j1}^{(\mathbb J)}+\varepsilon_{j2}^{(\mathbb J)}+\varepsilon_{j3}^{(\mathbb J)}+
\varepsilon_{j4}^{(\mathbb J)}$ with
\begin{align}
\varepsilon_{j1}^{(\mathbb J)}&:=\frac1{\sqrt n}X_{jj},\,
\varepsilon_{j2}^{(\mathbb J)}:=-\frac1n{\sum_{k\ne l\in\mathbb T_{\mathbb J,j}}}X_{jk}X_{jl}R^{(\mathbb J,j)}_{kl},\,
\varepsilon_{j3}^{(\mathbb J)}:=-\frac1n{\sum_{k\in\mathbb T_{\mathbb J,j}}}(X_{jk}^2-1)R^{(\mathbb J,j)}_{kk},
\notag\\
\varepsilon_{j4}^{(\mathbb J)}&:=\frac1n(\Tr \mathbf R^{(\mathbb J)}-\Tr\mathbf R^{(\mathbb J,j)}),
\Lambda_n^{(\mathbb J)}:=m_n^{(\mathbb J)}(z)-s(z)=\frac1n\Tr\mathbf R^{(\mathbb J)}-s(z).\label{epsjn}
\end{align}

Since $|s(z)|\le1$, the representation \eqref{repr01} yields, for any $p\ge1$,
\begin{equation}\label{in1}
|R_{jj}^{(\mathbb J)}|^p\le 3^p+3^p|\varepsilon_j^{(\mathbb J)}|^p|R_{jj}^{(\mathbb J)}|^p+3^p|\Lambda_n^{(\mathbb J)}|^p|R_{jj}^{(\mathbb J)}|^p.
\end{equation}
Applying the Cauchy -- Schwartz inequality, we get
\begin{equation}\label{ineq10}
\E|R_{jj}^{(\mathbb J)}|^p\le 3^p+3^p\E^{\frac12}|\varepsilon_j^{(\mathbb J)}|^{2p}\E^{\frac12}|R_{jj}^{(\mathbb J)}|^{2p}+3^p\E^{\frac12}|\Lambda_n^{(\mathbb J)}|^{2p}
\E^{\frac12}|R_{jj}^{(\mathbb J)}|^{2p}.
\end{equation}
We shall investigate now the behavior of $\E|\varepsilon_j^{(\mathbb J)}|^{2p}$ and $\E|\Lambda_n^{(\mathbb J)}|^{2p} $.
First we note,
\begin{equation}\notag
\E|\varepsilon_j^{(\mathbb J)}|^{2p}\le 4^{2p}\sum_{\nu=1}^4\E|\varepsilon_{j\nu}^{(\mathbb J)}|^{2p}.
\end{equation}
\begin{lem}\label{lem1} Assuming the  conditions of Theorem \ref{main} we have, for any $p\ge1$, and for any $z=u+iv\in\mathbb C_+$,
\begin{equation}\notag
\E|\varepsilon_{j1}^{(\mathbb J)}|^{2p}\le \frac{\mu_4}{n^{\frac p2+1}}.
\end{equation}
\end{lem}
\begin{proof}The proof follows immediately from the definition of $\varepsilon_{j1}$ and condition \eqref{moment}.
\end{proof}
\begin{lem}\label{lem2*}
Assuming the conditions of Theorem \ref{stiltjesmain} we have, for any $p\ge1$, and for any $z=u+iv\in\mathbb C_+$,
\begin{equation}\notag
\E|\varepsilon_{j4}^{(\mathbb J)}|^{2p}\le \frac1{n^{2p}v^{2p}}.
\end{equation}
\end{lem}
\begin{proof}For a proof of this Lemma see \cite[Lemma 4.1]{GT:03}.
\end{proof}
Let $A_1>0$ and $0\le v_1\le 4$ be a fixed.
\begin{lem}\label{lem5} Assuming the  conditions of Theorem \ref{main}, and assuming for all $\mathbb J\subset T$ with $|\mathbb J|\le L$ and 
all $l\in\mathbb T_{\mathbb J}$
\begin{equation}\label{cond3}
\E|R_{ll}^{(\mathbb J)}|^{q}\le C_0^{q}, \text{ for } 1\le q\le A_1(nv)^{\frac14}\text{  and for }v\ge v_1,
\end{equation} 
we have, for all $v\ge v_1/s_0$, and for all $\mathbb J\subset\mathbb T$ with $|\mathbb J|\le L-1$, 
\begin{equation}\notag
\E|\varepsilon_{j3}^{(\mathbb J)}|^{2p}\le(C_1p)^{2p}n^{-p}s_0^{2p}C_0^{4p},\text{ for } 1\le p\le A_1(nv)^{\frac14}.
\end{equation}
\end{lem}

\begin{proof}Recall that $s_0=2^4$ and  note that if $p\le  A_1(nv)^{\frac14}$ for $v\ge v_1/s_0$ then $q=2p\le A_1(nv)^{\frac14}$ for $v\ge v_1$.
 Let $v':=vs_0$. If $v\ge v_1/s_0$ then $v'\ge v_1$. We have
 \begin{equation}\label{rem1}
  q=2p\le 2A_1(nv)^{\frac14}= 2A_1(nv's_0^{-1})^{\frac14}=A_1(nv')^{\frac14}.
 \end{equation}
 We apply now Rosenthal's inequality for the moments of sums of independent random variables and get
\begin{align}\notag
\E|\varepsilon_{j3}^{(\mathbb J)}|^{2p}\le (C_1p)^{2p}n^{-2p}\Big(\E\big(\sum_{l\in\mathbb T_{\mathbb J,j}}|R_{ll}^{(\mathbb J,j)}|^{2}\big)^{p}
+\E|X_{jl}|^{4p}\sum_{l\in\mathbb T_{\mathbb J,j}}\E|R^{(\mathbb J,j)}_{ll}|^{2p}\Big).
\end{align}
According to inequality \eqref{rem1}  we may apply Lemma \ref{schlein} and  
 condition \eqref{cond3} for $q=2p$. We get, for $v\ge v_1/s_0$,
\begin{align}\notag
\E|\varepsilon_{j3}^{(\mathbb J)}|^{2p}\le(C_1p)^{2p}n^{-p}s_0^{2p}C_0^{2p}.
\end{align}
We use as well that by the conditions of Theorem \ref{main},
$ \E|X_{jl}|^{4p}\le D_0^{4p-4}n^{p-1}\mu_4$,
and by Jensen's inequality,
$(\frac1n\sum_{l\in\mathbb T_j}|R_{ll}^{(j)}|^{2})^{p}\le \frac1n\sum_{l\in\mathbb T_j}|R_{ll}^{(j)}|^{2p}$.
Thus, Lemma \ref{lem5} is proved.
\end{proof}
\begin{lem}\label{lem6}Assuming the  conditions of Theorem \ref{main}, 
condition \eqref{cond3}, 
for $v\ge v_1$ and $p\le A_1(nv)^{\frac14}$,
we have, for any $v\ge v_1/s_0$ and $p\le A_1(nv)^{\frac14}$,
\begin{align}\notag
\E|\varepsilon_{j2}^{(\mathbb J)}|^{2p}\le 6(C_3p)^{4p}n^{-p}v^{-p}A_p^{(\mathbb J)}+
2(C_3p)^{4p}n^{-p}v^{-p}(C_0s_0)^p.
\end{align}

\end{lem}
\begin{proof}We apply Burkholder's inequality for  quadratic forms. See Lemma \ref{burkh1} in the Appendix. We obtain
\begin{align}
\E|\varepsilon_{j2}^{(\mathbb J)}|^{2p}&\le (C_1p)^{2p}n^{-2p}\Bigg(\E\big(\sum_{l\in\mathbb T_{\mathbb J,j}}\big|
\sum_{r\in\mathbb T_{\mathbb J,j}\cap\{1,\ldots,l-1\}}X_{jr}
R^{(\mathbb J,j)}_{lr}\big|^2\big)^p\notag\\&+\max_{j,k}\E|X_{jk}|^{2p}
\sum_{l\in\mathbb T_{\mathbb J,j}}
\E\big|\sum_{r\in\mathbb T_{\mathbb J,j}\cap\{1,\ldots,l-1\}}X_{jr}
R^{(\mathbb J,j)}_{lr}\big|^{2p}\Bigg).\notag
\end{align}
Using now the quantity $Q_0^{(\mathbb J,j)}$ for the first term and Rosenthal's inequality and condition \eqref{moment} for the second term, we obtain
with Lemma \ref{resol00}, inequality \eqref{res2}, in the Appendix and $C_3=C_1C_2$
\begin{align}
\E|\varepsilon_{j2}^{(\mathbb J)}|^{2p}&
\le (C_2p)^{2p}n^{-p}\E|Q^{(\mathbb J,j)}_{0}|^{p}\notag\\&\qquad\qquad\qquad+(C_3p)^{4p}n^{-\frac{3p}2}\frac1n\sum_{l\in\mathbb T_{\mathbb J,j}}\E\big(
\sum_{r\in\mathbb T_{\mathbb J,j}\cap\{1,\ldots,l-1\}}|
R^{(\mathbb J,j)}_{lr}|^2\big)^{p}\notag\\&\qquad\qquad\qquad+(C_3p)^{4p}n^{-p}\frac1{n^2}
\sum_{l\in\mathbb T_{\mathbb J,j}}\sum_{r\in\mathbb T_{\mathbb J,j}\cap\{1,\ldots,l-1\}}\E|R^{(\mathbb J,j)}_{lr}|^{2p}\notag\\&\qquad
\le 
(C_2p)^{2p}n^{-p}\E|Q^{(\mathbb J,j)}_{0}|^{p}+(C_3p)^{4p}n^{-\frac{3p}2}v^{-p}\frac1n\sum_{l\in\mathbb T_{\mathbb J,j}}\E|R_{ll}^{(\mathbb J,j)}|^p\notag\\&
\qquad\qquad\qquad
+(C_3p)^{4p}n^{-p}v^{-p}\frac1{n^2}
\sum_{l\in\mathbb T_{\mathbb J,j}}\E|R_{ll}^{(\mathbb J,j)}|^p.\notag
\end{align}
By Lemma \ref{schlein} and condition \eqref{cond3}, we get
\begin{align}
\E|\varepsilon_{j2}^{(\mathbb J)}|^{2p}&\le 
(C_3p)^{2p}n^{-p}\E|Q^{(\mathbb J,j)}_{0}|^{p}
+2(C_3p)^{3p}n^{-p}v^{-p}(C_0s_0)^p.\notag
\end{align}
Applying now Lemma \ref{bp1}, we get the claim.
Thus, Lemma \ref{lem6} is proved.
\end{proof}
Recall that
\begin{equation}\notag
\Lambda_n^{(\mathbb J)}=\frac1n\Tr\mathbf R^{(\mathbb J)}-s(z),\quad\text{and}
\quad T_n^{(\mathbb J)}(z)=\frac1n\sum_{j\in\mathbb T_{\mathbb J}}\varepsilon_j^{(\mathbb J)}R_{jj}^{(\mathbb J)}.
\end{equation}

\begin{lem}\label{lem7}Assuming the conditions of Theorem \ref{main}, we have
\begin{equation}\notag
|\Lambda_n^{(\mathbb J)}|\le C(\sqrt{|T_n^{(\mathbb J)}(z)|}+\frac{\sqrt{|\mathbb J|}}{\sqrt n}).
\end{equation}

\end{lem}
\begin{proof}See e. g.  inequality (2.10) in \cite{SchleinMaltseva:2013}. For completeness we include short proof here. Obviously
\begin{align}\label{lambda00}
\Lambda_n^{(\mathbb J)}(z)&=
=\frac{s(z)(T_n^{(\mathbb J)}(z)-\frac{|\mathbb J|}n)}{z+2s(z)+\Lambda_n^{(\mathbb J)}(z)}.
\end{align} 
Denote by 
\begin{equation}\notag
 \widetilde T_n(z)=s(z)(T_n^{(\mathbb J)}(z)-\frac{|\mathbb J|}n).
\end{equation}
First we assume that $|z+2s(z)+\Lambda_n^{(\mathbb J)}|>\sqrt{|\widetilde T_n(z)|}$ Then the claim of Lemma \ref{lem7} holds. In the case $|z+2s(z)+\Lambda_n^{(\mathbb J)}|\le 
\sqrt{|\widetilde T_n(z)|}$, we assume $|\Lambda_n^{(\mathbb J)}|>2\sqrt{|\widetilde T_n(z)|}$.
Otherwise the Lemma is proved. Under this assumptions we have
\begin{equation}\label{5.39}
 |z+2s(z)|\ge |\Lambda^{(\mathbb J)}|-|z+2s(z)+\Lambda_n^{(\mathbb J)}|\ge \sqrt{|\widetilde T_n(z)|}
\end{equation}
On the other hand
\begin{equation}\label{5.40}
 |z+s(z)+m_n(z)|\ge \im|z+s(z)|\ge \frac12\im{z+2s(z)}=\frac12\sqrt {z^2-4}.
\end{equation}
We take here the branch of $\sqrt z$ such that $\im\sqrt z\ge0$. Note that for $z\in[-2,2]$ we have $\re{z^2-4}\le0$. This implies that
\begin{equation}\label{5.41}
 \im{z^2-4}\ge\frac{\sqrt 2}2|z^2-4|^{\frac12}=\frac{\sqrt 2}2|z+2s(z)|.
\end{equation}
Inequalities \eqref{5.39},  \eqref{5.40} and \eqref{5.41} together imply
\begin{equation}
 |z+s(z)+m_n^{(\mathbb J)}(z)|\ge \frac1{2\sqrt2}|z+2s(z)|\ge \frac{\sqrt2}2\sqrt{|\widetilde T_n(z)}.
\end{equation}
The last inequality and Eq. \eqref{lambda00} complete the proof of lemma \ref{lem7}.

\end{proof}
\begin{lem}\label{lem8}Assuming the  conditions of Theorem \ref{main} and 
condition \eqref{cond3},  we obtain, for $|\mathbb J|\le Cn^{\frac12}$
\begin{align}
\E|\Lambda_n^{(\mathbb J)}|^{2p}&\le\frac{C^p}{n^{\frac p4}}+ \Big(\frac{\mu_4}
{n^{\frac p2+1}}+\frac1{n^{2p}v^{2p}}+(C_1p)^{2p}n^{-p}s_0^{2p}C_0^{4p}\notag\\&\quad+6(C_3p)^{4p}n^{-p}v^{-p}A_{p}^{(\mathbb J)}+
2(C_3p)^{4p}n^{-p}v^{-p}(C_0s_0)^p\Big)^{\frac12}(C_0s_0)^p. \notag
\end{align}
\end{lem}
\begin{proof}
By Lemma \ref{lem6}, we have
\begin{equation}\notag
\E|\Lambda_n^{(\mathbb J)}|^{2p}\le C^p\E|T_n^{(\mathbb J)}(z)|^p+\frac {|\mathbb J|^{\frac p2}}{n^{\frac p2}}\le C^p\E|T_n^{(\mathbb J)}(z)|^p+\frac {C}{n^{\frac p4}}.
\end{equation}
Furthermore,
\begin{align}
\E|T_n^{(\mathbb J)}(z)|^p\le\Big(\frac1n\sum_{j\in\mathbb T}\E|\varepsilon_{j}^{(\mathbb J)}|^{2p}\Big)^{\frac12}
\Big(\frac1n\sum_{j\in\mathbb T}\E|R_{jj}^{(\mathbb J)}|^{2p}\Big)^{\frac12}.\notag
\end{align}
Lemmas \ref{lem1} -- \ref{lem6} together with Lemma \ref{schlein} imply
\begin{align}\label{ineq1}
\frac1n\sum_{j\in\mathbb T}\E|\varepsilon_{j}^{(\mathbb J)}|^{2p} \le  & 4^{2p-1}\Big( \frac{\mu_4}
{n^{\frac p2+1}}+\frac1{n^{2p}v^{2p}}+(C_1 p)^{2p}n^{-p}s_0^{2p}C_0^{4p}\notag\\&
+6(C_3p)^{4p}n^{-p}v^{-p}A_{p}^{(\mathbb J)}+2(C_3p)^{4p}n^{-p}v^{-p}(C_0s_0)^p\Big).
\end{align}

Thus, Lemma \ref{lem8} is proved.
\end{proof}
\begin{lem}\label{lem01}Assuming the  conditions of Theorem \ref{main} and 
condition \eqref{cond3}, there exists an absolute constant $C_4$ such that, for
$p\le A_1(nv)^{\frac14}$ and $v\ge v_1/s_0$, we have, uniformly in  $\mathbb J\subset \mathbb T$ such that $|\mathbb J|\le Cn^{\frac12}$,
\begin{align}\notag
A_p^{(\mathbb J)}\le  C_4^p.
\end{align}
\end{lem}
\begin{proof}

We start from the obvious inequality, using $|s(z)|\le 1$,
\begin{align}\notag
\E|\Psi^{(\mathbb J)}|^{2p}\le 3^{2p}(1+(nv)^{-2p}+\E|\Lambda_n^{(\mathbb J)}(z)|^{2p}).
\end{align}
Furthermore, applying Lemma \ref{lem8}, we get
\begin{align}\label{a10}
\E|\Psi^{(\mathbb J)}|^{2p}&\le 3^{2p}\Bigg(1+(nv)^{-p}+\Big(\frac{\mu_4}
{n^{\frac p2+1}}+\frac1{n^{2p}v^{2p}}+(C_1p)^{2p}n^{-p}s_0^{2p}C_0^{4p}\notag\\&+6(C_3p)^{4p}n^{-p}v^{-p}A_{p}^{(\mathbb J)}+
2(C_3p)^{4p}n^{-p}v^{-p}(C_0s_0)^p\Big)^{\frac12}(C_0s_0)^p\Bigg).
\end{align}
By definition,
\begin{equation}\label{a1}
A_p^{(\mathbb J)}\le 1+\E^{\frac12}(\Psi^{(\mathbb J)})^{2p}.
\end{equation}
Inequalities \eqref{a1} and \eqref{a10} together imply
\begin{align}
A_p^{(\mathbb J)}&\le1+3^{p}\Bigg(1+(nv)^{-\frac{p}2}+(C_0s_0)^{\frac p2}\Big(\mu_4^{\frac14}n^{-\frac p8}+\frac1{n^{\frac p2}v^{\frac p2}}
+(C_1p)^{\frac p2}n^{-\frac p4}s_0^{\frac p2}C_0^{p}\notag\\&+3(C_3p)^{p}n^{-\frac p4}v^{
-\frac p4}(A_p^{(\mathbb J)})^{\frac14}+
2(C_3p)^{p}n^{-\frac p4}v^{-\frac p4}(C_0s_0)^{\frac p4}\Big)\Bigg).\notag
\end{align}
Let $C'=s_0\max\{9,C_3^{\frac12}, C_0^{\frac32}C_1^{\frac12}, 3C_3C_0^{\frac12}, C_3C_0^{\frac34}\}$.
Using Lemma \ref{simple} with $x=(A_p^{(\mathbb J)})^{\frac14}$, $t=4$, $r=1$, we get
\begin{align}
A_p^{(\mathbb J)}&\le {C'}^p\Big(1+(nv)^{-\frac{p}2}+\mu_4^{\frac12}n^{-\frac p8}+\frac1{n^{\frac p2}v^{\frac p2}}\notag\\&\qquad\qquad\qquad
+p^{\frac p2}n^{-\frac p4}+p^{p}n^{-\frac p4}v^{-\frac p4}+
p^{\frac{4p}3}n^{-\frac p3}v^{-\frac p3}\Big).\notag
\end{align}
For $p\le A_1(nv)^{\frac14}$, we get, for $z\in\mathbb G$,
\begin{equation}\notag
A_p^{(\mathbb J)}\le C_4^p,
\end{equation}
where $C_4$ is some absolute constant. We may take $C_4=2C'$.
\end{proof}
\begin{cor}\label{cor4.8}Assuming the  conditions of Theorem \ref{main} and  
condition \eqref{cond3}, we have , for $v\ge v_1/s_0$, and for any $\mathbb J\subset \mathbb T$ such that $|\mathbb J|\le \sqrt n$
\begin{align}\label{lll}
\E|\Lambda_n^{(\mathbb J)}|^{2p}&\le  C_0^{2p}\Big(\frac{4^{\frac p4}\mu_4^{\frac12}s_0^{p}}{n^{\frac p4}v^{\frac p4}}+\frac{s_0^{\frac p2}}{n^pv^p}
+\frac{C_5^pp^{2p}}{n^{\frac p2}v^{\frac p2}}\Big),
\end{align}
where
\begin{align}\notag
 C_5:=4C_1^2s_0^{4}+6^{\frac1p}C_3^4C_4+2^{\frac1p}C_3^4s_0^3.
\end{align}

\end{cor}
\begin{proof}
 Without loss of generality we may assume that $C_0>1$. The bound \eqref{lll} follows now from Lemmas \ref{lem8} and \ref{lem01}.
\end{proof}

\begin{lem}\label{lem9}Assuming the  conditions of Theorem \ref{main} and
condition \eqref{cond3} for $\mathbb J\subset \mathbb T$ such that $|\mathbb J|\le L\le \sqrt n$, there exist positive constant $A_0, C_0, A_1$ depending on $\mu_4, D_0$ only 
, such that we have, for $p\le A_1(nv)^{\frac14}$ and $v\ge v_1/s_0$  uniformly in $\mathbb J$ and $v_1$
\begin{align}\notag
\E|R_{jj}^{(\mathbb J)}|^p\le C_0^p
\end{align}
with $|\mathbb J|\le L-1$.
\end{lem}
\begin{proof}According to inequality \eqref{ineq10}, we have
\begin{align}
\E|R_{jj}^{(\mathbb J)}|^p&\le4^p(1+(\E^{\frac12}|\Lambda_n^{(\mathbb J)}|^{2p}+\E^{\frac12}|\varepsilon_j^{(\mathbb J)}|^{2p})
\E^{\frac12}|R_{jj}^{(\mathbb J)}|^{2p}).\notag
\end{align}
Applying  condition \eqref{cond3}, we get
\begin{align}
\E|R_{jj}^{(\mathbb J)}|^p&\le 4^p(1+(\E^{\frac12}|\Lambda_n^{(\mathbb J)}|^{2p}+\E^{\frac12}|\varepsilon_{j1}^{(\mathbb J)}|^{2p}
+\cdots+\E^{\frac12}|\varepsilon_{j4}^{(\mathbb J)}|^{2p})s_0^pC_0^p).\notag
\end{align} 
Combining results of Lemmas \ref{lem1} -- \ref{lem6} and Corollary \ref{cor4.8},  we obtain
\begin{align}
\E|&R_{jj}^{(\mathbb J)}|^p\le5^p\Bigg(1+s_0^pC_0^{2p}\Big(\frac{4^{\frac p4}\mu_4^{\frac14}s_0^{p}}{n^{\frac p4}v^{\frac p4}}+\frac{s_0^{p}}{n^pv^p}
+\frac{C_5^pp^{2p}}{n^{\frac p2}v^{\frac p2}}\Big)^{\frac12}\notag\\&\qquad\qquad\qquad+
s_0^pC_0^{3p}\bigg(\frac{4^{\frac p4}\mu_4^{\frac14}}{n^{\frac p4}v^{\frac p4}}
+\frac{s_0^{p}}{n^pv^p}
+\frac{C_5^pp^{2p}}{n^{\frac p2}v^{\frac p2}}\notag
 \bigg)\Bigg).\notag
\end{align}
We may rewrite the last inequality as follows
\begin{align}
\E|R_{jj}^{(\mathbb J)}|^p\le C_0^p\Big(\frac{5^p}{C_0^p}+\frac{{\widehat C}_1^{\frac p8}}{(nv)^{\frac p8}} +\frac {({\widehat C}_2p^4)^{\frac p4}}{(nv)^{\frac p4}}+
\frac {({\widehat C}_3p^4)^{\frac p2}}{(nv)^{\frac p2}}+
\frac {{\widehat C}_4^{ p}}{(nv)^{p}}\Big),\notag
\end{align}
where
\begin{align}
 {\widehat C}_1&=5^8s_0^{12}C_0^8\mu_4^{\frac1p},\notag\\
 {\widehat C}_2&=5^4s_0^4C_0^4C_5^2(1+2C_0^4\mu_4^{\frac2p}),\notag\\
 {\widehat C}_3&=5^2s_0^2C_0^2(s_0+C_5^2),\notag\\
 {\widehat C}_4&=5C_0^2s_0^2.\notag
\end{align}
Note that for 
\begin{equation}\label{a0}
 A_0\ge 2^8A_1^4\max\{{\widehat C}_1,\ldots,{\widehat C}_4\}
\end{equation}
 and $C_0\ge 25$, we obtain
 that  
\begin{equation}\notag
\E|R_{jj}^{(\mathbb J)}|^p\le C_0^p.
\end{equation}
Thus Lemma \ref{lem9} is proved.
\end{proof}

\begin{cor}\label{cor8}Assuming the  conditions of Theorem \ref{main}, 
  we have, for $p\le 8$ and $v\ge v_0= A_0n^{-1}$
there exist a  constant $C_0>0$ depending on $\mu_4$ and $D_0$ only such that for all $1\le j\le n$ and all $z\in\mathbb G$
\begin{equation}\label{r11}
\E|R_{jj}|^p\le C_0^p,
\end{equation}
and
\begin{equation}\label{r12}
 \E\frac1{|z+m_n(z)|^p}\le C_0^p.
\end{equation}

\end{cor}
\begin{proof}Let $L=[-\log_{s_0}v_0]+1$. Note that
$s_0^{-L}\le v_0$ and $A_1\frac {n^{\frac14}}{s_0^{\frac L4}}\ge A_1(nv_0)^{\frac14}$.
We may choose $C_0=25$ and $A_0, A_1$ such that \eqref{a0} holds and
\begin{equation}\notag
A_1(nv)^{\frac14}\ge 8.
\end{equation}
 Then, 
for $v=1$, and for any $p\ge 1$, for any set $\mathbb J\subset\mathbb T$ such that $|\mathbb J|\le L$
\begin{equation}\label{in10}
\E|R_{jj}^{(\mathbb J)}|^p\le C_0^p.
\end{equation}
By Lemma \ref{lem9}, inequality \eqref{in10} holds for $v\ge 1/s_0$ and for
$p\le A_1n^{\frac14}/s_0^{\frac14}$ and for $\mathbb J\subset\mathbb T$ such that $|\mathbb J|\le L-1$. After repeated
application of Lemma  \ref{lem9} (with  \eqref{in10} as assumption valid for $v \ge 1/s_0$) 
we arrive at the conclusion that the
inequality \eqref{in10} holds for $v\ge 1/s_0^{2}$,
$p\le A_1n^{\frac14}/s_0^{\frac12}$ and all $\mathbb J\subset\mathbb T$ such that $|\mathbb J|\le L-2$. 
Continuing this iteration inequality \eqref{in10} finally  holds
for $v\ge A_0n^{-1}$, $p\le 8$  and $\mathbb J=\emptyset$.\\
The proof of inequality of \eqref{r12} is similar:  We have by \eqref{stsemi} 
\begin{equation}\label{recip}
 \frac1{|z+m_n(z)|}\le \frac1{|s(z)+z|}+\frac{|\Lambda_n|}{|z+m_n(z)||z+s(z)|}\le |s(z)|(1+\frac{|\Lambda_n|}{|z+m_n(z)|}).
\end{equation}
Furthermore, using that $|m_n'(z)|\le |\frac1n\Tr \mathbf R^2|\le \frac1n\sum_{j=1}^n\E|R_{jj}|^2$ and 
Lemma \ref{resol00} inequality \eqref{res1} in the Appendix, we get
\begin{equation}\notag
 |\frac d{dz}\log(z+m_n(z))|\le \frac{|1+m_n'(z)|}{|z+m_n(z)|}\le \frac1v\frac{v+\im m_n(z)}{|z+m_n(z)|}\le \frac1v.
\end{equation}
By integration, this implies that (see the proof of Lemma \ref{schlein})
\begin{equation}\label{inverse}
\frac1{|(u+iv/s_0)+m_n(u+iv/s_0)|}\le \frac {s_0}{|(u+iv)+m_n(u+iv)|}.
\end{equation}
Inequality \eqref{recip} and the Cauchy--Schwartz inequality together imply 
\begin{align}\notag
 \E\frac1{|z+m_n(z)|^p}\le 2^p|s(z)|^p(1+\E^{\frac12}|\Lambda_n|^{2p}\E^{\frac12}\frac1{|z+m_n(z)|^{2p}}).
\end{align}
Applying inequality \eqref{inverse}, we obtain
\begin{align}\notag
 \E\frac1{|z+m_n(z)|^p}&\le 2^p|s(z)|^p(1+\E^{\frac12}|\Lambda_n|^{2p}s_0^pC_0^p).
\end{align}
Using Corollary \ref{cor4.8}, we get, for $v\ge 1/s_0$
\begin{align}
 \E\frac1{|z+m_n(z)|^p}&\le 2^p|s(z)|^p\Big(1+\Big(\frac{4^{\frac p8}\mu_4^{\frac14}s_0^{\frac p2}}{n^{\frac p8}v^{\frac p8}}
 +\frac{s_0^{\frac p4}}{n^{\frac p2}v^{\frac p2}}
+\frac{C_5^pp^{p}}{n^{\frac p4}v^{\frac p4}}\Big)s_0^pC_0^{2p}\Big).\notag
\end{align}
Thus inequality \eqref{r12} holds for $v\ge 1/s_0$ as well. Repeating this argument inductively with $A_0,A_1,C_)$ satisfying \eqref{a0}
for the regions $v \ge s_0^{-\nu}$, for $\nu=1,\ldots,L$ and  $z \in \mathbb G$, we get the claim.
Thus, Corollary \ref{cor8} is proved.

\end{proof}

\section{{Proof} of Theorem \ref{stiltjesmain}}\label{expect}
We return now to the representation \eqref{repr01} which implies that
\begin{align}\label{lambda}
s_n(z)&=\frac1n\sum_{j=1}^n\E R_{jj}=s(z)+\E\Lambda_n=s(z)+\E\frac{T_n(z)}{z+s(z)+m_n(z)}.
\end{align}
We may continue the last equality as follows
\begin{align}\label{eq00}
s_n(z)=s(z)+\E\frac{\frac1n\sum_{j=1}^n\varepsilon_{j4}R_{jj}}{z+s(z)+m_n(z)}+
\E\frac{\widehat T_n(z)}{z+s(z)+m_n(z)},
\end{align}
where 
$$
\widehat T_n=\sum_{\nu=1}^3\frac1n\sum_{j=1}^n\varepsilon_{j\nu}R_{jj}.
$$
Note that the definition of $\varepsilon_{j4}$ in \eqref{repr01} and equality \eqref{shur} together imply 
\begin{equation}\label{7.3}
\frac1n\sum_{j=1}^n\varepsilon_{j4}R_{jj}=\frac1n\Tr \mathbf R^2=\frac1n\frac{d m_n(z)}{dz}.
\end{equation}
Thus we may rewrite \eqref{eq00} as
\begin{align}\label{eq01}
s_n(z)&=s(z)+\frac1n\E\frac{ m_n'(z)}{z+s(z)+m_n(z)}+
\E\frac{\widehat T_n(z)}{z+s(z)+m_n(z)}.
\end{align}

Denote by
\begin{equation}\mathfrak T=\E\frac{\widehat T_n(z)}{z+s(z)+m_n(z)}.
\end{equation}

\subsection{Estimation of  $\mathfrak T$} We represent $\mathfrak T$ 
\begin{align}\notag
\mathfrak T=\mathfrak T_{1}+\mathfrak T_{2},
\end{align}
where
\begin{align}
\mathfrak T_{1}&=-\frac1n\sum_{j=1}^n\sum_{\nu=1}^3\E\frac{\varepsilon_{j\nu}
\frac1{z+m_n^{(j)}(z)}}{z+m_n(z)+s(z)},\notag\\
\mathfrak T_{2}&=\frac1n\sum_{j=1}^n\sum_{\nu=1}^3
\E\frac{\varepsilon_{j\nu}(R_{jj}+\frac1{z+m_n^{(j)}(z)})}{z+m_n(z)+s(z)}.\notag
\end{align}
\subsubsection{Estimation  of $\mathfrak T_{1}$}
We may decompose  $\mathfrak T_{1}$ as
\begin{align}\label{tii}
\mathfrak T_{1}=\mathfrak T_{11}+\mathfrak T_{12},
\end{align} 
where
\begin{align}
\mathfrak T_{11}&=-\frac1n\sum_{j=1}^n\sum_{\nu=1}^3\E\frac{\varepsilon_{j\nu}
\frac1{z+m_n^{(j)}(z)}}{z+m_n^{(j)}(z)+s(z)},\notag\\
\mathfrak T_{12}&=-\frac1n\sum_{j=1}^n\sum_{\nu=1}^3\E\frac{\varepsilon_{j\nu}\varepsilon_{j4}
\frac1{z+m_n^{(j)}(z)}}{(z+m_n^{(j)}(z)+s(z))(z+m_n(z)+s(z))}.\notag
\end{align}

It is easy to see that, by conditional expectation
\begin{equation}\label{fin104}
\mathfrak T_{11}=0.
\end{equation}
Applying the Cauchy--Schwartz inequality, for $\nu=1,2,3$,  we get
\begin{align}\label{fin100}
\Bigg|\E&\frac{\varepsilon_{j\nu}\varepsilon_{j4}
\frac1{z+m_n^{(j)}(z)}}{(z+m_n^{(j)}(z)+s(z))(z+m_n(z)+s(z))}\Bigg|\notag\\
&\le 
\E^{\frac12}\Bigg|\frac{\varepsilon_{j\nu}}{(z+m_n^{(j)}(z))(z+m_n^{(j)}(z)+s(z))}\Bigg|^2\E^{\frac12}\Bigg|\frac{\varepsilon_{j4}}{z+m_n(z)+s(z)}\Bigg|^2.
\end{align}
Applying the Cauchy -- Schwartz inequality again, we get
\begin{align}\label{fin200}
 \E^{\frac12}\Bigg|&\frac{\varepsilon_{j\nu}}{(z+m_n^{(j)}(z))(z+m_n^{(j)}(z)+s(z))}\Bigg|^2\notag\\&\le 
 \E^{\frac14}\frac{|\varepsilon_{j\nu}|^4}{|z+m_n^{(j)}(z)+s(z)|^4}\E^{\frac14}\frac1{|z+m_n^{(j)}(z)|^4}.
\end{align}
Inequalities \eqref{fin100}, \eqref{fin200}, Corollaries  \ref{lem140} and  \ref{cor8} together imply
\begin{align}\label{f19}
 \Bigg|\E&\frac{\varepsilon_{j\nu}\varepsilon_{j4}
\frac1{z+m_n^{(j)}(z)}}{(z+m_n^{(j)}(z)+s(z))(z+m_n(z)+s(z))}\Bigg|
&\le \frac C{nv}\E^{\frac14}\frac{|\varepsilon_{j\nu}|^4}{|z+m_n^{(j)}(z)+s(z)|^4}.
\end{align}
By Lemmas \ref{basic1} and \ref{lem00} we have for $\nu=1$ 
\begin{align}\label{f20}
 \E^{\frac14}\frac{|\varepsilon_{j\nu}|^4}{|z+m_n^{(j)}(z)+s(z)|^4}\le \frac {C}{\sqrt n\sqrt{|z^2-4|}}.
\end{align}
By Corollary \ref{corgot}, inequality \eqref{ingot} with $\alpha=0$ and $\beta=4$ in the Appendix we have for $\nu=2,3$
\begin{align}\label{f21}
 \E^{\frac14}\frac{|\varepsilon_{j\nu}|^4}{|z+m_n^{(j)}(z)+s(z)|^4}\le \frac {C}{\sqrt {nv}|z^2-4|^{\frac14}}.
\end{align}
with some constant $C>0$ depending on $\mu_4$ and $D_0$ only.
Using that, by Lemma \ref{lemG}, for $z\in\mathbb G$,
\begin{equation}\label{forG}
 \frac1{\sqrt n\sqrt{|z^2-4|}}\le \frac {\sqrt v}{\sqrt{nv}\sqrt{|z^2-4|}}\le \frac {C}{\sqrt{nv}|z^2-4|^{\frac14}}
\end{equation}
we get from \eqref{f19}, \eqref{f20} and  \eqref{f21} that
for $z\in\mathbb G$, 
\begin{align}\label{finisch1}
|\mathfrak T_{1}|\le \frac C{(nv)^{\frac32}|z^2-4|^{\frac14}}
\end{align}
with some constant $C>0$ depending on $\mu_4$ and $D_0$ only.
\subsubsection{Estimation of $\mathfrak T_{2}$}
Using the representation \eqref{repr01}, we write
\begin{align}\notag
\mathfrak T_{2}=\frac1n\sum_{j=1}^n\E\frac{\widetilde{\varepsilon}_j^2R_{jj}}{(z+m^{(j)}(z))(z+s(z)+m_n(z))}.
\end{align}
Furthermore we note that
\begin{equation}\notag
\widetilde{\varepsilon}_j^2=\varepsilon_{j2}^2+\eta_j,
\end{equation}
where
\begin{equation}\notag
\eta_j=(\varepsilon_{j1}+\varepsilon_{j3})^2+2(\varepsilon_{j1}+\varepsilon_{j3})
\varepsilon_{j2}.
\end{equation}

We  now decompose $\mathfrak T_{2}$ as follows
\begin{equation}\label{t2ii}
\mathfrak T_{2}=\mathfrak T_{21}+\mathfrak T_{22}+\mathfrak T_{23},
\end{equation}
where
\begin{align}
\mathfrak T_{21}&=\frac1n\sum_{j=1}^n\E\frac{{\varepsilon}_{j2}^2R_{jj}}{(z+m^{(j)}(z))(z+s(z)+m_n(z))},\notag\\
\mathfrak T_{22}&=\frac1n\sum_{j=1}^n\E\frac{\eta_jR_{jj}}{(z+m^{(j)}(z))(z+s(z)+m_n^{(j)}(z))},\notag\\
\mathfrak T_{23}&=\frac1n\sum_{j=1}^n\E\frac{\eta_jR_{jj}\varepsilon_{j4}}{(z+m^{(j)}(z))(z+s(z)+m_n^{(j)}(z))(z+s(z)+m_n(z))}.\notag
\end{align}
Applying the Cauchy -- Schwartz inequality, we obtain
\begin{align}\label{lang2}
|\mathfrak T_{22}|&\le \frac1n\sum_{j=1}^n\E^{\frac12}\frac{|\eta_j|^2}{|z+m^{(j)}(z)|^2|z+s(z)+m_n^{(j)}(z)|^2}\E^{\frac12}|R_{jj}|^2.
\end{align}
In what follows we denote by $C$ a generic constant depending on  $\mu_4$ and $D_0$ only.
We note that
\begin{align}
 \E\{|\eta_j|^2\Big|\mathfrak M^{(j)}\}&\le C(\E\{|\varepsilon_{j1}|^4\Big|\mathfrak M^{(j)}\}+\E\{|\varepsilon_{j3}|^4\Big|\mathfrak M^{(j)}\})
 \notag\\&+C\Big(\E\{|\varepsilon_{j1}|^4\Big|\mathfrak M^{(j)}\}+\E\{|\varepsilon_{j3}|^4\Big|\mathfrak M^{(j)}\}\Big)^{\frac12}
 \E^{\frac12}\{|\varepsilon_{j2}|^4\Big|\mathfrak M^{(j)}\}.\notag
\end{align}
Using Lemmas \ref{basic1}, \ref{basic2}, \ref{basic5}, we get
\begin{align}\label{lang3}
\E\{|\eta_j|^2\Big|\mathfrak M^{(j)}\}&\le\frac C{n^2}+\frac C{n^2}\frac1n\sum_{l\in\mathbb T_j}|R_{ll}^{(j)}|^4\notag\\&+
(\frac C{n}+\frac C{n}(\frac1n\sum_{l\in\mathbb T_j}|R_{ll}^{(j)}|^4)^{\frac12})\frac C{nv}\im m_n^{(j)}(z).
\end{align}
This inequality and  Lemma \ref{lem00} together imply
\begin{align}\label{lang4}
\E^{\frac12}&\frac{|\eta_j|^2}{|z+m^{(j)}(z)|^2|z+s(z)+m_n^{(j)}(z)|^2}\le \frac C{n\sqrt{|z^2-4|}}\E^{\frac12}\frac1{|z+m_n^{(j)}(z)|^2}\notag\\&+ 
\frac C{n\sqrt{|z^2-4|}}\Big(\frac1n\sum_{l\in\mathbb T_j}\E|R_{ll}^{(j)}|^4\Big)^{\frac14}\E^{\frac14}\frac1{|z+m_n^{(j)}(z)|^4}\notag\\&+
\frac C{n\sqrt v|z^2-4|^{\frac14}}\Big(1+\Big(\frac1n\sum_{l\in\mathbb T_j}\E|R_{ll}^{(j)}|^4\Big)^{\frac14}\Big)\E^{\frac14}\frac1{|z+m_n^{(j)}(z)|^4}
\end{align}

Inequality \eqref{lang4} and Corollary \ref{cor8} together imply
\begin{align}\notag
 |\mathfrak T_{22}|\le \frac C{n\sqrt{|z^2-4|}}+\frac C{n\sqrt v|z^2-4|^{\frac14}}.
\end{align}
Applying Lemma \eqref{lemG} for $z\in\mathbb G$, we get
\begin{equation}\label{finisch2}
 |\mathfrak T_{22}|\le\frac C{n\sqrt v|z^2-4|^{\frac14}}\le \frac C{nv^{\frac34}}.
\end{equation}
By H\"older's inequality, we have
\begin{align}
|\mathfrak T_{23}|&\le \frac1n\sum_{j=1}^n\E^{\frac12}\frac{|\eta_j|^2}{|z+m^{(j)}(z)|^2|z+s(z)+m_n^{(j)}(z)|^2}
\notag\\&\qquad\qquad\qquad\qquad\qquad\times\E^{\frac14}\frac{|\varepsilon_{j4}|^4}{|z+s(z)+m_n(z)|^4}\E^{\frac14}|R_{jj}|^4.\notag
\end{align}
Using now Lemmas \ref{lem14}, \ref{lemG} and Corollary \ref{cor8}, we may write, for $z\in\mathbb G$,
\begin{equation}\label{finisch3}
 |\mathfrak T_{23}|\le \frac C{n\sqrt v|z^2-4|^{\frac14}}\le \frac C{nv^{\frac34}}.
\end{equation}

We continue now with $\mathfrak T_{21}$. We represent it in the form
\begin{equation}\label{t21ii}
\mathfrak T_{21}=H_1+H_2,
\end{equation}
where
\begin{align}
H_1&=-\frac1n\sum_{j=1}^n\E\frac{{\varepsilon}_{j2}^2}{(z+m^{(j)}(z))^2(z+s(z)+m_n(z))},\notag\\
H_2&=\frac1n\sum_{j=1}^n\E\frac{{\varepsilon}_{j2}^2(R_{jj}+\frac1{z+m_n^{(j)}})}{(z+m^{(j)}(z))(z+s(z)+m_n(z))}.\notag
\end{align}
Furthermore, using the representation
\begin{equation}
 R_{jj}=-\frac1{z+m_n^{(j)}(z)}+\frac1{z+m_n^{(j)}(z)}(\varepsilon_{j1}+\varepsilon_{j2}+\varepsilon_{j3})R_{jj}
\end{equation}
(compare with  \eqref{repr001}), 
we bound $H_2$ in the following way
\begin{align}\notag
|H_2|\le H_{21}+H_{22}+H_{23},
\end{align}
where
\begin{align}
H_{21}&=\frac1n\sum_{j=1}^n\E\frac{4|{\varepsilon}_{j1}|^3|R_{jj}|}{|z+m^{(j)}(z)|^2|z+s(z)+m_n(z)|},\notag\\
H_{22}&=\frac1n\sum_{j=1}^n\E\frac{2|{\varepsilon}_{j2}|^3|R_{jj}|}{|z+m^{(j)}(z)|^2|z+s(z)+m_n(z)|},\notag\\
H_{23}&=\frac1n\sum_{j=1}^n\E\frac{2|{\varepsilon}_{j3}|^3|R_{jj}|}{|z+m^{(j)}(z)|^2|z+s(z)+m_n(z)|}.\notag
\end{align}
Using inequality \eqref{lar1} in the Appendix and H\"older  inequality, we get, for $\nu=1,2,3$
\begin{align}\label{inr3}
H_{2\nu}&\le \frac1n\sum_{j=1}^n\E\frac{4|{\varepsilon}_{j\nu}|^3|R_{jj}|}{|z+m^{(j)}(z)|^2|z+s(z)+m_n^{(j)}(z)|}\notag\\&
+\frac1n\sum_{j=1}^n\E\frac{4|{\varepsilon}_{j\nu}|^3|R_{jj}||\varepsilon_{j4}|}{|z+m^{(j)}(z)|^2|z+s(z)+m_n^{(j)}(z)||z+s(z)+m_n(z)|}\notag\\&
\le \frac Cn\sum_{j=1}^n\E^{\frac34}\frac{|{\varepsilon}_{j\nu}|^4}{|z+m^{(j)}(z)|^{\frac83}|z+s(z)+m_n^{(j)}(z)|^{\frac43}}
\E^{\frac14}|R_{jj}|^4.
\end{align}
Applying  Corollary \ref{corgot} with $\beta=\frac43$ and $\alpha=\frac83$, we obtain,  for $z\in\mathbb G$, and for $\nu=1,2,3$
\begin{align}
\E^{\frac34}\frac{|{\varepsilon}_{j\nu}|^4}{|z+m^{(j)}(z)|^{\frac83}|z+s(z)+m_n^{(j)}(z)|^{\frac43}}\le
\frac C{(nv)^{\frac32}}.\notag
\end{align}
This yields together with Corollary \ref{cor8} and inequality \eqref{inr3}
\begin{equation}\label{finisch4}
 H_2\le \frac C{(nv)^{\frac32}}.
\end{equation}

Consider now $H_1$. Using the equality
\begin{equation}\notag
 \frac1{z+m_n(z)+s(z)}=\frac1{z+2s(z)}-\frac{\Lambda_n(z)}{(z+2s(z))(z+m_n(z)+s(z))}
\end{equation}
and
\begin{equation}\label{lepsj4}
 \Lambda_n=\Lambda_n^{(j)}+\varepsilon_{j4},
\end{equation}
we represent it in the form
\begin{align}\label{h1}
H_1=H_{11}+H_{12}+H_{13},
\end{align}
where
\begin{align}
H_{11}&=-\frac1{(z+s(z))^2}\frac1n\sum_{j=1}^n\E\frac{\varepsilon_{j2}^2}{z+s(z)+m_n(z)}\notag\\&
=-s^2(z)\frac1n\sum_{j=1}^n\E\frac{\varepsilon_{j2}^2}{z+s(z)+m_n(z)},\notag\\
H_{12}&=-\frac1{(z+s(z))}\frac1n\sum_{j=1}^n\E\frac{\varepsilon_{j2}^2\Lambda_n^{(j)}
}{(z+m_n^{(j)}(z))^2(z+s(z)+m_n(z))},\notag\\
H_{13}&=-\frac1{(z+s(z))^2}\frac1n\sum_{j=1}^n\E\frac{\varepsilon_{j2}^2
\Lambda_n^{(j)}
}{(z+m_n^{(j)}(z))(z+s(z)+m_n(z))}.\notag
\end{align}
In order to apply conditional independence, we write 
\begin{align}\notag
H_{11}=H_{111}+H_{112},
\end{align}
where
\begin{align}
H_{111}&=-s^2(z)\frac1n\sum_{j=1}^n\E\frac{\varepsilon_{j2}^2}{z+m_n^{(j)}(z)+s(z)},\notag\\
H_{112}&=s^2(z)\frac1n\sum_{j=1}^n\E\frac{\varepsilon_{j2}^2\varepsilon_{j4}}{(z+s(z)+m_n(z))(z+m_n^{(j)}(z)+s(z))}.\notag
\end{align}
It is straightforward to check that
\begin{align}\notag
\E\{\varepsilon_{j2}^2|\mathfrak M^{(j)}\}=\frac1{n^2}\Tr (\mathbf R^{(j)})^2-\frac1{n^2}\sum_{l\in\mathbb T_j}(R^{(j)}_{ll})^2.
\end{align}
Using equality \eqref{7.3} for $m_n'(z)$ and the corresponding relation for ${m_n^{(j)}}'(z)$, we may write
\begin{equation}\notag
H_{111}=L_1+L_2+L_3+L_4,
\end{equation}
where
\begin{align}
L_1&=-s^2(z)\frac1n\E\frac{m_n'(z)}{z+m_n(z)+s(z)},\notag\\
L_2&=s^2(z)\frac1n\sum_{j=1}^n\E\frac{\frac1{n^2}\sum_{l\in\mathbb T_j}
(R^{(j)}_{ll})^2}{z+m_n^{(j)}(z)+s(z)},\\
L_3&=s^2(z)\frac1n\sum_{j=1}^n\E\frac{\frac1{n}((m_n^{(j)}(z))'-m_n'(z))}{z+m_n^{(j)}(z)+s(z)},\\
L_4&=s^2(z)\frac1n\sum_{j=1}^n\E\frac{\frac1{n}((m_n^{(j)}(z))'-m_n'(z))\varepsilon_{j4}}
{(z+m_n(z)+s(z))(z+m_n^{(j)}(z)+s(z))}.\notag
\end{align}
Using Lemmas \ref{lem00}, \ref{basic8}, \ref{beta}, and Corollary \ref{cor8}, it is 
straightforward to check that
\begin{align}
|L_2|&\le \frac C{n\sqrt{|z^2-4|}},\notag\\
|L_3|&\le \frac C{n^2v^2\sqrt{|z^2-4|}},\notag\\
|L_4|&\le \frac C{n^3v^3{|z^2-4|}}.\notag
\end{align}
Applying inequality \eqref{lar1}, we may write
\begin{align}\notag
 |H_{12}|\le \frac Cn\sum_{j=1}^n \E\frac{|\varepsilon_{j2}|^2|\Lambda_n^{(j)}|}
 {|z+m_n^{(j)}(z)||z+m_n^{(j)}(z)+s(z)|}.
\end{align}
Conditioning on $\mathfrak M^{(j)}$ and applying Lemma \ref{basic2}, Lemma \ref{lem00}, 
inequality \eqref{lem8.5.1}, Corollary \ref{cor8} 
and equality \eqref{lepsj4}, we get
\begin{align}\notag
 |H_{12}|\le \frac C{{nv}}\frac 1n\sum_{j=1}^n \E\frac{|\Lambda_n^{(j)}|}{|z+m_n^{(j)}(z)|}\le 
 \frac C{nv}\E^{\frac12}|\Lambda_n|^2+\frac C{n^2v^2}.
\end{align}
By Lemma \ref{lam1}, we get
\begin{equation}\label{h12}
 |H_{12}|\le\frac C{n^2v^2}.
\end{equation}
Similar we get
\begin{equation}\label{h13}
 |H_{13}|\le \frac C{n^2v^2}.
\end{equation}

We rewrite now the equations \eqref{eq00}  and \eqref{eq01} as follows, using the remainder term 
$\mathfrak T_3$,
which is bounded by means of inequalities \eqref{finisch1}, \eqref{finisch2}, \eqref{finisch3}, 
\eqref{finisch4}.
\begin{align}\label{final00}
\E\Lambda_n(z)=\E m_n(z)-s(z)=\frac{(1-s^2(z))}{n}\E\frac{m_n'(z)}{z+m_n(z)+s(z)}+\mathfrak T_3,
\end{align}
where
\begin{align}\notag
|\mathfrak T_3|\le \frac C{n\sqrt v^{\frac34}}+\frac C{n^{\frac32}v^{\frac32}|z^2-4|^{\frac14}}.
\end{align}
Note that
\begin{equation}\notag
1-s^2(z)=s(z)\sqrt{z^2-4}.
\end{equation}
In \eqref{final00} we estimate now the remaining quantity 
\begin{equation}\notag
\mathfrak T_4=-\frac{s(z)\sqrt{z^2-4}}{n}\E\frac{m_n'(z)}{z+m_n(z)+s(z)}.
\end{equation}

\subsection{Estimation of $\mathfrak T_4$}\label{better}
Using that $\Lambda_n=m_n(z)-s(z)$ we rewrite $\mathfrak T_4$ as
\begin{equation}\notag
\mathfrak T_4=\mathfrak T_{41}+\mathfrak T_{42}+\mathfrak T_{43},\notag
\end{equation}
where
\begin{align}
\mathfrak T_{41}&=-\frac{s(z)s'(z)}{n},\notag\\
\mathfrak T_{42}&=\frac{s(z)\sqrt{z^2-4}}n\E\frac{m_n'(z)-s'(z)}{z+m_n(z)+s(z)},\notag\\
\mathfrak T_{43}&=\frac{s(z)}n\E\frac{(m_n'(z)-s'(z))\Lambda_n}{z+m_n(z)+s(z)}.\notag
\end{align}

\subsubsection{Estimation of  $\mathfrak T_{42}$}
First we investigate $m_n'(z)$. The following equality holds
\begin{equation}\label{mn'}
m_n'(z)=\frac1n\Tr R^2=\sum_{j=1}^n \varepsilon_{j4}R_{jj}=
s^2(z)\sum_{j=1}^n \varepsilon_{j4}R_{jj}^{-1}+D_1,
\end{equation}
where
\begin{equation}\label{d1}
D_1=\sum_{j=1}^n \varepsilon_{j4}(R_{jj}-s(z))(1+R_{jj}^{-1}s(z)).
\end{equation}
Using equality \eqref{shur}, we may write
\begin{equation}\notag
m_n'(z)=\frac{s^2(z)}n\sum_{j=1}^n (1+\frac1n\sum_{l,k\in\mathbb T_j}X_{jl}X_{jk}[(R^{(j)})^2]_{lk})+D_1.
\end{equation}
Denote by
\begin{align}\beta_{j1}&=\frac1n\sum_{l\in\mathbb T_j}[(R^{(j)})^2]_{ll}-\frac1n\sum_{l=1}^n[(R)^2]_{ll}
=\frac1n\sum_{l\in\mathbb T_j}[(R^{(j)})^2]_{ll}-m_n'(z)\notag\\&=\frac1n\frac d{dz}(\Tr \mathbf R-\Tr\mathbf R^{(j)}),\notag\\
\beta_{j2}&=\frac1n\sum_{l\in\mathbb T_j}(X^2_{jl}-1)[(R^{(j)})^2]_{ll},\notag\\
\beta_{j3}&=\frac1n\sum_{l\ne k\in\mathbb T_j}X_{jl}X_{jk}[(R^{(j)})^2]_{lk}.
\end{align}
Using  these notation we may write
\begin{align}\notag
m_n'(z)=s^2(z)(1+m_n'(z))+\frac{s^2(z)}n\sum_{j=1}^n (\beta_{j1}+\beta_{j2}+\beta_{j3})+D_1.
\end{align}
Solving this equation with respect to $m_n'(z)$ we obtain
\begin{equation}\label{semi1}
m_n'(z)=\frac{s^2(z)}{1-s^2(z)}+\frac{1}{1-s^2(z)}(D_1+D_2),
\end{equation}
where
\begin{align}\notag
D_2=\frac{s^2(z)}n\sum_{j=1}^n (\beta_{j1}+\beta_{j2}+\beta_{j3}).
\end{align}
Note that for the semi-circular law
\begin{equation}\notag
\frac{s^2(z)}{1-s^2(z)}=\frac{s^2(z)}{1+\frac{s(z)}{z+s(z)}}=-\frac{s(z)}{z+2s(z)}=
s'(z).
\end{equation}
Applying this relation we rewrite equality \eqref{semi1} as
\begin{equation}\label{mnderiv}
m_n'(z)-s'(z)=\frac{1}{s(z)(z+2s(z))}(D_1+D_2).
\end{equation}
Using the last equality, we may represent $\mathfrak T_{42}$ now as follows
\begin{equation}\notag
\mathfrak T_{42}=\mathfrak T_{421}+\mathfrak T_{422},
\end{equation}
where
\begin{align}
\mathfrak T_{421}&=\frac1n\E \frac{D_1}{z+m_n(z)+s(z)},\notag\\
\mathfrak T_{422}&=\frac1n\E \frac{D_2}{z+m_n(z)+s(z)}.\notag
\end{align}



Recall that, by \eqref{d1},
\begin{align}\mathfrak T_{421}&=-\frac1{n}\sum_{j=1}^n
\E\frac{\varepsilon_{j4}(R_{jj}-s(z))(1+R_{jj}^{-1}s(z))}{(z+s(z)+m_n(z))}.
\end{align}
Applying the Cauchy -- Schwartz inequality, we get for $z\in\mathbb G$,
\begin{equation}\notag
|\mathfrak T_{421}|\le \frac1{n}\sum_{j=1}^n\E^{\frac12}|R_{jj}-s(z)|^2
\E^{\frac14}\frac{|\varepsilon_{j4}|^4}{|z+s(z)+m_n(z)|^4}(1+|s(z)|\E^{\frac14}|R_{jj}|^{-4}).
\end{equation}
Using Corollary \ref{lem140} and Corollary \ref{cor8}, we get
\begin{align}\label{finisch7*}
|\mathfrak T_{421}|\le\frac C{n^{\frac32}v^{\frac32}}.
\end{align}
\subsubsection{Estimation of $\mathfrak T_{422}$}
We represent now $\mathfrak T_{421}$ in the form
\begin{align}\notag
\mathfrak T_{422}=\mathfrak T_{51}+\mathfrak T_{52}+\mathfrak T_{53},
\end{align}
where
\begin{align}\notag
\mathfrak T_{5\nu}&=\frac1{n^2}\sum_{j=1}^n\E\frac{\beta_{j\nu}}{z+m_n(z)+s(z)},\quad\text{for}\quad \nu=1,2,3.
\end{align}
At first we investigate $\mathfrak T_{53}$. Note that, by Lemma \ref{beta},
\begin{equation}\notag
|\beta_{j1}|\le \frac C{nv^2}.
\end{equation}
Therefore, for $z\in\mathbb G$, using Lemma \ref{lem00}, inequality \eqref{lem8.5.2},
and Lemma \ref{lemG},
\begin{equation}\label{t51}
|\mathfrak T_{51}|\le \frac C{n^2v^2\sqrt{|z^2-4|}}\le \frac C{n^{\frac32}v^{\frac32}|z^2-4|^{\frac14}}.
\end{equation}

Furthermore, we consider the quantity $\mathfrak T_{5\nu}$, for $\nu=2,3$.
Applying the Cauchy-Schwartz inequality and inequality \eqref{lar1} in the Appendix as well, 
we get
\begin{align}\notag
|\mathfrak T_{5\nu}|\le \frac C{n^2}\sum_{j=1}^n\E^{\frac12}\frac{|\beta_{j\nu}|^2}{|z+m_n^{(j)}(z)+s(z)|^2}.
\end{align}
By Lemma \ref{bet1} together with Lemma \ref{lem00} in the Appendix, we obtain
\begin{align}\notag
\E^{\frac12}\frac{|\beta_{j\nu}|^2}{|z+m_n^{(j)}(z)+s(z)|^2}\le \frac C{n^{\frac12}v^{\frac32}|z^2-4|^{\frac14}}.
\end{align}
This implies that
\begin{align}\label{finisch7}
|\mathfrak T_{5\nu}|\le \frac C{n^{\frac32}v^{\frac32}|z^2-4|^{\frac14}}.
\end{align}
Inequalities \eqref{t51} and \eqref{finisch7} yield
\begin{equation}\notag
 |\mathfrak T_{422}|\le \frac C{n^{\frac32}v^{\frac32}|z^2-4|^{\frac14}}.
\end{equation}
Combining \eqref{finisch7*} and \eqref{finisch7^}, we get, for $z\in\mathbb G$,
\begin{align}\label{finisch7^}
 |\mathfrak T_{42}|\le \frac C{n^{\frac32}v^{\frac32}|z^2-4|^{\frac14}}.
\end{align}
\subsubsection{Estimation of $\mathfrak T_{43}$} Recall that
\begin{equation}\notag
\mathfrak T_{43}=\frac{s(z)}n\E\frac{(m_n'(z)-s'(z))\Lambda_n}{z+m_n(z)+s(z)}.
\end{equation}
Applying equality \eqref{mnderiv}, we obtain
\begin{align}\notag
 \mathfrak T_{43}=\mathfrak T_{431}+\mathfrak T_{432},
\end{align}
where
\begin{align}
 \mathfrak T_{431}&=\frac{1}{n(z+2s(z))}\E\frac{D_1\Lambda_n}{z+m_n(z)+s(z)},\\
 \mathfrak T_{432}&=\frac{1}{n(z+2s(z))}\E\frac{D_2\Lambda_n}{z+m_n(z)+s(z)}.\notag
\end{align}
Applying the Cauchy -- Schwartz inequality, we get
\begin{align}\notag
|\mathfrak T_{431}|\le \frac{1}{n(z+2s(z))}\E^{\frac12}\frac{|D_1|^2}{|z+m_n(z)+s(z)|^2}\E^{\frac12}|\Lambda_n|^2.
\end{align}
By definition of $D_1$ and Lemmas \ref{lam1} and \ref{basic8}, we get
\begin{align}\notag
|\mathfrak T_{431}|
\le
\frac{C}{n^{3} v^2|z^2-4|}\frac1n\sum_{j=1}^n(1+|s(z)|\E^{\frac14}|R_{jj}|^{-4})\E^{\frac14}|R_{jj}-s(z)|^4.
\end{align}
Applying now Corollary \ref{cor8} and Lemma \ref{lem14}, we get
\begin{align}
|\mathfrak T_{431}|
 \le
\frac{4}{n^{3} v^{2}|z^2-4|^{\frac12}}.\notag
\end{align}
For $z\in\mathbb G$ this yields
\begin{equation}\notag
|\mathfrak T_{431}|
 \le
\frac{4}{n^{\frac32} v^{\frac32}|z^2-4|^{\frac14}}.
\end{equation}
Applying again the Cauchy -- Schwartz inequality, we get for $\mathfrak T_{432}$ accordingly
\begin{align}\notag
 |\mathfrak T_{432}|\le \frac C{n|z^2-4|^{\frac12}}\E^{\frac12}|D_2|^2\E^{\frac12}|\Lambda_n|^2.
\end{align}
By Lemma \ref{lam1}, we have
\begin{equation}\label{krak}
 |\mathfrak T_{432}|\le \frac C{n^2v|z^2-4|^{\frac12}}\E^{\frac12}|D_2|^2.
\end{equation}
By definition of $D_2$,
\begin{equation}\notag
 \E|D_2|^2\le \frac1n\sum_{j=1}^n(\E|\beta_{j1}|^2+\E|\beta_{j2}|^2+\E|\beta_{j3}|^2).
\end{equation}
Applying Lemmas \ref{bet1} with $\nu=2,3$, and \ref{beta}, we get
\begin{align}\label{krak1}
 \E|D_2|^2\le \frac{C}{n^2v^4}+\frac C{nv^3}.
\end{align}
Inequalities \eqref{krak} and \eqref{krak1} together imply, for $z\in\mathbb G$,
\begin{equation}\notag
 |\mathfrak T_{432}|\le \frac C{n^3v^3|z^2-4|^{\frac12}}+\frac C{n^{\frac52}v^{\frac52}|z^2-4|^{\frac12}}\le \frac C{n^{\frac32}v^{\frac32}|z^2-4|^{\frac14}}.
\end{equation}
Finally we observe that
\begin{equation}\notag
 s'(z)=-\frac{s(z)}{\sqrt{z^2-4}}
\end{equation}
and, therefore
\begin{equation}\notag
 |\mathfrak T_{41}|\le \frac C{n|z^2-4|^{\frac12}}. 
\end{equation}
For $z\in\mathbb G$ we may rewrite it 
\begin{equation}\label{final001}
 |\mathfrak T_{41}|\le \frac C{n\sqrt v}.
\end{equation}
Combining now relations \eqref{final00}, \eqref{h1}, \eqref{finisch4}, \eqref{t51}, \eqref{finisch7^}, \eqref{final001}, we get for $z\in\mathbb G$,
\begin{equation}\notag
 |\E\Lambda_n|\le \frac C{n v^{\frac34}}+\frac C{n^{\frac32}v^{\frac32}|z^2-4|^{\frac14}}.
\end{equation}
The last inequality completes the proof of Theorem \ref{stiltjesmain}.

\section{Appendix}
\subsection{Rosenthal's and Burkholder's inequalities}
In this subsection we state the Rosenthal and Burkholder inequalities, starting with Rosenthal's inequality.
Let $\xi_1,\ldots,\xi_n$ be independent random variables with $\E\xi_j=0$, $\E\xi_j^2=1$ and for $p\ge 1$ $\E|\xi_j|^p\le \mu_p$ for $j=1,\ldots,n$.
\begin{lem}\label{rosent}{\rm (Rosenthal's inequality)}

 There exists an absolute constant $C_1$ such that
 \begin{equation}\notag
 \E|\sum_{j=1}^na_j\xi_j|^p\le C_1^pp^p((\sum_{j=1}^p|a_j|^2)^{\frac p2}+\mu_p\sum_{j=1}^p|a_j|^p).
 \end{equation}
\end{lem}
\begin{proof}
 For a proof  see \cite[Theorem 3]{Rosenthal:1970} and \cite[inequality (A)]{Johnson:1985}.
\end{proof}
Let $\xi_1,\ldots\xi_n$ be a martingale-difference with respect to the $\sigma$-algebras $\mathfrak M_j=\sigma(\xi_1,\ldots,\xi_{j-1})$.
Assume that  $\E\xi_j^2=1$ and $\E|\xi_j|^p<\infty$.
\begin{lem}\label{burkh}{\rm (Burkholder's inequality)}
 There exist an absolute constant $C_2$ such that
 \begin{equation}\notag
  \E|\sum_{j=1}^n\xi_j|^p\le C_2^pp^p\Big(\E(\sum_{k=1}^n\E\{\xi_k^2|\mathfrak M_{k-1}\})^{\frac p2}+\sum_{k=1}^p\E|\xi_k|^p\Big).
 \end{equation}

\end{lem}
\begin{proof}
 For a proof of this inequality see \cite[Theorem 3.2]{Burkholder:1973} and \cite[Theorem 4.1]{Hitczenko:1990}.
\end{proof}
We rewrite the Burkholder inequality for quadratic forms in independent random variables.
Let $\zeta_1,\ldots,\zeta_n$ be independent random variables such that $\E\zeta_j=0$, $\E|\zeta_j|^2=1$ and $\E|\zeta_j|^p\le \mu_p$. 
Let $a_{ij}=a_{ji}$ for all $i,j=1,\ldots n$.
Consider the quadratic form
\begin{equation}\notag
 Q=\sum_{1\le j\ne k\le n}a_{jk}\zeta_j\zeta_k.
\end{equation}
\begin{lem}\label{burkh1}
 There exists an absolute constant $C_2$ such that
 \begin{equation}\notag
  \E|Q|^p\le C_2^p\Big(\E(\sum_{j=2}^{n}(\sum_{k=1}^{j-1}a_{jk}\zeta_k)^2)^{\frac p2}+\mu_p\sum_{j=2}^n\E|\sum_{k=1}^{j-1}a_{jk}\zeta_k|^p\Big).
 \end{equation}

\end{lem}
\begin{proof}
 Introduce the random variables, for $j=2,\ldots,n$,
 \begin{equation}\notag
  \xi_j=\zeta_j\sum_{k=1}^{j-1}a_{jk}\zeta_k.
 \end{equation}
 It is straightforward to check that
 \begin{equation}\notag
  \E\{\xi_j|\mathfrak M_{j-1}\}=0,
 \end{equation}
 and that $\xi_j$ are $\mathfrak M_{j}$ measurable.
 That means that $\xi_1,\ldots,\xi_n$ are martingale-differences.
 We may write
 \begin{equation}\notag
  Q=2\sum_{j=2}^n\xi_j.
 \end{equation}
Applying now Lemma \ref{burkh} and using
\begin{align}
 \E\{|\xi_j|^2|\mathfrak M_{j-1}\}&=(\sum_{k=1}^{j-1}a_{jk}\zeta_k)^2\E\zeta_j^2,\notag\\
 \E|\xi_j|^p&=\E|\zeta_j|^p\E|\sum_{k=1}^{j-1}a_{jk}\zeta_j|^p,
\end{align}
we get the claim.
Thus, Lemma \ref{burkh1} is proved.

\end{proof}
We prove as well the following simple Lemma
\begin{lem}\label{simple}
 Let  $t>r\ge 1$  and $a,b>0$. Any $x>0$ satisfying the inequality 
  \begin{equation}\label{w1}
  x^t\le a+bx^r
 \end{equation}
is explicitly bounded as follows
 \begin{equation}\label{w2}
  x^t\le \text{\rm e}a+\left(\frac{2t-r}{t-r}\right)^{\frac t{t-r}}b^{\frac t{t-r}}.
 \end{equation}

\end{lem}
\begin{proof}
 First assume that $x\le a^{\frac1t}$. Then inequality \eqref{w2} holds. If $x\ge a^{\frac1t}$, then according to inequality \eqref{w1}
 \begin{equation}\notag
  x^{t-r}\le a^{\frac{t-r}t}+b,
 \end{equation}
or
\begin{equation}\notag
 x^t\le (a^{\frac{t-r}t}+b)^{\frac{t}{t-r}}.
\end{equation}
Using that for any $\alpha>0$ and $a>0,b>0$ 
\begin{equation}\notag
 (a+b)^{\alpha}\le(a+\frac a{\alpha})^{\alpha}+(b+ {\alpha} b)^{\alpha} \le \text{e}a^{\alpha}+(1+\alpha)^{\alpha}b^{\alpha},
\end{equation}
we get the claim.
\end{proof}

In what follows we prove several lemmas about the resolvent matrices.
Recall the equation, for $j\in\mathbb T_{\mathbb J}$, (compare with \eqref{repr001})
\begin{equation}\notag
 R_{jj}^{(\mathbb J)}=-\frac1{z+m_n^{(\mathbb J)}(z)}+\frac1{z+m_n^{(\mathbb J)}(z)}\varepsilon_j^{(\mathbb J)}R^{(\mathbb J)}_{jj},
\end{equation}
where 
\begin{align}
\varepsilon_{j1}^{(\mathbb J)}&=-\frac{X_{jj}}{\sqrt n},\quad
\varepsilon_{j2}^{(\mathbb J)}=\frac1n\sum_{l\ne k\in\mathbb T_{\mathbb J,j}}X_{jl}X_{jk}R^{(\mathbb J,j)}_{kl},\notag\\
\varepsilon_{j3}^{(\mathbb J)}&=\frac1n\sum_{l\in\mathbb T_{\mathbb J,j}}(X_{jl}^2-1)R^{(\mathbb J,j)}_{ll},\quad
\varepsilon_{j4}^{(\mathbb J)}=m_n^{(\mathbb J)}(z)-m_n^{(\mathbb J,j)}(z).\notag
\end{align}
Summing these equations for $j\in\mathbb T_{\mathbb J}$, we get
\begin{equation}\label{main1}
m_n^{(\mathbb J)}(z)=-\frac{n-|\mathbb J|}{n(z+m_n^{(\mathbb J))}(z)}+\frac{T_n^{(\mathbb J)}}{z+m_n^{(\mathbb J)}(z)},
\end{equation}
where 
\begin{equation}\notag
 T_n^{(\mathbb J)}=\frac1n\sum_{j=1}^n\varepsilon_j^{(\mathbb J)}R_{jj}^{(\mathbb J)}.
\end{equation}
Note that
\begin{equation}\label{main2}
 \frac1{z+m_n^{(\mathbb J)}(z)}=\frac1{z+s(z)}-\frac{m_n^{(\mathbb J)}(z)-s(z)}{(s(z)+z)(z+m_n^{(\mathbb J)}(z))}=
 -s(z)+\frac{s(z)\Lambda_n^{(\mathbb J)}(z)}{z+m_n^{(\mathbb J)}(z)},
\end{equation}
where 
\begin{equation}\notag
 \Lambda_n^{(\mathbb J)}=\Lambda_n^{(\mathbb J)}(z)=m_n^{(\mathbb J)}(z)-s(z).
\end{equation}
Equalities \eqref{main1} and \eqref{main2} together imply 
\begin{equation}
 \Lambda_n^{(\mathbb J)}=-\frac{s(z)\Lambda_n^{(\mathbb J)}}{z+m_n^{(\mathbb J)}(z)}+\frac{T_n^{(\mathbb J)}}{z+m_n^{(\mathbb J)}(z)}+\frac{|\mathbb J|}{n(z+m_n^{(\mathbb J)}(z))}.
\end{equation}
Solving this with respect to $\Lambda_n^{(\mathbb J)}$, we get
\begin{equation}\label{main3}
 \Lambda_n^{(\mathbb J)}=\frac{T_n^{(\mathbb J)}}{z+m_n^{(\mathbb J)}(z)+s(z)}+\frac{|\mathbb J|}{n(z+m_n^{(\mathbb J)}(z)+s(z))}.
\end{equation}
\begin{lem}\label{lem00}Assuming the  conditions of Theorem \ref{main}, there exists an absolute constant $c_0>0$ such that for  $\mathbb J\subset \mathbb T$, 
\begin{equation}\label{lem8.5.1}
 |z+m_n^{(\mathbb J)}(z)+s(z)|\ge\im m_n^{(\mathbb J)}(z),
\end{equation}
moreover, for any $z\in\mathbb G$
\begin{equation}\label{lem8.5.2}
 |z+m_n^{(\mathbb J)}(z)+s(z)|\ge c_0\sqrt{|z^2-4|}.
\end{equation}

\end{lem}

\begin{proof}Firstly  note that $\im(z+s(z))\ge0$ and $\im m_n^{(\mathbb J)}(z)\ge 0$. Therefore
 \begin{equation}\notag
  |z+m_n^{(\mathbb J)}(z)+s(z)|\ge \im(z+m_n^{(\mathbb J)}(z)+s(z))\ge \im m_n^{(\mathbb J)}(z).
 \end{equation}
Furthermore,
 \begin{equation}\notag
 |z+m_n^{(\mathbb J)}(z)+s(z)|\ge\im(z+s(z))=\frac12\im(z+\sqrt{z^2-4})\ge \frac12\im\sqrt{z^2-4}.
 \end{equation}
Note that for $z\in\mathbb G$ 
\begin{equation}
 \re(z^2-4)<0.
\end{equation}
Therefore,
\begin{equation}\notag
 \im\sqrt{z^2-4}\ge \frac{\sqrt2}2\sqrt{|z^2-4|}.
\end{equation}
Thus Lemma \ref{lem00} is proved.
\end{proof}
\begin{lem}\label{resol00}
For any $z=u+iv$ with $v>0$ and for any $\mathbb J\subset \mathbb T$, we have
\begin{align}\label{res1}
 \frac1n\sum_{l,k\in\mathbb T_{\mathbb J}}|R^{(\mathbb J)}_{kl}|^2\le v^{-1}\im m_n^{(\mathbb J)}(z).
\end{align}
For any  $l\in\mathbb T_{\mathbb J}$
\begin{equation}\label{res2}
 \sum_{k\in\mathbb T_{\mathbb J}}|R^{(\mathbb J)}_{kl}|^2\le v^{-1}\im R^{(\mathbb J)}_{ll}.
\end{equation}
and
\begin{equation}\label{res20}
 \sum_{k\in\mathbb T_{\mathbb J}}|[(\mathbf R^{(\mathbb J)})^2]_{kl}|^2\le v^{-3}\im R^{(\mathbb J)}_{ll}.
\end{equation}
Moreover, for any $\mathbb J\subset T$ and for any $l\in\mathbb T_{\mathbb J}$ we have
\begin{align}\label{res3}
\frac1n\sum_{l\in\mathbb T_{\mathbb J}}|[(\mathbf R^{(\mathbb J)})^2]_{ll}|^2\le v^{-3}\im m_n^{(\mathbb J)}(z),
\end{align}
and, for any $p\ge1$
\begin{align}\label{res4}
\frac1n\sum_{l\in\mathbb T_{\mathbb J}}|[(\mathbf R^{(\mathbb J)})^2]_{ll}|^p\le v^{-p}\frac1n\sum_{l\in\mathbb T_{\mathbb J}}\im^pR^{(\mathbb J)}_{ll}.
\end{align}
Finally,
\begin{align}\label{res5}
\frac1n\sum_{l,k\in\mathbb T_{\mathbb J}}|[(\mathbf R^{(\mathbb J)})^2]_{lk}|^2\le v^{-3}\im m_n^{(\mathbb J)}(z),
\end{align}
and
\begin{align}\label{res6}
\frac1n\sum_{l,k\in\mathbb T_{\mathbb J}}|[(\mathbf R^{(\mathbb J)})^2]_{lk}|^{2p}\le v^{-3p}\frac1n\sum_{l\in\mathbb T_{\mathbb J}}\im^pR^{(\mathbb J)}_{ll},
\end{align}
We have as well
\begin{align}\label{res7}
\frac1{n^2}\sum_{l,k\in\mathbb T_{\mathbb J}}|[(\mathbf R^{(\mathbb J)})^2]_{lk}|^{2p}\le v^{-2p}(\frac1n\sum_{l\in\mathbb T_{\mathbb J}}\im^pR^{(\mathbb J)}_{ll})^2.
\end{align}
\end{lem}
\begin{proof}For $l\in\mathbb T_{\mathbb J}$ let us denote by $\lambda^{(\mathbb J)}_l$ for $l\in\mathbb T_{\mathbb J}$ eigenvalues of matrix $\mathbf W^{(\mathbb J)}$.
Then we may write (compare \eqref{orth1}) 
\begin{equation}
\frac1n\sum_{l,k\in\mathbb T_{\mathbb J}}|R^{(\mathbb J)}_{kl}|^2\le\frac1n\sum_{l\in\mathbb T_{\mathbb J}}\frac1{|\lambda^{(\mathbb J)}_l-z|^2}.
 \end{equation}
 Note that, for any $x\in\mathbb R^1$
 \begin{equation}\notag
  \im\frac1{x-z}=\frac{v}{|x-z|^2}.
 \end{equation}
We may write
\begin{equation}\notag
 \frac1{|\lambda^{(\mathbb J)}_l-z|^2}=v^{-1}\im\frac1{\lambda^{(\mathbb J)}_l-z}
\end{equation}
and
\begin{equation}\notag
\frac1n\sum_{l,k\in\mathbb T_{\mathbb J}}|R^{(\mathbb J)}_{kl}|^2\le v^{-1}\im(\frac1n\sum_{l\in\mathbb T_{\mathbb J}}\frac1{\lambda^{(\mathbb J)}_l-z})
=v^{-1}\im m_n^{(\mathbb J)}(z).
 \end{equation}
 So, inequality \eqref{res1} is proved.
 Let denote now by $\mathbf u_l^{(\mathbb J)}=(u^{(\mathbb J)}_{lk})_{k\in\mathbb T_{\mathbb J}}$ the eigenvector of matrix $\mathbf W^{(\mathbb J)}$ 
 corresponding to the eigenvalue $\lambda^{(\mathbb J)}_l$.
Using  this notation we may write
 \begin{equation}\label{orth1}
  R^{(\mathbb J)}_{lk}=\sum_{q\in\mathbb T_{\mathbb J}}\frac1{\lambda_q^{(\mathbb J)}-z}u^{(\mathbb J)}_{lq}u^{(\mathbb J)}_{kq}.
 \end{equation}
It is straightforward to check the following inequality
\begin{align}
 \sum_{k\in\mathbb T_{\mathbb J}}|R^{(\mathbb J)}_{kl}|^2&\le\sum_{q\in\mathbb T_{\mathbb J}}\frac1{|\lambda^{(\mathbb J)}_q-z|^2}|u^{(\mathbb J)}_{lq}|^2\notag\\&=
 v^{-1}\im\Big(\sum_{q\in\mathbb T_{\mathbb J}}\frac1{\lambda^{(\mathbb J)}_q-z}|u^{(\mathbb J)}_{lq}|^2\Big)=v^{-1}\im R_{ll}^{(\mathbb J)}.
\end{align}
Thus, inequality \eqref{res2} is proved.
Similar we get
\begin{equation}
 \sum_{k\in\mathbb T_{\mathbb J}}|[(R^{(\mathbb J)})^2]_{kl}|^2\le\sum_{q\in\mathbb T_{\mathbb J}}\frac1{|\lambda^{(\mathbb J)}_q-z|^4}|u^{(\mathbb J)}_{lq}|^2\le
 v^{-3}\im R^{(\mathbb J)}_{ll}.
\end{equation}
This proves  inequality \eqref{res20}. 
To prove  inequality \eqref{res3} we observe that
\begin{equation}\label{resol1} 
 |[(\mathbf R^{(\mathbb J)})^2]_{ll}|\le \sum_{k\in\mathbb T_{\mathbb J}}|R^{(\mathbb J)}_{lk}|^2.
\end{equation}
This inequality implies
\begin{equation}
 \frac1n\sum_{l\in\mathbb T_{\mathbb J}}|[(\mathbf R^{(\mathbb J)})^2]_{ll}|^2\le \frac1n\sum_{l\in\mathbb T_{\mathbb J}}
 (\sum_{k\in\mathbb T_{\mathbb J}}|R^{(\mathbb J)}_{lk}|^2)^2.\notag
\end{equation}
Applying now inequality \eqref{res2}, we get
\begin{align}
\frac1n\sum_{l\in\mathbb T_{\mathbb J}}|[(\mathbf R^{(\mathbb J)})^2]_{ll}|^2\le v^{-2}\frac1n\sum_{l\in\mathbb T_{\mathbb J}}\im^2R^{(\mathbb J)}_{ll}.\notag
\end{align}
This leads (using  $|R^{(\mathbb J)}_{ll}|\le v^{-1}$ ) to the following bound
\begin{align}
\frac1n\sum_{l\in\mathbb T_{\mathbb J}}|[(\mathbf R^{(\mathbb J)})^2]_{ll}|^2\le v^{-3}\frac1n\sum_{l\in\mathbb T_{\mathbb J}}\im R^{(\mathbb J)}_{ll}=
v^{-3}\im m_n^{(\mathbb J)}(z).\notag
\end{align}
Thus inequality \eqref{res3} is proved.
Furthermore, applying inequality \eqref{resol1}, we may write
\begin{align}
\frac1n\sum_{l\in\mathbb T_{\mathbb J}}|[(\mathbf R^{(\mathbb J)})^2]_{ll}|^4\le\frac1n\sum_{l\in\mathbb T_{\mathbb J}}
 (\sum_{k\in\mathbb T_{\mathbb J}}|R^{(\mathbb J)}_{lk}|^2)^4.\notag
\end{align}
Applying \eqref{res2}, this inequality yields
\begin{align}
\frac1n\sum_{l\in\mathbb T_{\mathbb J}}|[(\mathbf R^{(\mathbb J)})^2]_{ll}|^4\le v^{-4}\frac1n\sum_{l\in\mathbb T_{\mathbb J}}
 \im^4R^{(\mathbb J)}_{ll}.\notag
\end{align}
The last inequality proves  inequality \eqref{res4}.
Note that
\begin{align}
\frac1n\sum_{l,k\in\mathbb T_{\mathbb J}}|[(\mathbf R^{(\mathbb J)})^2]_{lk}|^2&\le \frac1n\Tr|\mathbf R^{(\mathbb J)}|^4=\frac1n\sum_{l\in\mathbb T_{\mathbb J}}
\frac1{|\lambda^{(\mathbb J)}_l-z|^4}\notag\\&\le v^{-3}\im\frac1n\sum_{l\in\mathbb T_{\mathbb J}}\frac1{\lambda^{(\mathbb J)}_l-z}=v^{-3}\im m_n^{(\mathbb J)}(z).\notag
\end{align}
Thus, inequality \eqref{res5} is proved.
To finish we note that
\begin{align}
 \frac1n\sum_{l,k\in\mathbb T_{\mathbb J}}|[(\mathbf R^{(\mathbb J)})^2]_{lk}|^4\le \frac1n\sum_{l\in\mathbb T_{\mathbb J}}
 (\sum_{k\in\mathbb T_{\mathbb J}}|[(R^{(\mathbb J)})^2]_{lk}|^2)^2.\notag
\end{align}
Applying inequality \eqref{res20}, we get
\begin{align}
 \frac1n\sum_{l,k\in\mathbb T_{\mathbb J}}|[(\mathbf R^{(\mathbb J)})^2]_{lk}|^4\le v^{-6}\frac1n\sum_{l\in\mathbb T_{\mathbb J}}
 (\im R^{(\mathbb J)}_{ll})^2.\notag
\end{align}
To prove inequality \eqref{res7}, we note 
\begin{equation}
|[(\mathbf R^{(\mathbb J)})^2]_{lk}|^2\le(\sum_{q\in\mathbb T_{\mathbb J}}|R^{(\mathbb J)}_{lq}|^2)(\sum_{q\in\mathbb T_{\mathbb J}}|R^{(\mathbb J)}_{kq}|^2). \notag
\end{equation}
This inequality implies
\begin{align}
 \frac1{n^2}\sum_{l,k\in\mathbb T_{\mathbb J}}|[(\mathbf R^{(\mathbb J)})^2]_{lk}|^{2p}\le 
 (\frac1n\sum_{l,k\in\mathbb T_{\mathbb J}}(\sum_{q\in\mathbb T_{\mathbb J}}|R^{(\mathbb J)}_{lq}|^2)^p)^2.\notag
\end{align}
Applying inequality \eqref{res1}, we get the claim.
Thus, Lemma \ref{resol00} is proved.
\end{proof}
\begin{lem}\label{schlein}For any $s\ge 1$, and for any $z=u+iv$ and for any $\mathbb J\subset \mathbb T$, and $j\in\mathbb T_{\mathbb J}$,
\begin{equation}\notag
 |R_{jj}^{(\mathbb J)}(u+iv/s)|\le s|R_{jj}^{(\mathbb J)}(u+iv)|.
\end{equation}

\end{lem}
\begin{proof}See   \cite[Lemma 3.4] {SchleinMaltseva:2013}.
 For the readers convenience we include the short argument here.
 Note that, for any $j\in\mathbb T_{\mathbb J}$,
 \begin{equation}
  |\frac d{dv}\log R^{(\mathbb J)}_{jj}(u+iv)|\le 
  \frac1{|R_{jj}^{(\mathbb J)}(u+iv)|}|\frac d{dv} R_{jj}^{(\mathbb J)}(u+iv)|.\notag
 \end{equation}
Furthermore,
\begin{equation}\notag
 \frac d{dv} R_{jj}^{(\mathbb J)}(u+iv)=[(\mathbb R^{(\mathbb J)})^2]_{jj}(u+iv)
\end{equation}
and
\begin{equation}\notag
 |[(\mathbf R^{(\mathbb J)})^2]_{jj}(u+iv)|\le v^{-1}\im R^{(\mathbb J)}_{jj}.
\end{equation}
From here it follows that
\begin{equation}\notag
  |\frac d{dv}\log R_{jj}^{(\mathbb J)}(u+iv)|\le v^{-1}.
 \end{equation}
 We may write now
 \begin{equation}\notag
  |\log R_{jj}^{(\mathbb J)}(u+iv)-\log R_{jj}^{(\mathbb J)}(u+iv/s)|\le \int_{v/s}^v\frac{du}u=\log s.
 \end{equation}
The last inequality yields the claim.
Thus Lemma \ref{schlein} follows.
\end{proof}
\begin{lem}\label{eps1*}
 Assuming the conditions of Theorem \ref{main}, we get 
\begin{align}\notag
 \E|\varepsilon_{j1}|^2\le \frac {C}{n}.
\end{align}
\end{lem}
\begin{proof}
 The proof follows immediately from the definition of $\varepsilon_{j1}$ and conditions of Theorem \ref{main}.
\end{proof}

\subsubsection{Some Auxiliary Bounds for Resolvent Matrices for $\im z=4$}We need the bound for the $\varepsilon_{j\nu}$, and $\eta_{j}$ for $V=4$.
\begin{lem}\label{eps2}Assuming the conditions of Theorem \ref{main}, we get
\begin{align}\notag
 \E|\varepsilon_{j2}|^2\le \frac {C}{n}.
\end{align}

\end{lem}
\begin{proof}
 Conditioning on $\mathfrak M^{(j)}$ , we get
 \begin{align}\notag
  \E|\varepsilon_{j2}|^2\le n^{-2}\sum_{k,l\in\mathbb T_j}\E|R^{(j)}_{kl}|^2.
 \end{align}
Applying now Lemma \ref{resol00}, inequality \eqref{res2}, we get 
with $\im R^{(\mathbb J)}\le\frac14$,
\begin{align}
 \E|\varepsilon_{j2}|^2\le \frac14n^{-2}\sum_{l\in\mathbb T_j}\im R_{ll}^{(j)}\le \frac1{16n}.\notag
\end{align}
Thus Lemma \ref{eps2} is proved.
\end{proof}
\begin{lem}\label{eps3}Assuming the conditions of Theorem \ref{main}, we get
\begin{align}
 \E|\varepsilon_{j3}|^2\le \frac {C}{n}.\notag
\end{align}

\end{lem}

\begin{proof}
 Conditioning on $\mathfrak M^{(j)}$ , we obtain as above
 \begin{equation}\notag
  \E|\varepsilon_{j3}|^2\le \frac{\mu_4}{n^2}\E\sum_{l\in\mathbb T_j}|R^{(j)}_{ll}|^2\le \frac{\mu_4}{16n}.
 \end{equation}
Thus Lemma \ref{eps3} is proved.
 
\end{proof}
\begin{lem}\label{etaj}Assuming the conditions of Theorem \ref{main}, we get
\begin{align}\notag
 \E|\eta_{j}|^2\le \frac {C}{n}.
\end{align}

\end{lem}
\begin{proof}
 The proof is similar to proof of Lemma \ref{eps2}. We need to use that $|[(R^{(j)})^2]_{kl}|\le V^{-2}=\frac1{16}$ and 
 $\sum_{l\in\mathbb T_j}|[(R^{(j)})^2]_{kl}|^2\le V^{-4}$.
\end{proof}

\begin{lem}\label{lem2} Assuming the conditions of Theorem \ref{main}, we get
\begin{align}\notag
 \E|\varepsilon_{j4}|^q\le \frac {C^q}{n^{ q}}.
\end{align}k,

\end{lem}
\begin{proof}
 The result follows immediately from the bound
 \begin{equation}\notag
  |\varepsilon_{j4}|\le \frac1{nv}, \text{ a. s.}
 \end{equation}
See for instance \cite{GT:2003},  Lemma 3.3.
\end{proof}

\subsection{Some auxiliary bounds for resolvent matrices for $z\in\mathbb G$}Introduce now the region
\begin{align}\label{region}
 \mathbb G &:=\{z=u+iv\in\mathbb C^+:\,u\in\mathbb J_{\varepsilon}, v\ge v_0/\sqrt{\gamma}\},\;\; \text{where} \; v_0= A_0 n^{-1},\\
  \mathbb J_{\varepsilon} &=[-2+\varepsilon,2-\varepsilon], \quad \varepsilon:=c_1n^{-\frac23}, \quad \gamma=\gamma(u)=
  \min\{2-u, 2+u\}.\notag
\end{align}
In the next lemma we state some useful inequalities for the region $\mathbb G$.
\begin{lem}\label{lemG}
 For any $z\in\mathbb G$ we have
 \begin{align}\notag
  |z^2-4|\ge 2\max\{\gamma,v\},\qquad
  {nv}\sqrt{|z^2-4|}\ge 2A_0.
 \end{align}

\end{lem}
\begin{proof}
 We observe that
 \begin{equation}\notag
  |z^2-4|=|z-2||z+2|\ge 2\sqrt{\gamma^2+v^2}.
 \end{equation}
This inequality proves the Lemma.
\end{proof}

\begin{lem}\label{basic1}Assuming the conditions of Theorem \ref{main}, there exists an absolute constant $C>0$ such that for any $j=1,\ldots,n$,
\begin{equation}
 \E\{|\varepsilon_{j1}|^4\big|\mathfrak M^{(j)}\}\le \frac {C\mu_4}{n^2}.
\end{equation}

\end{lem}
\begin{proof}
 The result follows immediately from the definition of $\varepsilon_{j1}$.
\end{proof}
\begin{lem}\label{basic2}Assuming the conditions of Theorem \ref{main}, there exists an absolute constant $C>0$ such that for any $j=1,\ldots,n$, 
 \begin{equation}\label{basic3}
  \E\{|\varepsilon_{j2}|^2\big|\mathfrak M^{(j)}\}\le \frac C{nv}\im m_n^{(j)}(z),
 \end{equation}
and
\begin{equation}\label{basic4}
 \E\{|\varepsilon_{j2}|^4\big|\mathfrak M^{(j)}\}\le \frac {C\mu_4^2}{n^2v^2}\im^2m_n^{(j)}(z).
\end{equation}

\end{lem}
\begin{proof}
 Note that r.v.'s $X_{jl}$, for $l\in\mathbb T_j$ are independent of $\mathfrak M^{(j)}$ and that for $l,k\in\mathbb T_j$ $R^{(j)}_{lk}$
 are measurable with respect to $\mathfrak M^{(j)}$.
 This implies that $\varepsilon_{j2}$ is a quadratic form with coefficients $R^{(j)}_{lk}$ independent of $X_{jl}$. 
 Thus
 its variance and fourth moment are easily available.
 \begin{equation}\notag
  \E\{|\varepsilon_{j2}|^2\big|\mathfrak M^{(j)}\}=\frac1{n^2}\sum_{l\ne k\in\mathbb T_j}|R^{(j)}_{lk}|^2\le \frac1{n^2}\Tr\mathbb |\mathbf R^{(j)}|^2.
 \end{equation}
Here we use the notation $|\mathbf A|^2=\mathbf A\mathbf A^*$ for any matrix $\mathbf A$.
Applying Lemma \ref{resol00}, inequality \eqref{res1},
we get equality \eqref{basic3}.

Furthermore, direct calculations show that
\begin{align}
 \E\{|\varepsilon_{j2}|^4\big|\mathfrak M^{(j)}\}&\le \frac C{n^2}(\frac1n\sum_{l\ne k\in\mathbb T_j}|R^{(j)}_{lk}|^2)^2
 +\frac {C\mu_4^2}{n^2}\frac1{n^2}\sum_{l\in\mathbb T_{j}}|R^{(j)}_{lk}|^4\notag\\&\le \frac {C\mu_4^2}{n^2}(\frac1n\sum_{l\ne k\in\mathbb T_j}|R^{(j)}_{lk}|^2)^2
 \le\frac{C\mu_4^2}{n^2v^2}(\im m_n^{(j)}(z))^2.\notag
\end{align}

Here again we used Lemma \ref{resol00}, inequality \eqref{res1}.
Thus Lemma \ref{basic2} is proved.
\end{proof}
\begin{lem}\label{basic5}Assuming the conditions of Theorem \ref{main}, there exists an absolute constant $C>0$ such that for any $j=1,\ldots,n$,
 \begin{align}\label{basic6}
  \E\{|\varepsilon_{j3}|^2\big|\mathfrak M^{(j)}\}\le \frac {C\mu_4}{n}\frac1n\sum_{l\in\mathbb T_j}|R^{(j)}_{ll}|^2,
 \end{align}
 and
 \begin{align}\label{basic7}
  \E\{|\varepsilon_{j3}|^4\big|\mathfrak M^{(j)}\}\le \frac {C}{n^2}(\frac1n\sum_{l\in\mathbb T_j}|R^{(j)}_{ll}|^2)^2
  +\frac {C\mu_4}{n^2}\frac1n\sum_{l\in\mathbb T_j}|R^{(j)}_{ll}|^4.
 \end{align}

\end{lem}
\begin{proof}The first inequality is easy. To prove the second, we apply
 Rosenthal's inequality. We obtain
\begin{align}
\E\{|\varepsilon_{j3}|^4\big|\mathfrak M^{(j)}\}\le \frac {C\mu_4}{n^2}(\frac1n\sum_{l\in\mathbb T_j}|R^{(j)}_{ll}|^2)^2
+\frac {C\mu_8}{n^3}\frac1n\sum_{l\in\mathbb T_j}|R^{(j)}_{ll}|^4. \notag
\end{align}
Using  $|X_{jl}|\le Cn^{\frac 14}$ we get 
$\mu_8\le Cn\mu_4$ and the claim.
Thus Lemma \ref{basic5} is proved.
\end{proof}
\begin{cor}\label{corgot}Assuming the conditions of Theorem \ref{main}, there exists an absolute constant $C>0$, 
depending on $\mu_4$ and $D_0$ only, such that for any $j=1,\ldots,n$,  $\nu=1,2,3$
$z\in\mathbb G$, and $0\le \alpha\le \frac12A_1(nv)^{\frac14}$ and $\beta\ge1$
\begin{align}
 \E\frac{|\varepsilon_{j\nu}|^2}{|z+m_n^{(j)}(z)+s(z)|^{\beta}|z+m_n^{(j)}(z)|^{\alpha}}\le \frac C{nv|z^2-4|^{\frac{\beta-1}2}}
\end{align}
and for $\beta\ge2$
 \begin{align}\label{ingot}
 \E\frac{|\varepsilon_{j\nu}|^4}{|z+m_n^{(j)}(z)+s(z)|^{\beta}|z+m_n^{(j)}(z)|^{\alpha}}\le \frac C{n^2v^2|z^2-4|^{\frac{\beta-2}2}}.
\end{align}
We have as well, for $\nu=2,3$, and $\beta\ge 4$
\begin{align}\label{ingot1} \E\frac{|\varepsilon_{j\nu}|^8}{|z+m_n^{(j)}(z)+s(z)|^{\beta}|z+m_n^{(j)}(z)|^{\alpha}}\le\frac C{n^4v^4|z^2-4|^{\frac{\beta-4}2}}
\end{align}

\end{cor}
\begin{proof}
 For $\nu=1$, by Lemma \ref{lem00}, we have
 \begin{equation}
  \E\frac{|\varepsilon_{j\nu}|^2}{|z+m_n^{(j)}(z)+s(z)||z+m_n^{(j)}(z)|^{\alpha}}\le \frac1{n|z^2-4|^{\frac{\beta}2}}\E|X_{jj}|^4\E\frac1{|z+m_n^{(j)}(z)|^{\alpha}}.\notag
 \end{equation}
Applying now Corollary \ref{cor8} and lemma \ref{lemG}, we get the claim. The proof of the second inequality for $\nu=1$ is similar.
For $\nu=2$ we apply Lemmas \ref{basic2} and \ref{lem00}, inequality \eqref{basic3} and obtain, 
using inequality \eqref{lem8.5.1},
\begin{align}
  \E&\frac{|\varepsilon_{j2}|^2}{|z+m_n^{(j)}(z)+s(z)||z+m_n^{(j)}(z)|^{\alpha}}\notag\\&
  \qquad\qquad\qquad\le 
  \frac1{nv|z^2-4|^{\frac{\beta-1}2}}\E\frac{\im m_n^{(j)}(z)}
  {|z+m_n^{(j)}(z)+s(z)||z+m_n^{(j)}(z)|^{\alpha}}
  \notag\\&\qquad\qquad\qquad\qquad\qquad\qquad\qquad\qquad\qquad
  \le \frac C{nv}\E\frac{1}{|z+m_n^{(j)}(z)|^{\alpha}}.\notag
 \end{align}
 Similar, using Lemma \ref{basic2}, inequality \eqref{basic4}, and 
 inequality \eqref{lem8.5.1} we get
 \begin{align}
  \E&\frac{|\varepsilon_{j2}|^4}{|z+m_n^{(j)}(z)+s(z)|^{\beta}|z+m_n^{(j)}(z)|^{\alpha}}\notag\\
  &\qquad\qquad\qquad\le 
  \frac C{n^2v^2|z^2-4|^{\frac{\beta-2}2}}\E\frac{\im^2 m_n^{(j)}(z)}{|z+m_n^{(j)}(z)+s(z)|^2
  |z+m_n^{(j)}(z)|^{\alpha}}\notag\\&
  \qquad\qquad\qquad\qquad\qquad\qquad\qquad\le 
  \frac C{n^2v^2|z^2-4|^{\frac{\beta-2}2}}\E\frac{1}{|z+m_n^{(j)}(z)|^{\alpha}}.\notag
  \end{align}
 Applying Corollary \ref{cor8}, we get the claim.
 For $\nu=3$, we apply Lemma \ref{basic5}, inequalities \eqref{basic6} and \eqref{basic7} and Lemma \ref{lem00}. We get
 \begin{align}
  \E&\frac{|\varepsilon_{j3}|^2}{|z+m_n^{(j)}(z)+s(z)||z+m_n^{(j)}(z)|^{\alpha}}\notag\\
  &\qquad\qquad\qquad\le 
  \frac C{n\sqrt{|z^2-4}}\E\frac{1}{|z+m_n^{(j)}(z)|^{\alpha}}\Big(\frac1n\sum_{l\in\mathbb T_j}|R^{(j)}_{ll}|^2\Big),
 \end{align}
and
\begin{align}
  \E&\frac{|\varepsilon_{j3}|^4}{|z+m_n^{(j)}(z)+s(z)|^2|z+m_n^{(j)}(z)|^{\alpha}}\notag\\
  &\qquad\qquad\qquad\le 
  \frac C{n^2{|z^2-4|}}\E\frac{1}{|z+m_n^{(j)}(z)|^{\alpha}}\Big(\frac1n\sum_{l\in\mathbb T_j}|R^{(j)}_{ll}|^4\Big).\notag
 \end{align}
 Using now the Cauchy -- Schwartz inequality and Corollary \ref{cor8}, we get the claim.
 \end{proof}

\begin{lem}\label{basic8}Assuming the conditions of Theorem \ref{main}, there exists an absolute constant $C>0$ such that for any $j=1,\ldots,n$,
 \begin{align}
  |\varepsilon_{j4}|\le \frac C{nv}\quad\text{a.s.}\notag
 \end{align}
 \end{lem}
\begin{proof}This inequality follows from 
\begin{equation}\label{shur}
 \Tr \mathbf R-\Tr\mathbf R^{(j)}=(1+\frac1n\sum_{l,k\in\mathbb T_j}X_{jl}X_{jk}[(R^{(j)})^2]_{kl})R_{jj}=R^{-1}_{jj}\frac {dR_{jj}}{dz},
\end{equation}
which may be obtained using the Schur complement formula.
See, for instance  \cite{GT:2003}, Lemma  3.3.
 
\end{proof}
\begin{cor}\label{forgot}
 Assuming the conditions of Theorem \ref{main}, there exists an absolute constant $C>0$ such 
 that for any $j=1,\ldots,n$,
 \begin{align}\label{8.79}
 \E |\varepsilon_{j}|^4\le \frac C{n^2v^2}.
 \end{align}
\end{cor}
\begin{proof}
 By definition of $\varepsilon_j$ (see \ref{repr001}), we have
 \begin{align}\notag
  \E|\varepsilon_j|^4\le 4^3(\E|\varepsilon_{j1}|^4+\cdots+\E|\varepsilon_{j4}|^4).
 \end{align}
Applying now Lemmas \ref{basic1}, \ref{basic2} (inequality \eqref{basic4}), 
\ref{basic5} (inequality \eqref{basic7}) and \ref{basic8}
and taking expectation where  needed \eqref{8.79} follows.
Thus the Corollary is proved.
\end{proof}

Introduce quantities
\begin{align}\label{shur2}
\eta_{j1}&=\frac1n\sum_{l\in\mathbb T_j}
[(R^{(j)})^2]_{ll}=\frac1n\Tr(\mathbf  R^{(j)})^2,\notag\\
\eta_{j2}&=\frac1n\sum_{l\ne k\in\mathbb T_j}
X_{jl}X_{jk}[(\mathbf R^{(j)})^2]_{lk},\notag\\
\eta_{j3}&=\frac1n\sum_{l\in\mathbb T_j}
(X_{jl}^2-1)[(\mathbf R^{(j)})^2]_{ll}.
\end{align}
\begin{lem}\label{deriv1}Assuming the conditions of Theorem \ref{main}, we have, for  $z\in\mathbb G$, for any $j=1,\ldots,n$ 
\begin{align}
\E\{|\eta_{j3}|^4\big|\mathfrak M^{(j)}\}&\le Cn^{-2}v^{-6}\im^2m_n^{(j)}(z)
+Cn^{-2}v^{-4}\mu_4\frac1n\sum_{l\in\mathbb T_j}\im^4R^{(j)}_{ll},\label{var1}\\
\E\{|\eta_{j3}|^8\big|\mathfrak M^{(j)}\}&\le\frac {C\mu_4^4}{n^4v^{12}}\im^4m_n^{(j)}(z)
+\frac {C\mu_4^4}{n^4v^{8}}\frac1n\sum_{l\in\mathbb T_j}\im^8R^{(j)}_{ll}.\label{var2}
\end{align}

\end{lem}
\begin{proof}
Direct calculation shows
\begin{align}
 \E\{|\eta_{j3}|^4\big|\mathfrak M^{(j)}\}&\le\frac {C\mu_4}{n^2}(\sum_{l\in\mathbb T_j}|[(R^{(j)})^2]_{ll}|^2)^2+
 \frac {C\mu_8}{n^3}\frac1n\sum_{l\in\mathbb T_j}|[(R^{(j)})^2]_{ll}|^4.\notag
\end{align}
Applying Lemma \ref{resol00}, the inequality \eqref{res4} and the inequality $\mu_8\le Cn\mu_4$, we get
\begin{align}
 \E\{|\eta_{j3}|^4\big|\mathfrak M^{(j)}\}\le Cn^{-2}v^{-6}\im^2m_n^{(j)}(z)+Cn^{-2}\mu_4\frac1n\sum_{l\in\mathbb T_j}\im^4R^{(j)}_{ll}.\notag
 \end{align}
 Thus, inequality \eqref{var1} is proved.
 Consider now the $8$th moment of $\eta_{j3}$. Applying Rosenthal's inequality, we get
 \begin{align}
 \E\{|\eta_{j3}|^8\big|\mathfrak M^{(j)}\}&\le \frac {C\mu_4^4}{n^4}(\frac1n\sum_{l\in\mathbb T_j}|[(R^{(j)})^2]_{ll}|^2)^4+
 \frac {C\mu_{16}}{n^7}\frac1n\sum_{l\in\mathbb T_j}|[(R^{(j)})^2]_{ll}|^8.\notag
 \end{align}
Using now Lemma \ref{resol00}, inequalities \eqref{res3} and \eqref{res4}, and that $\mu_{16}\le Cn^3\mu_4$, we obtain
\begin{align}
\E\{|\eta_{j3}|^8\big|\mathfrak M^{(j)}\} \le \frac {C\mu_4^4}{n^4v^{12}}\im^4m_n^{(j)}(z)+\frac {C\mu_4^4}{n^4v^{8}}\frac1n\sum_{l\in\mathbb T_j}\im^8R^{(j)}_{ll}.\notag
\end{align}
Thus, inequality \eqref{var2} is proved.
\end{proof}
\begin{lem}\label{deriv2}Assuming the conditions of Theorem \ref{main}, we have, for  $z\in\mathbb G$, for any $j=1,\ldots,n$ 
\begin{align}
\E\{|\eta_{j2}|^4\big|\mathfrak M^{(j)}\}&\le \frac {C\im^2m_n^{(j)}(z)}{n^2v^{6}}+
\frac {C\mu_4^2}{n^3v^6}\frac1n\sum_{l\in\mathbb T_j}|R^{(j)}_{ll}|^2,\label{var3}\\
\E\{|\eta_{j2}|^8\big|\mathfrak M^{(j)}\}&\le\frac C{n^4v^{12}}\big(\im m_n^{(j)}(z)\big)^4+
 \frac {C\mu_4^2}{n^4v^{8}}\Big(\frac1n\sum_{l\in\mathbb T_j}\im ^4R^{(j)}_{ll}\Big)^2\notag\\&+
 \frac {C\mu_4^2}{n^4v^{8}}\Big(\frac1n\sum_{l\in\mathbb T_j}\im^3 R^{(j)}_{ll}\Big)^2
 \big(\im m_n^{(j)}(z)\big)+
 \frac {C\mu_4^4}{n^6v^{12}}\Big(\frac1n\sum_{l\in\mathbb T_j}\big(\im R^{(j)}_{ll}\big)^2\Big)^2
 \notag\\&+
 \frac {C\mu_4^2}{n^6v^{12}}\Big(\frac1n\sum_{l\in\mathbb T_j}\big(\im R^{(j)}_{ll}\big)^2\Big)
 \big(\im m_n^{(j)}(z)\big)^2.\label{var4}
\end{align}

\end{lem}
\begin{proof}
 Direct calculation shows that
 \begin{align}
\E\{|\eta_{j2}|^4\big|\mathfrak M^{(j)}\}&\le \frac C{n^2}(\frac1{n}
\sum_{l\ne k\in\mathbb T_j}|[(R^{(j)})^2]_{kl}|^2)^2+
\frac {C\mu_4^2}{n^4}\sum_{l\ne k\in\mathbb T_j}|[(R^{(j)})^2]_{kl}|^4.\notag
\end{align}
Applying inequalities \eqref{res5} and \eqref{res6} of Lemma \ref{resol00}, we get
\begin{align}
\E\{|\eta_{j2}|^4\big|\mathfrak M^{(j)}\}&\le \frac {C\im^2 m_n^{(j)}(z)}{n^2v^{6}}+
\frac {C\mu_4^2}{n^3v^6}\frac1n\sum_{l\in\mathbb T_j}|R^{(j)}_{ll}|^2.\notag
\end{align}
Furthermore, we have
\begin{align}
 \E\{|\eta_{j2}|^8\big|\mathfrak M^{(j)}\}&\le\frac C{n^4}\Big(\frac1{n}
 \sum_{l\ne k\in\mathbb T_j}|[(R^{(j)})^2]_{kl}|^2\Big)^4+
 \frac {C\mu_8^2}{n^8}\sum_{l\ne k\in\mathbb T_j}|[(R^{(j)})^2]_{kl}|^8\notag\\&+
 \frac {C\mu_6^2}{n^6}\Big(\frac1n\sum_{l\ne k\in\mathbb T_j}|[(R^{(j)})^2]_{kl}|^6\Big)
 \Big(\frac1{n}\sum_{l\ne k\in\mathbb T_j}|[(R^{(j)})^2]_{kl}|^2\Big)\notag\\&+
 \frac {C\mu_4^4}{n^6}\Big(\frac1n\sum_{l\ne k\in\mathbb T_j}|[(R^{(j)})^2]_{kl}|^4\Big)^2
 \notag\\&+
 \frac {C\mu_4^2}{n^6}\Big(\frac1n\sum_{l\ne k\in\mathbb T_j}|[(R^{(j)})^2]_{kl}|^4\Big)
 \Big(\frac1{n}\sum_{l\ne k\in\mathbb T_j}|[(R^{(j)})^2]_{kl}|^2\Big)^2.\notag
\end{align}
Note that
\begin{equation}
 \mu_6\le C\sqrt n\mu_4,\quad \mu_8\le Cn\mu_4.\notag
\end{equation}
Using this relations, we get
\begin{align}
 \E\{|\eta_{j2}|^8\big|\mathfrak M^{(j)}\}&\le\frac C{n^4}\Big(\frac1{n}
 \sum_{l\ne k\in\mathbb T_j}|[(R^{(j)})^2]_{kl}|^2\Big)^4+
 \frac {C\mu_4^2}{n^6}\sum_{l\ne k\in\mathbb T_j}|[(R^{(j)})^2]_{kl}\Big|^8\notag\\&+
 \frac {C\mu_4^2}{n^5}\Big(\frac1n\sum_{l\ne k\in\mathbb T_j}|[(R^{(j)})^2]_{kl}|^6\Big)
 \Big(\frac1{n}\sum_{l\ne k\in\mathbb T_j}|[(R^{(j)})^2]_{kl}|^2\Big)\notag\\&+
 \frac {C\mu_4^4}{n^6}\Big(\frac1n\sum_{l\ne k\in\mathbb T_j}|[(R^{(j)})^2]_{kl}|^4\Big)^2
 \notag\\&+
 \frac {C\mu_4^2}{n^6}\Big(\frac1n\sum_{l\ne k\in\mathbb T_j}|[(R^{(j)})^2]_{kl}|^4\Big)
 \Big(\frac1{n}\sum_{l\ne k\in\mathbb T_j}|[(R^{(j)})^2]_{kl}|^2\Big)^2.\notag
\end{align}
Applying now Lemma \ref{resol00}, we obtain
\begin{align}
 \E\{|\eta_{j2}|^8\big|\mathfrak M^{(j)}\}&\le\frac C{n^4v^{12}}(\im m_n^{(j)}(z))^4+
 \frac {C\mu_4^2}{n^4v^{8}}\Big(\frac1n\sum_{l\in\mathbb T_j}(\im R^{(j)}_{ll})^4\Big)^2
 \notag\\+
 \frac {C\mu_4^2}{n^5v^{12}}&\Big(\frac1n\sum_{l\in\mathbb T_j}\im ^3R^{(j)}_{ll}\Big)
 (\im m_n^{(j)}(z))+
 \frac {C\mu_4^4}{n^6v^{12}}\Big(\frac1n\sum_{l\in\mathbb T_j}(\im R^{(j)}_{ll})^2\Big)^2
 \notag\\&+
 \frac {C\mu_4^2}{n^6}\Big(\frac1n\sum_{l\in\mathbb T_j}(\im R^{(j)}_{ll})^2\Big)
 (\im m_n^{(j)}(z))^2.\notag
\end{align}
Thus, the Lemma is proved.
\end{proof}
\begin{lem}\label{lem14}Assuming the conditions of Theorem \ref{main}, we have
\begin{align}
\E\frac{|\varepsilon_{j4}|^4}{|z+s(z)+m_n(z)|^4}&\le\frac C{n^4v^4}. \notag
\end{align}
\end{lem}
\begin{proof}Using the representations \eqref{shur}-\eqref{shur2} we have 
\begin{equation}\label{shu}
  \varepsilon_{j4}=\frac1n(1+\frac1n\sum_{l, k\in\mathbb T_j}X_{jl}X_{jk}[(R^{(j)})^2]_{lk})R_{jj}=\frac1n(1+\eta_{j1} +\eta_{j2}+\eta_{j3}) R_{jj}.
  \end{equation}
 Applying the Cauchy--Schwartz inequality, we get
\begin{align}
\E|&\frac{\varepsilon_{j4}}{z+s(z)+m_n(z)}|^4\notag\\&\qquad\le \frac C{n^4}\Bigg(1+\E^{\frac12}\Bigg(\frac{|\frac1n\sum_{l,k\in\mathbb T_j}
X_{jl}X_{jk}[(R^{(j)})^2]_{lk}|}{|z+s(z)+m_n(z)|}\Bigg)^8\Bigg)\E^{\frac12}|R_{jj}|^8.\notag
\end{align}

Using Corollary \ref{cor8},  we we may write
\begin{align}
\E|&\frac{\varepsilon_{j4}}{z+s(z)+m_n(z)}|^4\notag\\&\qquad\qquad\le\frac C{n^4}\Big(1+\E^{\frac12}\Big|\frac{\eta_{j1}}{z+m_n(z)+s(z)}\Big|^8
+\E^{\frac12}\Big|\frac{\eta_{j2}}{z+m_n(z)+s(z)}\Big|^8\notag\\&\qquad\qquad\qquad\qquad\qquad\qquad\qquad+
\E^{\frac12}\Big|\frac{\eta_{j3}}{z+m_n(z)+s(z)}\Big|^8\Big).\notag
\end{align}
Observe that, 
\begin{align}
\frac1{|z+s(z)+m_n(z)|}&\le\frac1{|z+s(z)+m_n^{(j)}(z)|}(1+
\frac{|\varepsilon_{j4}|}{|z+s(z)+m_n(z)|} ).\notag
\end{align}
Therefore, by Lemmas \ref{basic8}, \ref{lem00} and  \ref{lemG}, for $z\in\mathbb G$,
\begin{align}\label{lar1}
\frac1{|z+s(z)+m_n(z)|}&\le\frac C{|z+s(z)+m_n^{(j)}(z)|}.
\end{align}
Using \eqref{lar1} we get  by definition of $\eta_{j1}$, Lemma \ref{resol00}, and  inequality \eqref{res1}
\begin{align}
\E^{\frac12}\Big|\frac{\eta_{j1}}{z+m_n(z)+s(z)}\Big|^8\le C\E^{\frac12}\frac{v^{-8}\im^8 m_n^{(j)}(z)}{|z+m_n^{(j)}(z)+s(z)|^8}\le v^{-4}.\notag
\end{align}
Furthermore, applying inequality \eqref{lar1} again, we obtain
\begin{align}
\E^{\frac12}\Big|\frac{\eta_{j2}}{z+m_n(z)+s(z)}\Big|^8&\le\E^{\frac12}\Big|\frac{\eta_{j2}}{z+m_n^{(j)}(z)+s(z)}\Big|^8.\notag
\end{align}
Conditioning with respect to $\mathfrak M^{(j)}$ and applying Lemma \ref{deriv1}, we obtain
\begin{align}
\E^{\frac12}\Big|\frac{\eta_{j2}}{z+m_n(z)+s(z)}\Big|^8&\le\E^{\frac12}\frac C{|z+m_n^{(j)}(z)+s(z)|^8}\Big(\frac1{n^4v^{12}}(\im m_n^{(j)}(z))^4\notag\\&
\quad\quad\quad\quad\quad\quad\quad\quad+
\frac{\mu_4^4}{n^4v^8}\frac 1n\sum_{l\in\mathbb T_{\mathbb J}}(\im R^{(j)}_{ll})^8\Big).\notag
\end{align}
Using  Lemma \ref{lem00}, inequality \eqref{lem8.5.1},
together with Corollary \ref{cor8} we get
\begin{align}
\E^{\frac12}\Big|\frac{\eta_{j2}}{z+m_n(z)+s(z)}\Big|^8&\le\frac C{n^2v^6|z^2-4|}+\frac{C\mu_{4}^2}{n^{2}v^4|z^2-4|^2}.\notag
\end{align}

Applying inequality \eqref{lar1} and conditioning with respect to $\mathfrak M^{(j)}$ and applying Lemma \ref{deriv2}, we get
\begin{align}\label{ka0}
\E^{\frac12}\Big|&\frac{\eta_{j3}}{z+m_n(z)+s(z)}\Big|^8\le \E^{\frac12}\frac1{|z+s(z)+m_n^{(j)}(z)|^8}\Bigg(\frac C{n^4v^{12}}(\im m_n^{(j)}(z))^4\notag\\&+
 \frac {C\mu_4^2}{n^4v^{8}}\Big(\frac1n\sum_{l\in\mathbb T_j}(\im R^{(j)}_{ll})^4\Big)^2+
 \frac {C\mu_4^2}{n^5v^{12}}\Big(\frac1n\sum_{l\in\mathbb T_j}(\im R^{(j)}_{ll})^3\Big)(\im m_n^{(j)}(z))\notag\\&+
 \frac {C\mu_4^4}{n^6v^{12}}\Big(\frac1n\sum_{l\in\mathbb T_j}(\im R^{(j)}_{ll})^2\Big)^2+
 \frac {C\mu_4^2}{n^6v^{12}}\Big(\frac1n\sum_{l\in\mathbb T_j}(\im R^{(j)}_{ll})^2\Big)(\im m_n^{(j)}(z))^2
\Bigg).
\end{align}
Using that $|z+m_n^{(j)}(z)+s(z)|\ge \im m_n^{(j)}(z)$ together with  Lemma \ref{resol00}, we arrive at
\begin{align}
\E^{\frac12}\Big|\frac{\eta_{j3}}{z+m_n(z)+s(z)}\Big|^8&\le \frac C{n^{2}v^{6}|z^2-4|}+\frac C{n^{2}v^{4}|z^2-4|^2}
\notag\\&+
\frac C{n^{\frac52}v^{6}|z^2-4|^{\frac74}}
+
\frac C{n^{3}v^{6}|z^2-4|^{2}}+
\frac C{n^{3}v^{6}|z^2-4|^{\frac32}}.\notag
\end{align}
Summarizing we may write now, for $z\in\mathbb G$,
\begin{align}
\E\frac{|\varepsilon_{j4}|^4}{|z+s(z)+m_n(z)|^4}&\le\frac C{n^4v^4}+\frac C{n^{6}v^{6}|z^2-4|}+\frac C{n^{6}v^{4}|z^2-4|^2}
\notag\\&+
\frac C{n^{\frac{13}2}v^{6}|z^2-4|^{\frac74}}
+
\frac C{n^{7}v^{6}|z^2-4|^{2}}+
\frac C{n^{7}v^{6}|z^2-4|^{\frac32}}.\notag
\end{align}
For $z\in\mathbb G$,   see \eqref{region} and Lemma \ref{lemG}, this inequality may be simplified by means of the following bounds (with $v_0= A_0n^{-1}$)
\begin{align}\label{lowerbound}
 n^{\frac52}v^2|z^2-4|^{\frac74}\ge n^{\frac52}v_0^2\gamma^{-1+\frac74}\ge C\sqrt n\gamma^{\frac34}\ge C,\notag\\
 n^3v^2|z^2-4|\ge C,\quad n^3v^2|z^2-4|^{\frac32}\ge C,\notag\\
 n^2|z^2-4|^2\ge C.
\end{align}

Using these relation,  we obtain
\begin{align}\notag
\E\frac{|\varepsilon_{j4}|^4}{|z+s(z)+m_n(z)|^4}&\le\frac C{n^4v^4}.
\end{align}
Thus Lemma \ref{lem14} is proved.
\end{proof}
\begin{cor}\label{lem140}Assuming the  conditions of Theorem \ref{main}, we have
\begin{align}\notag
\E\frac{|\varepsilon_{j4}|^2}{|z+s(z)+m_n(z)|^2}&\le\frac C{n^2v^{2}}.
\end{align}

\end{cor}
\begin{proof}
 The result follows immediately from Lemma \ref{lem14} and Jensen's inequality.
\end{proof}
\begin{lem}\label{lam1}Assuming the conditions of Theorem \ref{main}, we have, for  $z\in\mathbb G$,
\begin{align}
\E|\Lambda_n|^2\le \frac C{n^2v^2}.
\end{align}
\end{lem}
\begin{proof}We write
\begin{align}\notag
 \E|\Lambda_n|^2=\E\Lambda_n\overline \Lambda_n=\E\frac{T_n}{z+m_n(z)+s(z)}\overline\Lambda_n=\sum_{\nu=1}^4\E\frac{T_{n\nu}}{z+m_n(z)+s(z)}\overline\Lambda_n,
\end{align}
where
\begin{align}\notag
 T_{n\nu}:=\frac1n\sum_{j=1}^n\varepsilon_{j\nu}R_{jj}, \text{ for }\nu=1,\ldots,4.
\end{align}
First we observe that by \eqref{shur}
\begin{align}\notag
 |T_{n4}|=\frac1n|m_n'(z)|\le \frac1{nv}\im m_n(z).
\end{align}
 Hence $|z+m_n^{(j)}(z)+s(z)|\ge \im m_n^{(j)}(z)$ and Jensen's inequality yields
\begin{equation}\label{tn4}
 |\E\frac{T_{n4}}{z+m_n(z)+s(z)}\overline\Lambda_n|\le \frac 1{nv}\E^{\frac12}|\Lambda_n|^2.
\end{equation}
Furthermore, we represent $T_{n1}$ as follows
\begin{equation}\notag
 T_{n1}=T_{n11}+T_{n12},
\end{equation}
where
\begin{align}
 T_{n11}&=-\frac1n\sum_{j=1}^n\varepsilon_{j1}\frac1{z+m_n(z)},\notag\\
 T_{n12}&=\frac1n\sum_{j=1}^n\varepsilon_{j1}(R_{jj}+\frac1{z+m_n(z)}).\notag
\end{align}
Using these notations we may write
\begin{align}
 V_1:=\E\frac{T_{n11}}{z+m_n(z)+s(z)}\overline\Lambda_n=-\E\frac{(\frac1n\sum_{j=1}^n\varepsilon_{j1})}{(z+m_n(z))(z+s(z)+m_n(z))}\overline \Lambda_n.\notag
\end{align}
Applying the Cauchy -- Schwartz inequality twice and using the definition of $\varepsilon_{j1}$ (see \eqref{repr001}), we get by Lemma \ref{lem00}
\begin{align}
 |V_1|\le \frac1{|z^2-4|^{\frac12}}\E^{\frac14}\Big|\frac1{n\sqrt n}\sum_{j=1}^nX_{jj}\Big|^4\E^{\frac14}|\frac1{z+m_n(z)|^4}\E^{\frac12}|\Lambda_n|^2.
\end{align}
By Rosenthal's inequality, we have, for $z\in\mathbb G$
\begin{equation}\label{v1}
 |V_1|\le \frac C{n|z^2-4|^{\frac12}}\E^{\frac12}|\Lambda_n|^2\le \frac1{nv}\E^{\frac12}|\Lambda_n|^2.\notag
\end{equation}
Using \eqref{repr001} we  rewrite $T_{n12}$, obtaining
\begin{align}\notag
 V_2:=\E\frac{T_{n12}}{z+m_n(z)+s(z)}\overline\Lambda_n=\frac1{n\sqrt n}\sum_{j=1}^n\E\frac{X_{jj}\varepsilon_jR_{jj}}{(z+m_n(z))(z+m_n(z)+s(z))}\overline\Lambda_n.
\end{align}
By the Cauchy -- Schwartz inequality, 
using the definition of $\varepsilon_j$ (see representation \ref{repr001}), we obtain
\begin{align}\label{v3}
 |V_2|\le \frac1{\sqrt n}\sum_{\nu=1}^4\E^{\frac12}\frac{|\frac1n\sum_{j=1}^n\varepsilon_{j\nu}X_{jj}R_{jj}|^2}{|z+m_n(z)+s(z)|^2|z+m_n(z)|^2}
 \E^{\frac12}|\Lambda_n|^2=:\sum_{\nu=1}^4V_{2\nu}.
\end{align}
For $\nu=1$, we have
\begin{align}
 V_{21}\le\frac1{ n}\E^{\frac12}\frac{\Big|\frac1n\sum_{j=1}^nX_{jj}^2R_{jj}\Big|^2}{|z+m_n(z)+s(z)|^2|z+m_n(z)|^2}\E^{\frac12}|\Lambda_n|^2.\notag
\end{align}
Applying the Cauchy -- Schwartz inequality  twice and Lemma \ref{lem00}, we arrive at
\begin{align}\label{v4}
 V_{21}\le \frac1{ n\sqrt{|z^2-4|}}\E^{\frac14}\Big(\frac1n\sum_{j=1}^nX_{jj}^4\Big)^2\E^{\frac18}\Big(\frac1n\sum_{j=1}^n|R_{jj}|^2\Big)^4
 \E^{\frac18}\frac1{|z+m_n(z)|^8}\E^{\frac12}|\Lambda_n|^2.
\end{align}
Observe that
\begin{align}\label{x4}
 \E\Big(\frac1n\sum_{j=1}^nX_{jj}^4\Big)^2&=\Big(\frac1n\sum_{j=1}^n\E X_{jj}^4\Big)^2+\E\Big(\frac1n\sum_{j=1}^n(X_{jj}^4-\E X_{jj}^4)\Big)^2\notag\\&\le 
 \mu_4^2+\frac2{n^2}\sum_{j=1}^n\E|X_{jj}|^8\le (\mu_4 +D_0^4)\mu_4\le C. 
\end{align}
The last inequality, inequality \eqref{v4}, Corollary \ref{cor8} and Lemma \ref{lemG} together imply
\begin{equation}\label{v6}
 V_{21}\le \frac C{n\sqrt{|z^2-4|}}\E^{\frac12}|\Lambda_n|^2\le \frac C{nv}\E^{\frac12}|\Lambda_n|^2.
\end{equation}
Furthermore, for $\nu=4$, by Lemma \ref{basic8} we have
\begin{align}
 V_{24}\le \frac1{nv\sqrt n}\E^{\frac12}\frac{\frac1n\sum_{j=1}^n|X_{jj}|^2|R_{jj}|^2}{|z+m_n(z)+s(z)|^2|z+m_n(z)|^2}
 \E^{\frac12}|\Lambda_n|^2.\notag
\end{align}
Applying the Cauchy -- Schwartz inequality and Lemma \ref{lem00}, we get
\begin{align}
 V_{24}\le\frac1{nv\sqrt n\sqrt{|z^2-4|}}\E^{\frac14}\Big(\frac1n\sum_{j=1}^nX_{jj}^4\Big)^2\E^{\frac18}\Big(\frac1n\sum_{j=1}^n|R_{jj}|^2\Big)^4
 \E^{\frac18}\frac1{|z+m_n(z)|^8}\E^{\frac12}|\Lambda_n|^2.\notag
\end{align}
Similar to inequality \eqref{v6}, applying Lemma \ref{lemG}, inequality \eqref{x4} and Corollary \ref{cor8},  we get
\begin{equation}\label{v7}
 V_{24}\le  \frac C{nv}\E^{\frac12}|\Lambda_n|^2.
\end{equation}
By H\"older's inequality, we have for $\nu=2,3$,
\begin{equation}
 V_{2\nu}\le \frac1{\sqrt n}\frac1n\sum_{j=1}^n\E^{\frac14}\frac{|\varepsilon_{j\nu}|^4|X_{jj}|^4}{|z+m_n(z)+s(z)|^4}\E^{\frac18}\frac1{|z+m_n(z)|^8}
 \E^{\frac18}|R_{jj}|^8
 \E^{\frac12}|\Lambda_n|^2.\notag
\end{equation}
Note that for $\nu=2,3$, r.v. $X_{jj}$ doesn't depend on $\varepsilon_{j\nu}$ and on $\sigma$-algebra  $\mathfrak M^{(j)}$.
Using inequality \eqref{lar1} for $z\in\mathbb G$, we get
\begin{align}
 \E\frac{|\varepsilon_{j\nu}|^4|X_{jj}|^4}{|z+m_n(z)+s(z)|^4}\le C\E\frac{|\varepsilon_{j\nu}|^4|X_{jj}|^4}{|z+m_n^{(j)}(z)+s(z)|^4}\le 
 C\mu_4\E\frac{|\varepsilon_{j\nu}|^4}{|z+m_n^{(j)}(z)+s(z)|^4}.\notag
\end{align}
Applying now Lemmas \ref{basic2}, \ref{basic5} and \ref{lemG}, arrive at
\begin{equation}\label{v8}
 V_{2\nu}\le  \frac C{nv}\E^{\frac12}|\Lambda_n|^2,\text{ for }\nu=2,3.
\end{equation}
Inequalities \eqref{v6}, \eqref{v7}, \eqref{v8} together imply
\begin{equation}\label{v9}
 V_2\le \frac C{nv}\E|\Lambda_n|^2.
\end{equation}

Consider now the quantity
\begin{equation}
 Y_{\nu}:=\E\frac{T_{n\nu}}{z+m_n(z)+s(z)}\overline\Lambda_n,\notag
\end{equation}
for $\nu=2,3$.
We represent it as follows
\begin{equation}
 Y_{\nu}=Y_{\nu1}+Y_{\nu2},\notag
\end{equation}
where
\begin{align}
 Y_{\nu1}&=-\frac1n\sum_{j=1}^n\E\frac{\varepsilon_{j\nu}\overline\Lambda_n}{(z+m_n^{(j)}(z))(z+m_n(z)+s(z))},\notag\\
 Y_{\nu2}&=\frac1n\sum_{j=1}^n\E\frac{\varepsilon_{j\nu}(R_{jj}+\frac1{z+m_n^{(j)}(z)})\overline\Lambda_n}{z+m_n(z)+s(z)}.\notag
\end{align}
By the representation \eqref{repr01}, which is similar to \eqref{repr001} we have
\begin{align}
 Y_{\nu2}=\sum_{\mu=1}^3\frac1n\sum_{j=1}^n\E\frac{\varepsilon_{j\nu}\varepsilon_{j\mu}\overline\Lambda_nR_{jj}}{(z+m_n(z)+s(z))(z+m_n^{(j)}(z))}.\notag
\end{align}
Using inequality \eqref{lar1}, we may write, for $z\in\mathbb G$
\begin{align}
 |Y_{\nu2}|\le\sum_{\mu=1}^3\frac Cn\sum_{j=1}^n\E\frac{|\varepsilon_{j\nu}||\varepsilon_{j\mu}||\overline\Lambda_n||R_{jj}|}{|z+m_n^{(j)}(z)+s(z)||z+m_n^{(j)}(z)|}. \notag
\end{align}
Applying the Cauchy -- Schwartz inequality and the inequality $ab\le \frac12(a^2+b^2)$, we get
\begin{align}
 |Y_{\nu2}|&\le\sum_{\mu=1}^3\frac Cn\sum_{j=1}^n\E^{\frac12}\frac{|\varepsilon_{j\nu}|^2|\varepsilon_{j\mu}|^2|R_{jj}|^2}{|z+m_n^{(j)}(z)+s(z)|^2|z+m_n^{(j)}(z)|^2}
 \E^{\frac12}|\Lambda_n|^2\notag\\&\le\sum_{\mu=2}^3\frac Cn\sum_{j=1}^n\E^{\frac12}\frac{|\varepsilon_{j\mu}|^4|R_{jj}|^2}{|z+m_n^{(j)}(z)+s(z)|^2|z+m_n^{(j)}(z)|^2}
 \E^{\frac12}|\Lambda_n|^2\notag\\&
 +\frac Cn\sum_{j=1}^n\E^{\frac12}\frac{|\varepsilon_{j\nu}|^2|\varepsilon_{j1}|^2|R_{jj}|^2}{|z+m_n^{(j)}(z)+s(z)|^2|z+m_n^{(j)}(z)|^2}
 \E^{\frac12}|\Lambda_n|^2
 \notag\\&\le\sum_{\mu=2}^3\frac Cn\sum_{j=1}^n\E^{\frac14}\frac{|\varepsilon_{j\mu}|^8}{|z+m_n^{(j)}(z)+s(z)|^4|z+m_n^{(j)}(z)|^4}
 \E^{\frac12}|\Lambda_n|^2\E^{\frac14}|R_{jj}|^4\notag\\&+
 \frac Cn\sum_{j=1}^n\E^{\frac14}\frac{|\varepsilon_{j\nu}|^4|\varepsilon_{j1}|^4}{|z+m_n^{(j)}(z)+s(z)|^4|z+m_n^{(j)}(z)|^4}
 \E^{\frac12}|\Lambda_n|^2\E^{\frac14}|R_{jj}|^4\notag
 .
\end{align}
Using Corollary \ref{corgot} with $\alpha=2$, we arrive at
\begin{equation}\label{fnu2}
 |Y_{\nu2}|\le \frac C{nv}\E^{\frac12}|\Lambda_n|^2.
\end{equation}

In order to estimate $Y_{\nu1}$ we introduce now the quantity

\begin{equation}
 \Lambda_n^{(j1)}=\frac1n\Tr\mathbf R^{(j)}-s(z)+\frac{s(z)}n+\frac1{n^2}\Tr{\mathbf R^{(j)}}^2s(z).\notag 
\end{equation}
 Recall that
 \begin{align}
  \eta_{j1}&=\frac1n\sum_{l\in\mathbb T_j}[(\mathbf R^{(j)})^2]_{ll},\quad \eta_{j2}=\frac1n\sum_{k\ne l\in\mathbb T_j}X_{jk}X_{jl}[(\mathbf R^{(j)})^2]_{l,k},\notag\\
  \eta_{j3}&=\frac1n\sum_{l\in\mathbb T_j}(X_{jl}^2-1)[(\mathbf R^{(j)})^2]_{ll}.
 \end{align}
Note that
\begin{equation}\label{7.57}
 |\eta_{j1}|\le \frac1n|\Tr(\mathbf R^{(j)})^2].
\end{equation}
We use that (see \cite[Lemma 7.5]{GT:2013})
\begin {equation}\label{7.58}
 \varepsilon_{j4}=\frac1n(1+\eta_{j1}+\eta_{j2}+\eta_{j3})R_{jj}.
\end {equation}
Note that
\begin{align}
 \delta_{nj}&=\Lambda_n-{\widetilde\Lambda}_n^{(j)}=-\varepsilon_{j4}-\frac{s(z)}n-\frac1n\eta_{j0}s(z)\notag\\&=
 \frac1n(R_{jj}-s(z))(1+\eta_{j1})+\frac1n(\eta_{j2}+\eta_{j3})R_{jj}.\notag
\end{align}
This yields 
\begin{equation}\label{7.59}
 |\delta_{nj}|\le \frac1n(1+|\eta_{j1}|)|R_{jj}-s(z)|+\frac1n|\eta_{j2}+\eta_{j3}||R_{jj}|
\end{equation}

We represent $Y_{\nu1}$ in the form
\begin{equation}\notag
 Y_{\nu1}=Z_{\nu1}+Z_{\nu2}+Z_{\nu3}+Z_{\nu4},
\end{equation}
where
\begin{align}
 Z_{\nu1}&=-\frac1n\sum_{j=1}^n\E\frac{\varepsilon_{j\nu}{\overline\Lambda}_n^{(j1)}}{(z+m_n^{(j)}(z))(z+m_n^{(j)}(z)+s(z))},\notag\\
 Z_{\nu2}&=\frac1n\sum_{j=1}^n\E\frac{\varepsilon_{j\nu}\overline\delta_{nj}}{(z+m_n^{(j)}(z))(z+m_n(z)+s(z))},\notag\\
 Z_{\nu3}&=\frac1n\sum_{j=1}^n\E\frac{\varepsilon_{j\nu}{\overline\Lambda}_n\varepsilon_{j4}}{(z+m_n^{(j)}(z))(z+m_n^{(j)}(z)+s(z))(z+m_n(z)+s(z))},\notag\\
 Z_{\nu4}&=-\frac1n\sum_{j=1}^n\E\frac{\varepsilon_{j\nu}{\overline\delta}_{nj}\varepsilon_{j4}}{(z+m_n^{(j)}(z))(z+m_n^{(j)}(z)+s(z))(z+m_n(z)+s(z))}.\notag
\end{align}
First,  note that by conditionally independence
\begin{equation}\label{7.60}
 Z_{\nu1}=0.
\end{equation}
Furthermore, applying H\"older's inequality, we get
\begin{align}
 |Z_{\nu3}|&\le \frac1n\sum_{j=1}^n\E^{\frac14}\frac{|\varepsilon_{j\nu}|^4}{|z+m_n^{(j)}(z)|^4|z+m_n^{(j)}(z)+s(z)|^4}\notag\\&
 \qquad\qquad\qquad\qquad\qquad\times
 \E^{\frac14}\frac{|\varepsilon_{j4}|^4}{|z+m_n(z)+s(z)|^4}\E^{\frac12}|\Lambda_n^{(j1)}|^2.\notag
\end{align}
Using  Corollary \ref{corgot} with $\alpha=4$ and Lemmas \ref{lem14} and \ref{lem00}, we obtain
\begin{align}
 |Z_{\nu3}|\le \frac{C}{(nv)^{\frac32}|z^2-4|^{\frac14}}\E^{\frac12}|\Lambda_n|^2.\notag
\end{align}
For $z\in\mathbb G$ we may rewrite this bound using Lemma \ref{lemG}
\begin{equation}
 |Z_{\nu3}|\le \frac{C}{nv}\E^{\frac12}|\Lambda_n|^2.
\end{equation}
Furthermore, note that 
\begin{equation}\notag
 |1+\eta_{j1}|\le v^{-1}\im\{z+m_n^{(j)}(z)\}\le\im\{z+m_n^{(j)}(z)+s(z)\}.
\end{equation}
This inequality together with \eqref{7.59} implies that 
\begin{equation}\label{7.62}
 |Z_{\nu4}|\le \widetilde Z_{\nu4}+\widehat Z_{\nu4},
\end{equation}
where
\begin{align}
\widetilde Z_{\nu4}&=\frac1{n^2v}\sum_{j=1}^n\E\frac{|\varepsilon_{j\nu}\varepsilon_{j4}||R_{jj}-s(z)|}{|z+m_n^{(j)}(z)||z+m_n(z)+s(z)|},\notag\\
 \widetilde Z_{\nu4}&=\frac1{n^2v}\sum_{j=1}^n\E\frac{|\varepsilon_{j\nu}\varepsilon_{j4}||\eta_{j2}+\eta_{j3}||R_{jj}-s(z)|}{|z+m_n^{(j)}(z)||z+m_n(z)+s(z)|}
\end{align}
By representation $(3.2)$, we have 
\begin{equation}\label{7.63}
 |R_{jj}-s(z)|\le |\Lambda_n||R_{jj}|+|\varepsilon_j||R_{jj}|.\notag
\end{equation}
This implies that
\begin{equation}
 \widetilde Z_{\nu4}\le \widetilde Z_{\nu41}+\widetilde Z_{\nu42},\notag
\end{equation}
where
\begin{align}
 \widetilde Z_{\nu41}&=\frac1{n^2v}\sum_{j=1}^n\E\frac{|\varepsilon_{j\nu}\varepsilon_{j4}||\Lambda_n|}{|z+m_n^{(j)}(z)||z+m_n(z)+s(z)|},\notag\\
\widetilde  Z_{\nu42}&=\frac1{n^2v}\sum_{j=1}^n\E\frac{|\varepsilon_{j\nu}\varepsilon_{j4}||\varepsilon_j|}{|z+m_n^{(j)}(z)||z+m_n(z)+s(z)|}.\notag
\end{align}

Applying H\"older inequality, we get
\begin{align}
\widetilde Z_{\nu41}&\le \frac1{nv}\E^{\frac12}|\Lambda_n|^{2}\E^{\frac18}\Big(\frac1n\sum_{j=1}^n|\varepsilon_{j\nu}|^4\Big)^2
\E^{\frac1{16}}\Big(\frac1n\sum_{j=1}^n|R_{jj}|^{16}\Big)\notag\\&\times\E^{\frac1{16}}\Big(\frac1n\sum_{j=1}^n\frac1{|z+m_n^{(j)}(z)|^{16}}\Big)
\Big(\frac1n\sum_{j=1}^n\E^{\frac14}\frac{|\varepsilon_{j4}|^4}{|z+m_n(z)+s(z)|^4}\Big).\notag
\end{align}
By Lemma 7.2.1, inequality $(7.39)$ and Corollary 5.14, we have
\begin{equation}\label{7.64}
 \widetilde Z_{\nu41}\le \frac C{n^2v^2}\E^{\frac12}|\Lambda_n|^2.
\end{equation}
Futhermore, applying H\"older inequality again, we get
\begin{align}
\widetilde Z_{\nu42}&\le \frac1{nv}\E^{\frac14}\Big(\frac1n\sum_{j=1}^n|\varepsilon_{j\nu}|^4\Big)\E^{\frac14}\Big(\frac1n\sum_{j=1}^n|\varepsilon_j|^4\Big)
\E^{\frac1{4}}\Big(\frac1n\sum_{j=1}^n|R_{jj}|^{8}\Big)\notag\\&\times\E^{\frac1{8}}\Big(\frac1n\sum_{j=1}^n\frac1{|z+m_n^{(j)}(z)|^{8}}\Big)
\E^{\frac14}\Big(\frac1n\sum_{j=1}^n\frac{|\varepsilon_{j4}|^4}{|z+m_n(z)+s(z)|^4}\Big).
\end{align}

The last inequality, Corollaries 5.14, 7.17, Lemma 7.22 together imply
\begin{equation}
 \widetilde Z_{\nu42}\le \frac C{n^2v^2}.\label{7.66}
\end{equation}
Inequalities \eqref{7.64} and \eqref{7.66} together imply
\begin{equation}\label{7.67}
 |Z_{\nu4}|\le \frac C{nv}\E^{\frac12}|\Lambda_n|^2+\frac C{n^2v^2}. 
\end{equation}

To bound $Z_{\nu2}$ we first apply  inequality \eqref{lar1} and obtain
\begin{align}
  |Z_{\nu2}|&\le\frac Cn\sum_{j=1}^n\E\frac{|\varepsilon_{j\nu}||\delta_{nj}|}
  {|z+m_n^{(j)}(z)||z+m_n^{(j)}(z)+s(z)|}.\notag
\end{align}
Furthermore, similar to bound $Z_{\nu4}$ -- inequality $(7.62)$ -- we amy write 
$$
|Z_{\nu2}|\le \widetilde Z_{\nu2}+\widehat Z_{\nu2},
$$
where 
\begin{align}
 \widetilde Z_{\nu2}&=\frac C{n^2}\sum_{j=1}^n\E\frac{|\varepsilon_{j\nu}||R_{jj}-s(z)|}{|z+m_n^{(j)}(z)|},\notag\\
 \widehat Z_{\nu2}&=\frac C{n^2}\sum_{j=1}^n\E\frac{|\varepsilon_{j\nu}||\eta_{j2}+\eta_{j3}||R_{jj}|}{z+m_n^{(j)}(z)||z+m_n^{(j)}(z)+s(z)|}.\notag
\end{align}
Applying inequality $(7.63)$ and Cauchy -- Schwartz inequality, we get
\begin{align}\label{7.68}
\widetilde Z_{\nu2}&\le \E^{\frac12}|\Lambda_n|^2\frac C{n^2}\sum_{j=1}^n\E^{\frac14}|\varepsilon_{j\nu}|^4\E^{\frac18}\frac1{|z+m_n^{(j)}(z)|^8}
\E^{\frac18}|R_{jj}|^8\notag\\&+
\frac C{n^2}\sum_{j=1}^n\E^{\frac14}|\varepsilon_{j\nu}|^4\E^{\frac18}\frac1{|z+m_n^{(j)}(z)|^8}\E^{\frac18}|R_{jj}|^8\E^{\frac14}|\varepsilon_{j}|^4.
\end{align}
Lemmas $7.15$, $7.16$, $7.22$, inequality $(7.39)$ and Corollary $5.14$ together imply
\begin{equation}
 \widetilde Z_{\nu2}\le \frac C{nv}\E^{\frac12}|\Lambda_n|^2+\frac C{n^2v^2}.
\end{equation}

Applying now H\"older's inequality, we get
\begin{align}\notag
 |\widehat Z_{\nu2}|&\le\frac C{n^2}\sum_{j=1}^n\E^{\frac14}\frac{|\varepsilon_{j\nu}|^4}{|z+m_n^{(j)}(z)+s(z)|^2}\E^{\frac14}\frac{|\eta_{j2}+\eta_{j3}|^4}{|z+m_n^{(j)}(z)+s(z)|^2}
 \notag\\&\qquad\qquad\times \E^{\frac14}|R_{jj}|^4\E^{\frac14}\frac1{|z+m_n^{(j)}(z)|^4}.\notag
\end{align}
The last inequality together with Lemmas $7.22$, $7.20$, $7.21$ and Corollaries $5.14$, $7.17$ implies
\begin{equation}\label{7.70}
 |\widehat Z_{\nu2}|\le \frac C{n^2v^2}.
\end{equation} 
Combining inequalities $(7.46)$, $(7.47)$,  $(7.54)$,  $(7.55)$,  $(7.60)$,  $(7.61)$,  $(7.67)$,  $(7.69)$, $(7.70)$,  we get
\begin{equation}\label{7.69}
 \E|\Lambda_n|^2\le \frac C{nv}\E^{\frac12}|\Lambda_n|^2+\frac C{n^2v^2}.
\end{equation}

Applying lemma 7.4 with $t=2$, $r=1$ completes the proof of Lemma $7.24$.

\end{proof}



We relabel $\eta_{j2},\,\eta_{j3}$ and introduce the following quantity
\begin{align}\beta_{j1}&=\frac1n\sum_{l\in\mathbb T_j}[(R^{(j)})^2]_{ll}-\frac1n\sum_{l=1}^n[(R)^2]_{ll},\notag\\
\beta_{j2}&=\frac1n\sum_{l\ne k\in\mathbb T_j}X_{jl}X_{jk}[(R^{(j)})^2]_{lk}=\eta_{j2},
\notag\\
\beta_{j3}&=\frac1n\sum_{l\in\mathbb T_j}(X^2_{jl}-1)[(R^{(j)})^2]_{ll}=\eta_{j3}
.\notag
\end{align}
\begin{lem}\label{bet1}Assuming the conditions of Theorem \ref{main}, we have, for $\nu=2,3$,
\begin{align}\notag
\E\{|\beta_{j\nu}|^2\Big|\mathfrak M^{(j)}\}\le\frac C{nv^3}\im m_n^{(j)}(z).
\end{align}
\end{lem}
\begin{proof}We recall that by $C$ we denote the generic constant depending on $\mu_4$ and $D_0$ only.
By definition of $\beta_{j\nu}$ for $\nu=2,3$, conditioning on $\mathfrak M^{(j)}$, we get 
\begin{align}
\E\{|\beta_{j2}|^2\big|\mathfrak M^{(j)}\}\le &\frac C{n^2}\sum_{l\ne k\in\mathbb T_j}|[(R^{(j)})^2]_{kl}|^2
\le \frac C{n^2}\sum_{l, k\in\mathbb T_j}|[(R^{(j)})^2]_{kl}|^2,\notag\\
\E\{|\beta_{j3}|^2\big|\mathfrak M^{(j)}\}\le&\frac C{n^2}\sum_{l\in\mathbb T_j}|[(R^{(j)})^2]_{ll}|^2\le 
\frac C{n^2}\sum_{l, k\in\mathbb T_j}|[(R^{(j)})^2]_{kl}|^2.\notag
\end{align}

Applying Lemma \ref{resol00}, we get the claim.
Thus Lemma \ref{bet1} is proved.
\end{proof}

\begin{lem}\label{beta} Assuming the conditions of Theorem \ref{main}, we have, for $j=1,\ldots,n$,
\begin{align}
\E|\beta_{j1}|&\le \frac C{nv^2}.\notag
\end{align}
\end{lem}
\begin{proof}
Let $\mathcal F_n^{(j)}(x)$ denote empirical spectral distribution function of matrix $\mathbf W^{(j)}$. According to {\it interlacing eigenvalues  Theorem}
(see \cite{Horn}, Theorem 4.38) we have
\begin{equation}\notag
 \sup_x|\mathcal F_n(x)-\mathcal F_n^{(j)}(x)|\le \frac Cn.
\end{equation}
Furthermore, we represent
\begin{equation}\notag
 \beta_{j1}=\int_{-\infty}^{\infty}\frac1{(x-z)^2}d(\mathcal F_n(x)-\mathcal F_n^{(j)}(x))+\frac1{n(n-1)}\Tr(\mathbf R^{(j)})^2.
\end{equation}
Integrating by parts, we get the claim.

Thus Lemma \ref{beta} is proved.
\end{proof}

\end{document}